\documentclass[10pt,reqno]{amsart}
\RequirePackage[OT1]{fontenc}
\RequirePackage{amsthm,amsmath}
\RequirePackage[numbers]{natbib}
\RequirePackage[colorlinks,linkcolor=blue,citecolor=blue,urlcolor=blue]{hyperref}
\usepackage{amssymb}
\usepackage{anysize}
\usepackage{hyperref}
\usepackage{extarrows}
\usepackage{multirow}
\usepackage{natbib}
\usepackage{stmaryrd}
\usepackage{color}
\usepackage{amsmath}
\usepackage{enumitem}
\usepackage{mathtools} 
\usepackage{dsfont} 
\usepackage{comment}
\usepackage{soul}

\numberwithin{equation}{section}
\theoremstyle{plain} 
\newtheorem{theorem}{Theorem}[section]
\newtheorem{lemma}[theorem]{Lemma}

\newtheorem{proposition}[theorem]{Proposition}

\theoremstyle{definition}
\newtheorem{definition}[theorem]{Definition}
\newtheorem{assumption}[theorem]{Assumption}

\theoremstyle{remark}
\newtheorem{remark}[theorem]{Remark}

\renewcommand{\Im}{\mathrm{Im}\,}

\newcommand{\E}{{\mathbb E }}

\newcommand{\R}{{\mathbb R }}
\newcommand{\N}{{\mathbb N}}

\renewcommand{\P}{{\mathbb P}}
\newcommand{\C}{{\mathbb C}}

\newcommand{\ii}{\mathrm{i}}
\newcommand{\ee}{\mathrm{e}}
\newcommand{\deq}{\mathrel{\mathop:}=}
\newcommand{\dd}{\mathrm{d}}
\newcommand{\ie}{\emph{i.e., }}
\newcommand{\eg}{\emph{e.g., }}
\newcommand{\cf}{\emph{c.f., }}

\newcommand{\wt}{\widetilde}
\newcommand{\ud}{\underline}

\newcommand{\bs}{\boldsymbol}

\def\Tr{\mathrm{Tr}}
\def\i{\text{i}}

\def\Dim{\Delta \widetilde{\mathrm{Im}}\,}
\def\Ai{\mathrm{Ai}}
\def\u{\mathbf{u}}

\def\one{\mathds{1}}
\def\app{ \theta_\eta}

\def\<{\langle}
\def\>{\rangle}
\def\intkappa{\int_{\kappa_1}^{\kappa_2}}
\def\X{\mathcal{X}}

\renewcommand{\mathbf}[1]{\bs{#1}}
 
\marginsize{25mm}{25mm}{25mm}{26mm}

\allowdisplaybreaks

\begin{document}

\begin{minipage}{0.85\textwidth}
 \vspace{2.5cm}
 \end{minipage}
\begin{center}
\large\bf Convergence rate to the Tracy--Widom laws for the largest eigenvalue\\ of Wigner matrices

\end{center}

\renewcommand{\thefootnote}{\fnsymbol{footnote}}	
\vspace{0.5cm}

\begin{center}
 \begin{minipage}{1.4\textwidth}

\begin{minipage}{0.33\textwidth}
	\begin{center}
		Kevin Schnelli\footnotemark[1]\\
		\footnotesize 
		{KTH Royal Institute of Technology}\\
		{\it schnelli@kth.se}
	\end{center}
\end{minipage}
\begin{minipage}{0.33\textwidth}
\begin{center}
Yuanyuan Xu\footnotemark[2]\\
\footnotesize 
		{Institute of Science and Technology Austria}\\
{\it yuanyuan.xu@ist.ac.at}
\end{center}
\end{minipage}
\end{minipage}
\end{center}

\bigskip

\footnotetext[1]{Supported by the Swedish Research Council Grant VR-2017-05195, and the Knut and Alice Wallenberg Foundation.}
\footnotetext[2]{Supported by the Swedish Research Council Grant VR-2017-05195, and the ERC Advanced Grant "RMTBeyond" No.~101020331}

\renewcommand{\thefootnote}{\fnsymbol{footnote}}	

\vspace{1cm}

\begin{center}
 \begin{minipage}{0.83\textwidth}\footnotesize{
 {\bf Abstract.}}
We show that the fluctuations of the largest eigenvalue of a real symmetric or complex Hermitian Wigner matrix of size $N$ converge to the Tracy--Widom laws at a rate $O(N^{-1/3+\omega})$, as~$N$ tends to infinity. For Wigner matrices this improves the previous rate $O(N^{-2/9+\omega})$ obtained by Bourgade~\cite{Bourgade extreme} for generalized Wigner matrices. Our result follows from a Green function comparison theorem, originally introduced by Erd\H{o}s, Yau and Yin~\cite{rigidity} to prove edge universality, on a finer spectral parameter scale with improved error estimates. The proof relies on the continuous Green function flow induced by a matrix-valued Ornstein--Uhlenbeck process. Precise estimates on leading contributions from the third and fourth order moments of the matrix entries are obtained using iterative cumulant expansions and recursive comparisons for correlation functions, along with uniform convergence estimates for correlation kernels of the Gaussian invariant ensembles.

\end{minipage}
\end{center}

 \vspace{5mm}
 
 {\small
\footnotesize{\noindent\textit{Date}: Feburary 3, 2022}\\
}
 
 \vspace{2mm}

\thispagestyle{headings}

\section{Introduction and main results}

In this paper we study a quantitative version of the edge universality for Wigner random matrices. Let~$H_N$ be a real symmetric or complex Hermitian Wigner matrix of size $N$. Then the edge universality asserts that the largest eigenvalue, $\lambda_N$, of $H_N$ satisfies
\begin{align}\label{le one}
\lim_{N \rightarrow \infty} \P \Big( N^{2/3} (\lambda_N-2)<r \Big) =\mathrm{TW}_{\beta}(r)\,,\qquad r\in\R\,,
 \end{align}
where $\mathrm{TW}_{\beta}$ are the cumulative distribution functions of the Tracy--Widom laws~\cite{TW1,TW2} and $\beta=1,2$ indicates the symmetry class ($\beta=1$ for real symmetric and $\beta=2$ for complex Hermitian Wigner matrices). The universality of the Tracy--Widom laws was first proved in~\cite{SiSo1,So1} for Wigner matrices whose entries have symmetric distributions. This symmetry assumption was partially removed in~\cite{PeSo1,PeSo2}. Edge universality for Wigner matrices whose entries have vanishing third moments was proved in~\cite{TV2}. Edge universality without moment matching was proved in~\cite{rigidity} for Wigner matrices and in~\cite{XXX,BEY edge universality} for generalized Wigner matrices. A necessary and sufficient condition on the entries' distributions for the edge universality to hold was given in~\cite{LY}.

The main result of this paper is an estimate on the rate of convergence in~\eqref{le one} for Wigner matrices. Theorem~\ref{kol_dist} below states that, for any fixed $r_0\in\R$ and small $\omega>0$,
\begin{align}\label{le two}
 \sup_{r > r_0}\Big|\P\Big( N^{2/3} (\lambda_N-2 )<r \Big)- \mathrm{TW}_{\beta}(r) \Big| \leq  N^{-1/3+\omega}\,,
\end{align}
for $N$ sufficiently large. For the Gaussian unitary ensemble (GUE, $\beta=2$) and Gaussian orthogonal ensemble (GOE, $\beta=1$) it was established in~\cite{gaussian_speed} that the convergence rate for the largest eigenvalue on a proper scaling is of order $O(N^{-2/3})$; see Theorem~\ref{convergence_gaussian} below. The first rate of convergence for non-invariant ensembles was recently given by Bourgade in~\cite{Bourgade extreme} where the upper bound $O(N^{-2/9+\omega})$ for the convergence rate was obtained for generalized Wigner matrices.

The proof of the estimate in~\eqref{le two} is based on the Green function comparison method for the edge universality by Erd\H{o}s, Yau and Yin~\cite{rigidity}. Our main technical result given in Theorem~\ref{green_comparison} compares the expectation of a suitably chosen function of the Green function of the Wigner matrix $H_N$ with the corresponding quantity for the Gaussian invariant ensembles.  Instead of the traditional Lindeberg type swapping strategy~\cite{Chatterjee,rigidity,TV2}, we use the continuous Green function flow induced by a matrix-valued Ornstein--Uhlenbeck process in combination with cumulant expansions~\cite{deformed,sparse}  for the comparison. To achieve the convergence rate $O(N^{-1/3})$ in~\eqref{le two} the comparison is required on a much finer spectral scale than the typical $O(N^{-2/3})$ edge scaling. This requires in turn precise estimates on the contributions to the Green function flow from third and fourth order moments of the matrix entries.

Contributions from third moments can be estimated using the idea of unmatched indices~\cite{rigidity}, however due to the finer spectral scale, we require expansions to arbitrary order in terms of the control parameter of the strong local law for the Green function~\cite{rigidity} to implement this idea. This step relies on applying cumulant expansions iteratively to Green functions and observing a cancellation to leading order~\cite{He+Knowles,HLY,sparse}. The usefulness of cumulant expansions in random matrix theory was recognized in~\cite{KKP} and has widely been used since, \eg~\cite{BoutetdeMonvel,ErdoesKruegerSchroeder,moment,LP}.

Contributions from fourth moments are controlled  by first showing that they can be reduced to trace-like correlation functions of products of Green functions. This first step is motivated by the Weingarten calculus~\cite{weingarten1} to compute Haar integrals of products of eigenvector components for the invariant Gaussian ensembles. The actual reduction for non-invariant ensembles relies on applying cumulant expansions iteratively. In a second step we compare the resulting trace-like correlation functions between Wigner matrices and the invariant ensembles  using again the interpolating flow. This leads to a hierarchy of correlation functions which, after expansion to arbitrary order, can be recursively estimated by the local law for the Green function. Finally, we need to control the trace-like correlation functions for the invariant ensembles. This is accomplished by using the uniform asymptotics~\cite{convergence_kernel} for correlation kernels of the invariant ensembles in the edge scaling.

Edge universality can also be studied through the dynamical approach of Erd\H{o}s, Schlein and Yau. The local relaxation time of Dyson's Brownian motion (DBM) at the edges is known~\cite{AH,Bourgade extreme,LandonYau_edge} to be of order $O(N^{-1/3})$. Combining his quantitative local relaxation estimates for the DBM with a Green function comparison for short times, Bourgade obtained in~\cite{Bourgade extreme} the convergence rate $O(N^{-2/9})$ to the Tracy--Widom laws for generalized Wigner matrices. In view of the local relaxation time of the DBM at the spectral edges, the convergence rate estimate in~\eqref{le two} may be optimal for Wigner matrices in general, though numerical simulations in~\cite{Forrester} indicate that certain Wigner matrices exhibit faster convergence rates after a scaling and centering of the largest eigenvalue. We suspect that such a centering would crucially depend on the fourth moments of the entries and the symmetry type of the matrices.

The methods presented in this paper are rather robust and can be applied to other random matrix models. Of interest in statistics are in particular convergence rate estimates for sample covariance matrices. For the white Wishart ensemble the convergence rate $O(N^{-2/3})$ after a proper scaling were obtained in~\cite{elkaroui,Ma}. Edge universality for sample covariance matrices was established in~\cite{PY1} and a first quantitative version appeared recently in~\cite{Wang}. In the accompanying article~\cite{SXsample} we establish the results corresponding to~\eqref{le two} for sample covariance matrices. In this paper we focus on estimating the contributions from third and fourth order moments of the matrix entries through assuming that the variances are uniform as for the invariant ensembles. Studying generalized Wigner matrices requires in addition new techniques to implement a variance profile and is thus postponed to our upcoming work~\cite{SXprofile}.

{\it Acknowledgment:} We thank Paul Bourgade, Maurice Duits, Peter J.\ Forrester and Rong Ma for useful comments and suggestions.

\subsection{Setup and main results}
Let $H \equiv H_N$ be an $ N \times N$ Wigner matrix satisfying the following.
\begin{assumption}\label{assump}
For a real symmetric ($\beta=1$) Wigner matrix, we assume the following.
	\begin{enumerate}
		\item[1.] The matrix entries $\{ H_{ij}\,|\,  i \leq j \}$ are independent real-valued centered random variables.
		\item[2.] For $i \neq j$, $\E[ (\sqrt{N} H_{ij})^2]=1$, and $\E [(\sqrt{N} H_{ii})^2] $ are uniformly bounded.
		\item[3.]  All moments of the entries of $\sqrt{N}H_N$ are uniformly bounded, \ie for any $k \geq 3$, there exists $C_k$ independent of $N$ such that, for all $1 \leq i,j \leq N$,
		\begin{equation}\label{moment_condition}
		\E [|\sqrt{N} H_{ij}|^k] \leq C_k\,.
		\end{equation}	
	\end{enumerate}

For a complex Hermitian ($\beta=2$) Wigner matrix, we assume the following.
	\begin{enumerate}
		\item[a.] The matrix entries $\{H_{ij}\, |\, i \leq  j  \}$ are independent complex-valued centered random variables.
		\item[b.] For $i \neq j$, $\E [|\sqrt{N} H_{ij}|^2]=1$, $\E[(H_{ij})^2]=0$, and $\E [(\sqrt{N} H_{ii})^2] $ are uniformly bounded.
		\item[c.] The bound (\ref{moment_condition}) holds true.
	\end{enumerate}
\end{assumption}
The Gaussian ensembles, which we denote by G$\beta$E for short, are Wigner matrices with Gaussian entries: For the Gaussian unitary ensemble (GUE, $\beta=2$) the off-diagonal matrix entries are standard complex-valued Gaussians (\ie $\sqrt{N}H_{ij} \stackrel{{\rm d}}{=} \mathcal{N}(0,\frac{1}{2})+\ii \mathcal{N}(0,\frac{1}{2})$) and the diagonal entries are standard real-valued Gaussians (\ie $\sqrt{N}H_{ii} \stackrel{{\rm d}}{=} \mathcal{N}(0,1)$).
Similarly, for the Gaussian orthogonal ensemble (GOE, $\beta=1$) the matrix entries are real-valued Gaussians with $\sqrt{N}H_{ij} \stackrel{{\rm d}}{=} \mathcal{N}(0,1)$ ($i\not=j$) and $\sqrt{N}H_{ii} \stackrel{{\rm d}}{=} \mathcal{N}(0,2)$. 

Let $(\lambda_j)_{j=1}^N$ be the eigenvalues of $H_N$ arranged in a non-decreasing order. It is well known that the largest eigenvalue $\lambda_N$ converges to the spectral edge $2$ in probability. The typical spacing of the top eigenvalues near~$2$ is of order $O(N^{-2/3})$, due to the square-root behavior at the end points of the limiting spectral density and eigenvalue rigidity. The limiting distribution of $N^{2/3}(\lambda_N-2)$ for the Gaussian ensembles was found by Tracy and Widom in \cite{TW1,TW2}.
The corresponding convergence rate was quantized by Johnstone and Ma \cite{gaussian_speed} in the following theorem.
\begin{theorem}[Convergence rate for the Gaussian ensembles]\label{convergence_gaussian}
Let $H_N$ be the GUE. For any fixed $r_0 \in \R$, there exists a constant $C=C(r_0)$ such that 
\begin{equation}\label{gaussian}
\sup_{r > r_0}\Big|\P^{\mathrm{GUE}}\Big( N^{2/3} (\lambda_N-2 )<r \Big)- \mathrm{TW}_{2}(r) \Big| \leq C N^{-2/3}.
\end{equation}
Moreover, considering the GOE with $N$ even, we have
\begin{equation}
\sup_{r > r_0}\Big|\P^{\mathrm{GOE}}\Big( (N-1)^{1/6}\sqrt{N} \Big(\lambda_N-(4-\frac{2}{N})^{1/2} \Big)<r \Big)- \mathrm{TW}_{1}(r) \Big| \leq C N^{-2/3}.
\end{equation}
\end{theorem}

The first quantitative convergence rate $O(N^{-2/9+\omega})$ for generalized Wigner matrices was obtained by Bourgade \cite{Bourgade extreme} using optimal local relaxation estimates for the Dyson Brownian motion and a quantitative Green function comparison theorem for short times.

The main result of this paper is an improved bound for the convergence rate of the distribution of $N^{2/3} (\lambda_N-2 )$ for arbitrary Wigner matrices to the Tracy--Widom laws.
\begin{theorem}[Convergence rate for Wigner matrices]\label{kol_dist}
Let $H_N$ be a real or complex Wigner matrix satisfying Assumption \ref{assump}. For any fixed $r_0 \in \R$ and small $\omega>0$, 

\begin{equation} 
\sup_{ r > r_0 }\Big|\P \Big( N^{2/3} (\lambda_N-2) < r \Big) - \mathrm{TW}_{\beta}(r) \Big| \leq  N^{-\frac{1}{3}+\omega},
\end{equation}
for sufficiently large $N \geq N_0(r_0,\omega)$. The corresponding statement holds for the smallest eigenvalue $\lambda_1$.
\end{theorem}

The proof of Theorem~\ref{kol_dist} relies on the Green function comparison method~\cite{EYY1, rigidity}. Let
 \begin{equation}\label{Green_fun}
G(z):=\frac{1}{ H_N-z}\,, \qquad m_N(z):=\frac{1}{N} \Tr G(z)\,, \quad\qquad z \in \C^+\,,
\end{equation}
denote the resolvent or Green function of the Wigner matrix $H_N$ and $m_N$ its normalized trace. The distribution of the rescaled largest eigenvalue can be linked to the expectation (of smooth functions) of the imaginary part of $m_N(z)$  for appropriately
chosen spectral parameters $z$; see Subsection~\ref{relating to green}. The main technical result of this paper is the following comparison theorem at the spectral edges.

\begin{theorem}[Green function comparison theorem]\label{green_comparison}
Let $F$ be a smooth function with uniformly bounded derivatives. For any small $\epsilon>0$, let $N^{-1+\epsilon} \leq \eta \leq N^{-2/3+\epsilon}$ and $|\kappa_1|, |\kappa_2| \leq C_0 N^{-2/3+\epsilon}$ for some $C_0>0$.  Then there exists some $c_0>0$ that does not depend on $\epsilon$, such that
\begin{equation}\label{gfct_eq}
\Big|\Big(\E-\E^{\mathrm{G \beta E}}\Big) \Big[ F\Big( N \int_{\kappa_1}^{\kappa_2} \Im m_N( 2+x+\ii \eta) \dd x\Big) \Big]\Big| \leq  N^{-1/3+c_0\epsilon},
\end{equation}
for sufficiently large $N \geq N_0(\epsilon,C_0)$.
\end{theorem}

\begin{remark}
A first Green function comparison theorem at the spectral edges was obtained in~\cite{rigidity} for spectral parameters $\eta$ of size $O(N^{-2/3-\epsilon})$ and with an error estimate of size $O(N^{-1/6+c_0\epsilon})$.

The constant $c_0$ in the upper bound in (\ref{gfct_eq}) can be chosen as any number bigger than one. An inspection of our proof in fact yields that the upper bound in (\ref{gfct_eq}) can be written as
$$\max\{K_4, |M_2-1|\} N^{-\frac{1}{3}+c_0\epsilon}+O(N^{-1/2+\epsilon}),$$
where $M_2=\max_{i} \big|\E[(\sqrt{N} h_{ii})^2]\big|$; $K_4=\max_{i \neq j} \big|c^{(4)}(\sqrt{N} h_{ij})\big|$, for $\beta=1$, and $K_4=\max_{i \neq j}\big| c^{(2,2)}(\sqrt{N} h_{ij})\big|$, for $\beta=2$, with $c^{(4)}(\sqrt{N} h_{ij})$ the fourth cumulant of $\sqrt{N}h_{ij}$ given in (\ref{cumulant_k}) and $c^{(2,2)}$ the corresponding $(2,2)$-cumulant defined in (\ref{cumulant_pq}).
\end{remark}

\begin{remark}
 The proof of the Green function comparison is based on a continuous interpolation given by a matrix-valued Ornstein--Uhlenbeck process; see~\eqref{sum_H}. On the level of the eigenvalues this evolution corresponds to Dyson's Brownian motion (DBM). Bourgade's proof of the convergence rate $O(N^{-2/9+\epsilon})$ consists of two parts: 1) the local relaxation estimate for the DBM for $t\gg N^{-1/3}$; 2) a quantitative version of the Green function comparison theorem for small times $t\ll 1$, which is not sharp. Optimizing the errors from these two parts, the error $N^{-2/9}$ is obtained at $t=N^{-1/9}$. In our proof, we improve the Green function comparison even for long times $t \sim \log N$  and then use standard perturbation theory to bridge to the Gaussian ensembles. 
\end{remark}

\subsection{Organization of the paper and outline of proofs}
The paper is organized as follows. In Section~\ref{sec:preliminary}, we provide the preliminaries for the proofs, \eg local law for the Green function and cumulant expansions; and recall some properties of the invariant ensembles. In Section \ref{sec: proof_main_result}, following the approach of~\cite{rigidity}, we first reduce the proof of the main result Theorem \ref{kol_dist} to the Green function comparison in Theorem~\ref{green_comparison}. We then prove Theorem~\ref{green_comparison} using the interpolating Green function flow and the key estimates on the resulting drift term stated in Proposition~\ref{prop_theta} below. 

In Section \ref{toy}, before we give the proof of Proposition~\ref{prop_theta} for arbitrary functions $F$, we prove the corresponding Green function comparison theorem in the simplest case, $F(x)=x$; see Proposition \ref{GCT_mn}. To make the statements easier, we first consider complex Hermitian Wigner matrices. The proof of Proposition \ref{GCT_mn} is carried out in Sections \ref{toy} to \ref{sec:unmatch}.  We sketch the proof in the following.

\begin{enumerate}
\item We first set up the interpolation between a given Wigner matrix and the GUE using the matrix Ornstein-Uhlenbeck process in (\ref{ou_process}). Using Ito's formula, we derive the stochastic evolution for the time-dependent normalized trace of the Green function $m_N(t,z)$ in (\ref{diff_eq_1}). It then suffices to estimate the drift term given in (\ref{claim}). Using the cumulant expansions of Lemma~\ref{cumulant}, we expand the expectation of the drift term up to the fourth order. We observe a precise cancellation of the second order terms in the cumulant expansions (\ref{step0}) for the off-diagonal entries. The cancellation of these second order terms is due to Assumption \ref{assump} (b.), namely that the variances of our Wigner matrices coincide with the invariant ensembles. It then suffices to estimate the third and fourth order terms in $(\ref{step0})$ as well as the remaining second order terms for the diagonal entries, which are averaged products of Green function entries.

\item All the third order terms, as well as the fourth order terms excluding the ones corresponding to the (2,2)-cumulants of the off-diagonal entries are unmatched; see Definition~\ref{unmatch_def}. The contributions from these unmatched terms are negligible, as stated in Proposition~\ref{unmatch_lemma} which is proved in Section~\ref{sec:unmatch}.  For GUE matrices,  corresponding estimates can be established using the Weingarten calculus as discussed in Subsection~\ref{sec:unmatch_gue}. In Subsection~\ref{sec:unmatch_toy}, we study an example of an unmatched term and introduce the expansion mechanism used to prove Proposition~\ref{unmatch_lemma} for general Wigner matrices. The key observation is that each time we perform the cumulant expansion on an unmatched term, we gain an additional off-diagonal Green function entry which slightly improves the estimate by the entrywise local law in (\ref{G}). In Subsection~\ref{sec:unmatch_proof}, we give the proof of Proposition~\ref{unmatch_lemma} for any unmatched term using the above expansion mechanism iteratively by counting the number of off-diagonal Green function entries.

\item The fourth order terms corresponding to the $(2,2)$-cumulants of the off-diagonal entries and the second order terms stemming from the diagonal entries are given in terms of matched terms with a certain structure; see Definition~\ref{def_type_AB}. Motivated by the GUE computations based on the Weingarten calculus in Subsection~\ref{subsec:gue}, we show that such terms can be expanded into trace-like correlation functions of Green functions referred to as type-0 terms in Definition \ref{def_type_AB}, as stated in Proposition~\ref{lemma_expand_type}. The proof of Proposition~\ref{lemma_expand_type} is presented in Subsection~\ref{toy_fourth} using cumulant expansions iteratively. The resulting type-0 terms are then estimated in Lemma~\ref{lemma_trace_wigner} which is proved using recursive comparisons and iterative expansions in Subsection~\ref{toy_type_zero}. The key observation is that, after deriving the stochastic evolution in~(\ref{late2}) under the Ornstein-Uhlenbeck flow for any type-0 term containing $d_1$ off-diagonal Green function entries, we can expand the corresponding drift term to arbitrary order using Propositions~\ref{unmatch_lemma} and~\ref{lemma_expand_type}, and end up with finitely many type-0 terms containing at least $d_1+1$ off-diagonal Green function entries as in~(\ref{late3}). By recursive comparison, Lemma~\ref{lemma_trace_wigner} follows from the local law in~(\ref{G}) for the Green function and the estimates of type-0 terms for the GUE in Lemma~\ref{lemma_trace}. The last Subsection~\ref{toy_type_zero_gue} is devoted to the proof of Lemma~\ref{lemma_trace} using the determinantal structure of the GUE and convergence properties of its correlation kernel in the edge scaling. 

\end{enumerate}

In Section \ref{sec:fourth}, we extend the above ideas to general functions $F$, and use the estimate (\ref{img_wigner}) from Proposition \ref{GCT_mn} as an input to prove Proposition \ref{prop_theta}. We then conclude with the Green function comparison in Theorem \ref{green_comparison} and hence our main result Theorem \ref{kol_dist}. In the last Section \ref{sec: real}, the real symmetric case is proved with the required modifications.

{\it Notation:} We will use the following definition on high-probability estimates from~\cite{Erdos+Knowles+Yau}. 
\begin{definition}\label{definition of stochastic domination}
Let $\mathcal{X}\equiv \mathcal{X}^{(N)}$ and $\mathcal{Y}\equiv \mathcal{Y}^{(N)}$ be two sequences of nonnegative random variables. We say~$\mathcal{Y}$ stochastically dominates~$\mathcal{X}$ if, for all (small) $\tau>0$ and (large)~$\Gamma>0$,
\begin{align}\label{prec}
\P\big(\mathcal{X}^{(N)}>N^{\tau} \mathcal{Y}^{(N)}\big)\le N^{-\Gamma},
\end{align}
for sufficiently large $N\ge N_0(\tau,\Gamma)$, and we write $\mathcal{X} \prec \mathcal{Y}$ or $\mathcal{X}=O_\prec(\mathcal{Y})$.
\end{definition}
We often use the notation $\prec$ also for deterministic quantities, then~\eqref{prec} holds with probability one. Properties of stochastic domination can be found in the following lemma.
\begin{lemma}[Proposition 6.5 in \cite{book}]\label{dominant}
	\begin{enumerate}
		\item $X \prec Y$ and $Y \prec Z$ imply $X \prec Z$;
		\item If $X_1 \prec Y_1$ and $X_2 \prec Y_2$, then $X_1+X_2 \prec Y_1+Y_2$ and $X_1X_2 \prec Y_1Y_2;$
		\item If $X \prec Y$, $\E Y \geq N^{-c_1}$ and $|X| \leq N^{c_2}$ almost surely with some fixed exponents $c_1$, $c_2>0$, then we have $\E X \prec \E Y$.
	\end{enumerate}
\end{lemma}

 For any vector $\mathbf v \in \C^{N}$, let $\mathbf v(j)$ be the $j$-th entry of the vector. For any matrix $A \in \C^{N \times N}$, the matrix norm induced by the Euclidean vector norm is given by $\|A\|_{2}:=\sigma_{\max}(A)$, where $\sigma_{\max}(A)$ denotes the largest singular value of $A$. We denote the sup norm of the matrix by $\|A\|_{\max}:=\max_{i,j}|A_{ij}|$. We use the notation $\ud{A}:=\frac{1}{N} \Tr A$ for the normalized trace.

Throughout the paper, we use~$c$ and~$C$ to denote strictly positive constants that are independent of~$N$. Their values may change from line to line. We use the standard Big-O and little-o notations for large~$N$. For $X,Y \in \R$, we write $X \ll Y$ if there exists a small $c>0$ such that $|X| \leq N^{-c} |Y|$ for large $N$. Moreover, we write $X \sim Y$ if there exist constants $c, C>0$ such that $c |Y| \leq |X| \leq C |Y|$ for large $N$. Finally, we denote the upper half-plane by $\C^+\deq\{z\in\C\,:\,\Im z>0\}$, and the non-negative real numbers by $\R^+\deq\{x\in\R\,:\,x \geq 0\}$.

\section{Preliminaries}\label{sec:preliminary}
In the section, we collect some basic notations, tools and results required in the subsequent sections, in particular we introduce the local law for the Green function of Wigner matrices and eigenvalue rigidity estimates; relate the distribution function of the largest eigenvalues to the normalized trace of the Green function; introduce the cumulant expansion formalism and finally recall properties of the GUE and the Airy kernel.  
 
\subsection{Local law for Wigner matrices}  
For a probability measure $\nu$ on $\R$ denote by $m_\nu$ its Stieltjes transform, \ie
\begin{align}
 m_\nu(z)\deq\int_\R\frac{\dd\nu(x)}{x-z}\,,\qquad z\in\C^+\,.
\end{align}
We refer to $z$ as spectral parameter and often write $z=E+\ii\eta$, $E\in\R$, $\eta>0$. Note that $m_{\nu}\,:\C^+\rightarrow\C^+$ is analytic and can be analytically continued to the real line outside the support of $\nu$. Moreover, $m_{\nu}$ satisfies $\lim_{\eta\nearrow\infty}\ii\eta {m_{\mu}}(\ii \eta)=-1$. The Stieltjes transform of the semicircle distribution $\rho_{sc}(x):=\frac{1}{2 \pi} \sqrt{(4-x^2)_+}$ is denoted by $m_{sc}(z)$. It is well know that $ m_{sc}(z)$ is the unique solution to
\begin{equation}\label{self_eq}
1+z m_{sc}(z)+ m_{sc}^2(z)=0\,,
\end{equation}
satisfying $\Im m_{sc}(z)>0$, for $\Im z>0$. The Stieltjes transform of the empirical eigenvalue measure of a Wigner matrix $H_N$, $\mu_N:=\frac{1}{N} \sum_{j=1}^N \delta_{\lambda_j}$, is then given by the normalized trace of its Green function defined in (\ref{Green_fun}).

Let $\kappa=\kappa(E)$ be the distance from $E \in \R$ to the closest edge point of the semicircle law, \ie
\begin{equation}\label{kappa}
\kappa:=\min \{ |E-2|, |E+2|  \}.
\end{equation}
Define the domain of the spectral parameter $z$,
\begin{equation}\label{d}
S_0:=\{z=E+\i \eta:|E| \leq 5, 0< \eta \leq 10 \}.
\end{equation}
The Stieltjes transform $m_{sc}$ has the following quantitative properties, for a reference, see \eg~\cite{book}.

\begin{lemma}\label{prop_msc} The Stieltjes transform of the semicircular law has the following properties:
	\begin{enumerate}
		\item	The imaginary part of $m_{sc}$ satisfies
		\begin{equation}\label{22}
		|{\Im}  m_{sc}(z)|  \sim \begin{cases}
		\sqrt{\kappa+\eta}, & \mbox{if } E \in[-2, 2], \\
		\frac{\eta}{\sqrt{\kappa+\eta}}, & \mbox{otherwise}\,,
		\end{cases}
		\end{equation}		
		uniformly in $z\in S_0$.
	\item There exists a strictly positive constant $c$, such that
	\begin{equation}\label{4}
	c \leq | m_{sc}(z)| \leq 1-c \eta\,,
	\end{equation}
	hold for all $z\in S_0$.
	\end{enumerate}
\end{lemma}

For any arbitrary small $\epsilon>0$, introduce the following subdomain of $S_0$,
\begin{equation}\label{ddd}
S \equiv S(\epsilon):=\big\{z=E+\ii \eta:  |E| \leq 5, N^{-1+\epsilon} \leq \eta \leq  10 \big\}.
\end{equation}

We also define the deterministic control parameter
\begin{equation}
\Psi \equiv \Psi(z):= \sqrt{ \frac{\Im  m_{sc}(z)}{N \eta}} +\frac{1}{N \eta}\,,\qquad z=E+\ii\eta.
\end{equation}
In particular, from (\ref{22}), for any $z \in S(\epsilon)$, we have
\begin{align}\label{control_eq}
\frac{C}{\sqrt{N}} \leq \Psi(z) \leq  C'N^{-\epsilon}\,.
\end{align}

With these notations, we are now ready to state the following local law for the Green function of a Wigner matrix.
\begin{theorem}[Local law for Wigner matrices \cite{rigidity}]\label{le theorem local law}
 Let $H$ be a symmetric or Hermitian $N$ by $N$ matrix satisfying Assumption~\ref{assump} and recall the Green function of $H$ and its normalized trace in (\ref{Green_fun}). Then we have
	\begin{equation}\label{g_estimate}
	\max_{1 \leq i,j \leq N} | G_{ij}(z) -\delta_{ij}  m_{sc} (z) |   \prec \Psi(z), \qquad | m_N(z) - m_{sc}(z) | \prec \frac{1}{N \eta}\,,
	\end{equation}
uniformly in $z\in S$.
\end{theorem}

\subsection{Rigidity of eigenvalues}

 The local law for the Green function in Theorem~\ref{le theorem local law} implies the following rigidity estimates for the  eigenvalues of $H$. Recall that the eigenvalues of $H$ are denoted as $(\lambda_j)_{j=1}^N$ arranged in a non-decreasing order. For $E_1<E_2$ ($E_1,E_2\in\R\cup\{\pm\infty\}$) denote the eigenvalue counting function by
\begin{equation}
\mathcal{N}(E_1,E_2):=\# \{j: E_1 \leq \lambda_j \leq E_2\}\,. 
\end{equation}  
We also define the classical location $\gamma_j$ of the $j$-th eigenvalue $\lambda_j$ by
\begin{equation}\label{classical}
\frac{j}{N}=\int_{-\infty}^{\gamma_j} \rho_{sc}(x) \dd x.
\end{equation}

\begin{theorem}[Eigenvalue rigidity \cite{rigidity}]
For any $E_1<E_2$, we have
\begin{equation}\label{rigidity3}
 \Big| \mathcal{N}(E_1,E_2)-N \int_{E_1}^{E_2} \rho_{sc}(x) \dd x \Big| \prec 1\,.
\end{equation}
In addition, for any $1\leq j \leq N$, we have
\begin{equation}\label{rigidity4}
 |\lambda_j-\gamma_j| \prec N^{-2/3} \Big( \min \{ j, N-j+1\} \Big)^{-1/3}\,.
\end{equation}
In particular, fix any $C_1$ and $C_2$, then for any small $\epsilon>0$ and large $\Gamma>0$ we have
\begin{equation}\label{rigidity1}
|\lambda_N-2 | \leq N^{-2/3+\epsilon}, \qquad \mathcal{N}(2-C_1N^{-2/3+\epsilon}, 2+C
_2 N^{-2/3+\epsilon}) \leq N^{2\epsilon},
\end{equation}
with probability bigger than $1-N^{\Gamma}$, for $N$ sufficiently large.
\end{theorem}

\subsection{Relating the distribution of the largest eigenvalue to the Green function}\label{relating to green}
Fix a small $\epsilon>0$ and set 
\begin{equation}\label{E_L}
E_L:=2+4N^{-2/3+\epsilon}.
\end{equation}
For any $E \leq E_L$, we define
\begin{equation}\label{le indicator chi}
\chi_{E}:=\one_{[E,E_L]}\,,
\end{equation}
and note that $\mathcal{N}(E,E_L)=\Tr \chi_E(H)$. For $\eta>0$, we define the mollifier $\theta_{\eta}$ by setting
\begin{equation}\label{le mollifier}
\theta_{\eta}(x):=\frac{\eta}{\pi(x^2+\eta^2)}=\frac{1}{\pi} \Im \frac{1}{x-\i \eta}.
\end{equation}
We can relate $\Tr \chi_{E} \star \theta_{\eta}(H)$ to the normalized trace of the Green function by the following identity,
\begin{equation}\label{approx}
\Tr \chi_{E} \star \theta_{\eta}(H)=\frac{N}{\pi} \int \chi_E(y) \Im m_N(y+\i \eta) \dd y=\frac{N}{\pi} \int_{E}^{E_L} \Im m_N(y+\i \eta) \dd y\,.
\end{equation}
The following lemma assures that $\Tr \chi_E(H)$ can be sufficiently well approximated by $\Tr \chi_E\star \theta_\eta(H)$ for $\eta \ll N^{-2/3}$. Relying on this approximation, the lemma after, Lemma~\ref{lemma1}, then yields the desired link between the distribution function of the rescaled largest eigenvalue of $H$ and the normalized trace of the  Green function using a cleverly chosen observable. This line of arguments was used first in \cite{rigidity} to prove the edge universality of Wigner matrices, where $\eta$ is chosen slightly smaller than the typical edge eigenvalue spacing $N^{-2/3}$. In order to obtain a quantitative convergence rate, we aim to choose here~$\eta$ much smaller with $\eta \gg N^{-1}$. A similar argument was used in~\cite{Bourgade extreme}. The proofs of Lemma~\ref{lemma2} and Lemma~\ref{lemma1} are modifications of~\cite{rigidity} in order to accommodate the small $\eta$ regime, and are postponed to Appendix.

\begin{lemma}\label{lemma2}
Let $E$, $\eta$ and $l_1$ be scale parameters satisfying $N^{-1}\ll\eta \ll l_1 \ll E_L-E  \leq C N^{-2/3+\epsilon}$. Then, for any $\Gamma>0$,
\begin{equation}\label{approx1}
\Big|\Tr \chi_E(H)- \Tr \chi_{E} \star \theta_{\eta}(H)\Big| \leq C\Big( \mathcal{N}(E-l_1,E+l_1)+\frac{\eta}{l_1} N^{2 \epsilon}\Big),
\end{equation}
holds with probability bigger than $1-N^{-\Gamma}$, for $N$ sufficiently large.
\end{lemma}
 Let $F\,:\,\R\longrightarrow\R$ be a smooth cut-off function such that
\begin{equation}\label{q_function}
F(x)=1, \quad \mbox{if} \quad |x| \leq 1/9; \qquad F(x)=0, \quad \mbox{if} \quad |x| \geq 2/9,
\end{equation}
and we assume that $F(x)$ is non-increasing for $x \geq 0$. Then one obtains from Lemma \ref{lemma2} the following result.
\begin{lemma}\label{lemma1}
Set  $l_1=N^{3\epsilon} \eta$ and $l=N^{3\epsilon}l_1$ such that $N^{-1} \ll \eta \ll l_1 \ll l \ll E_L-E  \leq C N^{-2/3+\epsilon}$. Then for any $\Gamma>0$, we have
\begin{equation}
 \Tr \chi_{E+l} \star \theta_{\eta}(H) -N^{-\epsilon} \leq \mathcal{N}(E, \infty) \leq  \Tr \chi_{E-l} \star \theta_{\eta}(H) +N^{-\epsilon},
\end{equation}
with probability bigger than $1-N^{-\Gamma}$, for $N$ sufficiently large. Furthermore, we have
\begin{equation}\label{approx2}
\E \Big[ F\Big( \Tr \chi_{E-l} \star \theta_{\eta}(H) \Big)\Big]-N^{-\Gamma} \leq \P \Big( \mathcal{N}(E, \infty) =0\Big) \leq   \E \Big[F\Big( \Tr \chi_{E+l} \star \theta_{\eta}(H) \Big)\Big]+N^{-\Gamma},
\end{equation}
where $F(x)$ is the cut-off function given in (\ref{q_function}).
\end{lemma}
Hence, recalling~\eqref{approx}, we have established the desired link to the normalized trace of the Green function.

\subsection{Cumulant expansion formulas}

A key tool of this paper are the following cumulant expansion identities. For reference, we refer to Lemma 3.1 in \cite{moment}.
\begin{lemma}\label{cumulant}
	Let $h$ be a complex-valued random variable with finite moments. Define the $(p,q)$-cumulant of $h$ to be 
	\begin{align}\label{cumulant_pq}
	c^{(p,q)}:=(-\ii)^{p+q} \Big( \frac{\partial^{p+q}}{\partial s^p \partial t^q} \log \E e^{\i s h+\i t \overline{h}} \Big) \Big|_{s,t=0}.
	\end{align}
	 Let $f: \C \times \C \longrightarrow \C$ be a smooth function and denote its derivatives by
	$$f^{(p,q)}(z_1,z_2):=\frac{\partial^{p+q}}{\partial z^p_1 \partial z^q_2} f(z_1,z_2).$$ 
	Then for any fixed $l \in \N$, we have
	\begin{align}\label{le first cumulant formula}
	\E\big[ \bar{h} f(h, \bar{h})\big]=\sum_{p+q+1=1}^l \frac{1}{p!q!} c^{(p,q+1)}\E\big[ f^{(p,q)}(h,\bar h)\big] +R_{l+1}\,,
	\end{align}
	where the error term $R_{l+1}$ can be bounded as
	\begin{align}\label{cumulant_error}
	|R_{l+1}| \leq C_l \E |h|^{l+1} \max_{p+q=l} \Big\{ \sup_{|z| \leq M} |f^{(p,q)}(z,\bar{z})| \Big\}+C_l \E \Big[ |h|^{l+1} 1_{|h|>M}\Big] \max_{p+q=l} \|f^{(p,q)}(z,\bar{z})\|_{\infty},
	\end{align}
	and $M>0$ is an arbitrary fixed cutoff.
	
	Moreover, we have the analogous cumulant expansion formula for a real-valued random variable $h$ with finite moments. Define the $k$-th cumulant of $h$ to be
	\begin{align}\label{cumulant_k}
	c^{(k)}:=(-\ii)^{k} \Big(\frac{\dd^k}{\dd t^k} \log \E e^{\ii t h} \Big)\Big|_{t=0}.
	\end{align}
	Let $f: \R \longrightarrow \C$ be a smooth function and denote by $f^{(k)}$ its $k$-th derivative. Then for any fixed $l \in \N$, we have
	\begin{align}
	\E \big[h f(h)\big]=\sum_{k+1=1}^l \frac{1}{k!} c^{(k+1)}\E[ f^{(k)}(h) ]+R_{l+1}\,,
	\end{align}
	where the error term satisfies
	$$|R_{l+1}| \leq C_l \E |h|^{l+1} \sup_{|x| \leq M} |f^{(l)}(x)| +C_l \E \Big[ |h|^{l+1} 1_{|h|>M}\Big] \|f^{(l)}\|_{\infty},$$
	and $M>0$ is an arbitrary fixed cutoff.
\end{lemma}

\subsection{GUE and the Airy kernel}
Let $H\equiv H_N$ belong to the GUE and denote the eigenvalues of the rescaled matrix $\sqrt{N} H$ by $(\mu_j)_{j=1}^N$ in non-decreasing order. The joint eigenvalue density is explicitly given~by
$${p}(\mu_1, \cdots, \mu_N)=\frac{1}{Z_{N,\beta}} \prod_{i < j} |\mu_i-\mu_j|^\beta \ee^{-\frac{\beta}{4} \sum_{i=1}^N \mu_i^2}, \qquad \beta=2,$$
with $Z_{N,\beta}$ be the normalization constant. 

The process of the eigenvalues is well known to be a determinantal point process \cite{kurt_det, sasha_det}. The n-point correlation function of the eigenvalue process is given by
\begin{align}\label{determinant}
p_n(\mu_1, \cdots, \mu_n)=\det[ K_N(\mu_i,\mu_j)]_{1 \leq i,j \leq n},
\end{align}
with the reproducing kernel given by
$$K_N(x,y):=\sum_{k=0}^{N-1} q_k(x) q_k(y) \ee^{-\frac{x^2+y^2}{4}},$$
where $q_{k}$ is the $k$-th Hermite polynomial given by
$$q_k(x):=(-1)^k \mathrm{e}^{\frac{x^2}{2}} \frac{\dd^k}{\dd x^k} \mathrm{e}^{-\frac{x^2}{2}}.$$
The Hermite polynomials are orthogonal with respect to the weight $\mathrm{e}^{-\frac{x^2}{2}}$ over $\R$. We further define the k-th Hermite function by
\begin{align}\label{le hermite functions}
\phi_k(x):=\frac{1}{\sqrt{\sqrt{2 \pi}k! }} \mathrm{e}^{-\frac{x^2}{4}} q_k(x)\,,
\end{align}
which is a solution to the differential equation
\begin{align}\label{hermite_fn}
\phi_k''(x)+(k+\frac{1}{2}-\frac{x^2}{4}) \phi_k(x)=0.
\end{align}
One then checks that $\{\phi_k\}$ form an orthonormal basis of $L^2(\R)$. The Christoffel-Darboux formula then states that
\begin{align}\label{kernel_xy}
K_N(x,y)=\sum_{k=0}^{N-1} \phi_k(x) \phi_k(y) =\sqrt{N} \frac{\phi_{N}(x) \phi_{N-1}(y)-\phi_{N-1}(x)\phi_{N}(y)}{x-y},\quad x \neq y,
\end{align}
as well as
\begin{align}\label{kernel_xx}
K_N(x,x) =\sqrt{N} \Big( \phi'_{N}(x) \phi_{N-1}(x)-\phi'_{N-1}(x)\phi_{N}(x) \Big).
\end{align}
We also have the trace identity for the kernel
\begin{align}\label{eq_1}
\int_{\R} K_N(x,x) \dd x=N,
\end{align}
and the reproducing formula
\begin{align}\label{eq_2}
K_N(x,y)=\int_{\R} K_N(x,z) K_N(z,y) \dd z.
\end{align}
More details can be found in \cite{AGZ, deift}.

Recall that the eigenvalues $(\lambda_j)_{j=1}^N$ of the GUE are given by $\lambda_j=\frac{\mu_{j}}{\sqrt{N}}$. Then the corresponding kernel for the eigenvalue process $(\lambda_j)$ is given by
\begin{align}\label{wt_kernel}
\widetilde K_N(x,y)=\sqrt{N} K_N(\sqrt{N}x, \sqrt{N}y).
\end{align}

In the edge regime, we rescale the eigenvalues as $\lambda_j=2+\frac{l_j}{N^{2/3}}$ and the corresponding kernel is then given by
\begin{align}\label{edge}
K^{\mathrm{edge}}_{N}(x,y):=\frac{1}{N^{2/3}} \wt K_N\Big( 2+\frac{x}{N^{2/3}},2+\frac{y}{N^{2/3}} \Big)=\frac{1}{N^{1/6}} K_N\Big( 2\sqrt{N}+\frac{x}{N^{1/6}},2\sqrt{N}+\frac{y}{N^{1/6}} \Big).
\end{align}

Next, recall that the Airy kernel is defined by
\begin{align}\label{airy}
K_{\mathrm{airy}}(x,y):= \frac{\Ai(x)\Ai'(y)-\Ai'(x)\Ai(y)}{x-y}\,,
\end{align}
with $\Ai$ be the Airy function of first kind, which is the solution of 
\begin{align}\label{airy_fun_def}
\Ai''(x)-x \Ai(x)=0\,, \qquad\quad x \in \R\,,
\end{align}
satisfying the boundary condition $\Ai(x) \rightarrow 0$ as $x\rightarrow \infty$. 
 As $x \rightarrow y$, the Airy kernel reduces to
\begin{align}\label{airy_xx}
K_{\mathrm{airy}}(x,x):=(\Ai'(x))^2-\Ai''(x)\Ai(x)=(\Ai'(x))^2-x(\Ai(x))^2.
\end{align}

\begin{lemma}[Lemma 3.9.33 in \cite{AGZ}]\label{lemma_airy_bound}
For fixed $L_0\in\R$, there exists a constant $C$, such that one has uniformly in $x,y \in [ L_0, +\infty)$ that
\begin{align}
\Big| \partial_x^{a}  \partial_y^{b} K_{\mathrm{airy}}(x,y)\Big| \leq C\,,\qquad a,b\in\{0,1\}\,.
\end{align}
Furthermore, we have the asymptotics
\begin{align}\label{airy_asym}
K_{\mathrm{airy}}(x,x) \sim_{x \rightarrow \infty} \frac{\mathrm{e}^{-\frac{4}{3}x^{\frac{3}{2}}}}{x}; \qquad K_{\mathrm{airy}}(x,x) \sim_{x \rightarrow -\infty} \sqrt{|x|}.
\end{align}
\end{lemma}

The following result of Deift and Gioev \cite{convergence_kernel} quantizes the convergence rate of the edge kernel in (\ref{edge}) to the limiting Airy kernel in (\ref{airy}).
\begin{theorem}[Theorem 1.1 in \cite{convergence_kernel}]\label{kernel_diff}
For fixed $L_0\in \R$, there exists constants $C,c>0$ depending on $L_0$, such that one has uniformly for $x,y \in [ L_0, +\infty)$,
\begin{align}\label{difference}
\Big| \partial_x^{a} \partial_y^{b} \Big[ K^{\mathrm{edge}}_{N}(x,y)-K_{\mathrm{airy}}(x,y)\Big]\Big| \leq CN^{-2/3} \ee^{-cx}\ee^{-cy}\,,\qquad a,b\in\{0,1\}\,.
\end{align}
\end{theorem}

\section{Proof of Theorem \ref{kol_dist}}\label{sec: proof_main_result}
In this section we give the proof of Theorem~\ref{kol_dist} from the main technical result, the Green function comparison theorem, Theorem~\ref{green_comparison}.
\begin{proof}[Proof of Theorem \ref{kol_dist}]
Because of the rigidity of the eigenvalues in (\ref{rigidity1}),  one easily verifies that, for any $\epsilon>0$ and $\Gamma>2/3$,
\begin{equation}\label{upper}
\sup_{|r| \geq N^{\epsilon}} \Big|\P\Big( N^{2/3} (\lambda_N- 2 )<r \Big)-\P^{\mathrm{G}\beta \mathrm{E}}\Big(N^{2/3} (\lambda_N-2 )<r\Big) \Big| \leq  N^{-\Gamma},
\end{equation}
for sufficiently large $N$. Hence in order to prove Theorem \ref{kol_dist}, it suffices to focus on $r_0 < r < N^{\epsilon}$ with $r_0$ as in Theorem~\ref{convergence_gaussian} and Theorem~\ref{kol_dist}. 

Set as in (\ref{E_L})
$$E:=2+N^{-2/3} r, \qquad \mbox{and } \quad E_{L}:=2+4N^{-2/3+\epsilon}.$$ 
Fix $\eta=N^{-1+\epsilon}$ and $l=N^{-1+ 7\epsilon}$ as in Lemma \ref{lemma1}. Here we choose $\epsilon>0$ sufficiently small such that $l \ll N^{-2/3}$. From (\ref{approx}) and (\ref{approx2}), we can relate the distribution of the largest eigenvalue to the normalized trace of the Green function as follows,
\begin{align}\label{im_identity}
\E \Big[ F\Big( N \int_{N^{-2/3}r-l}^{4N^{-2/3+\epsilon}}  \Im m_N(2+x+\ii \eta) \dd x\Big) \Big]&-N^{-\Gamma} \leq \P\Big( N^{2/3} (\lambda_N- 2 )<r \Big)=\P \Big(  \mathcal{N}(E, \infty)=0 \Big) \nonumber\\
\leq &   \E \Big[ F\Big( N \int_{N^{-2/3}r+ l}^{4N^{-2/3+\epsilon}} \Im m_N(2+x+\ii \eta) \dd x\Big)\Big]+N^{-\Gamma}.
\end{align}
By shifting the value of $r$ in the second inequality of (\ref{im_identity}) and combining with the first inequality of (\ref{im_identity}), we obtain
\begin{align}
 \P\Big( N^{2/3} (\lambda_N- 2 )<r -2N^{2/3} l\Big) -N^{-\Gamma} \leq &\E \Big[ F\Big( N \int_{N^{-2/3}r-l}^{4N^{-2/3+\epsilon}} \Im m_N(2+x+\ii \eta) \dd x\Big) \Big]\nonumber\\
  \leq &\P\Big( N^{2/3} (\lambda_N- 2 )<r \Big)+N^{-\Gamma}.
\end{align}
Note that the above inequalities hold true for $\beta=1,2$ and any Wigner matrices, including the Gaussian ensembles. From the known convergence rates for the Gaussian ensembles in Theorem \ref{convergence_gaussian} (for the GUE, and GOE with $N$ even), and the convergence rate $N^{-1/3}$ obtained in Theorem~1.2 of~\cite{choup} for the GOE with $N$ odd, we find
\begin{align*}
\mathrm{TW}_\beta\Big(r -2N^{2/3} l\Big) -CN^{-1/3} \leq \E^{\mathrm{G}\beta \mathrm{E}} \Big[ F\Big( N \int_{N^{-2/3}r-l}^{4N^{-2/3+\epsilon}} \Im m_N(2+x+\ii \eta) \dd x\Big) \Big] \leq \mathrm{TW}_\beta(r)+CN^{-1/3}.
\end{align*}
A similar upper and lower bound can be obtained in the same way when we consider $+l$ in the integral domain instead of $-l$. Since the Tracy--Widom distributions have smooth and uniformly bounded densities and $l=N^{-1+7\epsilon}$, we have
\begin{align}\label{rhs_gue}
\sup_{r_0 <r<N^{\epsilon}}\Big| \E^{\mathrm{G}\beta \mathrm{E}} \Big[ F\Big( N \int_{N^{-2/3}r \pm l}^{4N^{-2/3+\epsilon}} \Im m_N(2+x+\ii \eta) \dd x\Big) \Big] -\mathrm{TW}_\beta(r) \Big| =O(N^{-1/3+7\epsilon}).
\end{align}
Using the Green function comparison theorem, Theorem \ref{green_comparison}, there exists some $c_0>0$ independent of $\epsilon$ such that
\begin{align}\label{rhs_wigner}
\sup_{r_0 <r<N^{\epsilon}}\Big| \Big(\E- \E^{\mathrm{G} \beta \mathrm{E}} \Big)\Big[ F\Big( N \int_{N^{-2/3}r \pm l}^{4N^{-2/3+\epsilon}} \Im m_N(2+x+\ii \eta) \dd x\Big) \Big] \Big| \leq N^{-1/3+c_0\epsilon},
\end{align}
for sufficiently large $N$. In combination with (\ref{im_identity}) and (\ref{rhs_gue}), we choose $\epsilon<\frac{\omega}{\max\{c_0,7\}}$ in the setting of Theorem \ref{kol_dist} and obtain
\begin{align}
\sup_{r_0 <r < N^{\epsilon}}\Big|\P\Big( N^{2/3} (\lambda_N- 2 )<r \Big)-\mathrm{TW}_\beta(r) \Big| \leq N^{-1/3+\omega}.
\end{align}
Together with (\ref{upper}), we have hence completed the proof of Theorem \ref{kol_dist}.
\end{proof}

We now move on to the proof of the Green function comparison theorem, Theorem~\ref{green_comparison}. In the following, we first consider complex Hermitian Wigner matrices, as the complex Hermitian case is slightly easier than the real symmetric case. The proof of the Green function comparison theorem in the real symmetric setup is presented in Section~\ref{sec: real}.
\begin{proof}[Proof of Theorem \ref{green_comparison}]
Consider the matrix Ornstein-Uhlenbeck process $\big(h_{ab}(t)\big)_{a,b=1}^N$:
\begin{equation}\label{ou_process}
\dd h_{ab}(t)=\frac{1}{\sqrt{N}} \dd \beta_{ab}(t)- \frac{1}{2} h_{ab}(t) \dd t, \qquad h_{ab}(0)=(H_N)_{ab},
\end{equation}
 where $\big(\beta_{ab}(t)\big)_{a<b}^N$ are independent complex standard Brownian motions, $\big(\beta_{aa}(t)\big)_{a=1}^N$ are independent real standard Brownian motions, $\big(\beta_{ab}(t)\big)_{a<b}$ are independent from $\big(\beta_{aa}(t)\big)_{a=1}^N$, and $\beta_{ba}(t)=\overline{\beta_{ab}(t)}$. The initial condition $H_N$ is a complex Hermitian Wigner matrix satisfying Assumption \ref{assump}. In distribution the above is equivalent to writing
 \begin{equation}\label{sum_H}
H(t) \stackrel{{\rm d}}{=} \mathrm{e}^{-\frac{t}{2}}H_N +\sqrt{1-\mathrm{e}^{-t}} \mathrm{GUE}_N\,,\qquad  t \in \R^+,
 \end{equation}
 where $\mathrm{GUE}_N$ belongs to the GUE. For any $t \in \R^+$, $z \in \C \setminus \R$, we define
\begin{equation}\label{le time dependent G}
G(t,z):=\frac{1}{ H(t)-z}; \qquad  m_N(t,z):=\frac{1}{N} \Tr G(t,z).
\end{equation}
Recalling the local law Theorem \ref{le theorem local law} and Lemma \ref{prop_msc}, we obtain that the local law for $G(t,z)$,
	\begin{equation}\label{G}
	\max_{i,j} | G_{ij}(t,z) -\delta_{ij}  m_{sc} (z) |   \prec \Psi(z);\qquad | m_N(t, z) - m_{sc}(z) | \prec \frac{1}{N \eta}\,,
	\end{equation}
holds uniformly in $z \in S$ given in (\ref{ddd}) and $t \geq 0$. Indeed, we choose a mesh of the interval $0 \leq t \leq T:=8 \log N$ of size $N^{10}$, and obtain that (\ref{G}) holds uniformly in $z \in S$, $t \in [0,8\log N]$ from the continuity of the process (\ref{sum_H}) in time. Moreover,  (\ref{G}) also holds uniformly in $t \geq 8 \log N$ from (\ref{approxxxx}) below.

In the following, we often ignore the parameters and write for short
$$H\equiv H(t), \qquad  h_{ab} \equiv h_{ab}(t), \qquad G \equiv G(t,z), \qquad  m_N \equiv m_N(t,z), \qquad t \in \R^+,~z \in \C \setminus \R.$$
For a fixed small $\epsilon>0$ and some $C_0>0$, let 
\begin{align}\label{z_set}
N^{-1+\epsilon} \leq \eta \leq N^{-2/3+\epsilon}, \qquad |\kappa_1|, |\kappa_2|\leq C_0N^{-2/3+\epsilon},
\end{align}
with $\kappa_1<\kappa_2$. In view of~\eqref{approx} and~\eqref{approx2}, we are interested in the quantity
\begin{equation}\label{X}
\mathcal{X} \equiv \mathcal{X}(t):=N\int_{\kappa_1}^{\kappa_2} \Im m_N(t,2+x+\ii \eta) \dd x, \qquad t \in \R^+.
\end{equation}
Hence $\mathcal{X}$ is a function of $t$, $\eta$ as well as $\kappa_1$ and $\kappa_2$.

Let $F: \R \rightarrow \R$ be an arbitrary smooth function with uniformly bounded derivatives. The next lemma determines the evolution of the observable $F\big(\mathcal{X}(t)\big)$ under the Ornstein--Uhlenbeck flow in~\eqref{ou_process}. To alleviate the notation, we introduce the following abbreviations. Let $P\,:\,\R^+ \times \C \setminus \R \longrightarrow\C$ be an arbitrary function, then we introduce
\begin{equation}\label{tilde}
\widetilde{{\Im}}P \equiv \widetilde{{\Im}} P(t,z):=\frac{1}{2 \ii} (P(t,z)-P(t,\bar z))\,,\qquad t \in \R^+,~z\in\C\setminus \R\,.
\end{equation}
For example, for complex Wigner matrices, $\widetilde \Im G_{ij}(t,z) \neq  \Im G_{ij}(t,z)$, unless $i=j$. Further, we abbreviate, for $t \in \R^+$, and $z_1,z_2\in\C\setminus \R$,
\begin{align}\label{dim}
\Delta \widetilde \Im P \equiv (\Delta \widetilde \Im P)(t,z_1,z_2):=\frac{1}{2 \ii} \Big(P(t,z_2)-P(t,\overline{z_2})\Big) - \frac{1}{2 \ii}\Big(P(t,z_1)-P(t,\overline{z_1})\Big)\,,
\end{align}
where the spectral parameters are given as
\begin{align}\label{le zs} z_1=2+\kappa_1+\ii \eta\,,\qquad\qquad z_2=2+\kappa_2+\ii \eta\,,
\end{align}
with $\kappa_1,\,\kappa_2,$ and $\eta$ from~\eqref{z_set}. In particular, we have $z_1,z_2 \in S_{\mathrm{edge}}$ defined in (\ref{S_edge}) below.

Returning to $F(\mathcal{X})$, Ito's lemma yields the following result.
\begin{lemma}\label{lemma_first}
The observable $F(\X)$ satisfies the following stochastic differential equation:
\begin{equation}\label{dF}
\dd F(\mathcal{X})= \dd  M+\Theta \dd t ,
\end{equation}
with the diffusion term
\begin{equation}\label{mart}
\dd M=-\frac{1}{\sqrt{N}}  \sum_{a, b=1}^N\Big( F'(\X) \Dim G_{ba} \Big) \dd \beta_{ab},
\end{equation}
and the drift term $\Theta \equiv \Theta(t,z_1,z_2)$ is explicitly given in~\eqref{theta1} below. Moreover, $\E[\Theta]$ can be written as
\begin{align}\label{derivative_gene}
\E [\Theta] =& 
\sum_{\substack{ p+q+1=3\\ p,q \in \N}}^{4}K_{p,q+1}+E_2 +O_{\prec}(N^{-1/2}),
\end{align}
for $N$ sufficiently large, with
\begin{align}
K_{p,q+1}:=& \frac{1}{2 p! q! N^{\frac{p+q+1}{2}}}  \sum_{\substack{ a,b=1\\a \neq b}}^N s^{(p,q+1)}_{ab} \E \Big[\frac{ \partial^{p+q} F'(\X) \Dim G_{ba} }{\partial h^p_{ba} \partial h^q_{ab} }\Big];\label{K_term}\\
E_2:=&\frac{1}{2 N} \sum_{a=1}^N ( s^{(2)}_{aa} -1)\E \Big[\frac{ \partial F'(\X)  \Dim G_{aa}}{ \partial h_{aa} }\Big], \label{E_1}
\end{align}
where $s^{(p,q+1)}_{ab} \equiv s^{(p,q+1)}_{ab}(t)$ and $s^{(2)}_{aa} \equiv s^{(2)}_{aa}(t)$ are the cumulants of the rescaled time dependent entries $\sqrt{N} h_{ab}$ defined in (\ref{cumulant_pq}) and (\ref{cumulant_k}).
\end{lemma}

\begin{remark}\label{diffusion_mart}
The diffusion term $\dd M$ in (\ref{mart}) yields a martingale $M(t)$ upon integration. Note that the operator norm of the Green function has the deterministic bound $\|G(z)\|_2 \leq \frac{1}{\eta} \leq N^{1-\epsilon}$, given $z=E+\ii \eta$ with $\eta \geq N^{-1+\epsilon}$. Since $F$ has bounded derivatives, $|F'(\X) \Dim G_{ba}| =O(N^{1-\epsilon})$. Thus $M(t)$ is a true martingale with vanishing expectation. 
\end{remark}

\begin{remark}
 In~\eqref{derivative_gene}, only cumulants of order three and higher appear, i.e.\ $p+q+1\ge 3$. This is a consequence of our assumption that the second moments of the off-diagonal matrix entries match with the Gaussian ensembles; see item b.) in Assumption~\ref{assump}.
\end{remark}

\begin{proof}[Proof of Lemma \ref{lemma_first}]

Recall the dynamics of the Orstein-Uhlenbeck process in (\ref{ou_process}) and that $G$ is a function of the matrix entries $h_{ab}$. Using the first Ito's lemma and then the relation
\begin{equation}\label{dH}
\frac{\partial G_{ij}}{ \partial h_{ab}}=-G_{ia} G_{bj},
\end{equation}
 we compute 
\begin{align}\label{derivative2}
\dd G_{ij}(t,z)=& \frac{\partial G_{ij}}{\partial t} \dd t  + \sum_{a}\frac{\partial G_{ij}}{\partial h_{aa}} \dd h_{aa} +\frac{1}{2}\sum_{a} \frac{\partial^2 G_{ij}}{\partial h_{aa} \partial  h_{aa}} \dd h_{aa} \dd h_{aa}  \nonumber\\
&+ \sum_{a<b}\frac{\partial G_{ij}}{\partial h_{ab}} \dd h_{ab}  + \sum_{a<b}\frac{\partial G_{ij}}{\partial \overline{h_{ab}}} \dd \overline{h_{ab}} + \sum_{a< b} \frac{\partial^2 G_{ij}}{\partial h_{ab} \partial  \overline{h_{ab}}} \dd h_{ab} \dd \overline{h_{ab}} \nonumber\\
=&  - \frac{1}{\sqrt{N}} \sum_{a, b=1}^N G_{ia}G_{bj}  \dd \beta_{ab}+\frac{1}{2} \sum_{a, b=1}^N \Big( h_{ab} G_{ia} G_{bj}+\frac{1}{N}  G_{ib} G_{bj} G_{aa}+\frac{1}{N}  G_{ia} G_{aj} G_{bb} \Big) \dd t.
\end{align}

In view of $\X$ from (\ref{X}), we take the normalized trace of the Green function and the imaginary part. Using the symmetry of $H$ and 
$${G_{ij}(z)}=\overline{G_{ji}(\bar{z})},$$
we obtain the following stochastic differential equation:
\begin{align}
&\dd \Big( \Im m_N(t,z) \Big)=-\frac{1}{2 \ii N^{3/2}}\sum_{i, a, b=1}^N \Big( G_{ia} G_{bi}(z) -G_{ia} G_{bi}(\bar z)  \Big) \dd \beta_{ab}\nonumber\\
&\qquad\qquad\qquad+ \frac{1}{4N \ii}\sum_{i,a,b=1}^N \Big[ h_{ab} \Big(G_{ia} G_{bi}(z) -G_{ia} G_{bi}(\bar z) \Big)+\frac{1}{N} \Big(  G_{ib} G_{bi} G_{aa}(z)-G_{ib} G_{bi} G_{aa}(\bar z)  \Big)\nonumber\\
& \qquad \qquad \qquad \qquad \qquad\qquad+\frac{1}{N}  \Big( G_{ia} G_{ai} G_{bb}(z)-G_{ia} G_{ai} G_{bb}(\bar z)\Big) \Big]  \dd t \nonumber\\
\quad  =& - \frac{1}{N^{3/2}} \sum_{i, a, b=1}^N \widetilde{{\Im}} (G_{ia}G_{bi} ) \dd \beta_{ab}+\frac{1}{2N} \sum_{i,a, b=1}^N \Big[ h_{ab} \widetilde{{\Im}} ( G_{ia} G_{bi})+ \frac{1}{N} \widetilde{{\Im}} \big(  G_{ib} G_{bi} G_{aa}+ G_{ia} G_{ai} G_{bb} \big) \Big]\dd t,\nonumber
\end{align}
where we use the notation from~\eqref{tilde}.

Using Ito's formula similarly on $F(\X)$ and combining with (\ref{derivative2}), we obtain the stochastic differential equation (\ref{dF}), with the diffusion term given by 
\begin{align}\label{dm_1}
\dd M=-F'(\X)\Big(  \intkappa \frac{1}{\sqrt{N}}   \sum_{i, a, b=1}^N \widetilde{{\Im}} \Big(G_{ia} G_{bi}(t,2+x+\ii \eta)\Big) \dd x \Big) \dd \beta_{ab},
\end{align}
and the drift term given by (we omit the parameters $t$ and $z=2+x+\ii \eta$ of the Green functions)
\begin{align}\label{theta}
\Theta =& \frac{1}{2} \sum_{i, a, b=1}^N  h_{ab}\Big( F'(\X)  \intkappa  \widetilde{{\Im}}  (G_{ia} G_{bi})  \dd x\Big)+\frac{1}{2N}   \sum_{i, a, b=1}^N  F'(\X)  \intkappa \widetilde{{\Im}} \Big(  G_{ib} G_{bi} G_{aa}+ G_{ia} G_{ai} G_{bb}  \Big) \dd x\nonumber\\
&+\frac{1}{2}F''(\X) \frac{1}{N} \sum_{i,j=1}^N \sum_{a,b=1}^N \Big( \intkappa \widetilde{{\Im}} (G_{ia} G_{bi}) \dd x \Big)\Big( \intkappa \widetilde{{\Im}} (G_{jb} G_{aj}) \dd x \Big) .
\end{align}
Using $G^2(z)=\frac{\dd }{\dd z} G(z)$ and the definition of $\widetilde{{\Im}}$ in (\ref{tilde}), we write
\begin{align*}
 \sum_{i=1}^N  \int_{\kappa_1}^{\kappa_2}   \widetilde{{\Im}}  \Big((G_{ia} G_{bi})(t,2+x+\ii \eta) \Big)\dd x=& \int_{\kappa_1}^{\kappa_2}  \widetilde{{\Im}} \Big( \frac{\dd G_{ba}}{\dd x} (t,2+x+\ii \eta)\Big)\dd x= (\Dim G_{ba})(t,z_1,z_2),
 \end{align*}
 with $\Dim$ defined in (\ref{dim}) and $z_1,z_2$ given in (\ref{le zs}). Applied to the martingale term~\eqref{dm_1} we find~\eqref{mart}. Applied to the drift term (\ref{theta}), we find
 \begin{align}\label{theta1}
\Theta=& \frac{1}{2} \sum_{a, b=1}^N   h_{ab}\Big( F'(\X)  \Dim G_{ba} \Big) +\frac{1}{2N} \sum_{a, b=1}^N  \Big(  F'(\X)\Dim (G_{aa} G_{bb} )+ F''(\X) ( \Dim G_{ab} )( \Dim G_{ba}) \Big) .
\end{align}

Next, we take the expectation of $\Theta$ and apply the cumulant expansions in Lemma \ref{cumulant} with respect to the independent entries $h_{ab}$ in the first term on the right of~\eqref{theta1}. 
Using the relation (\ref{dH}), we compute
\begin{align}\label{int_1}
\frac{ \partial F'(\X) }{ \partial h_{ba} }=&F''(\X)  \sum_{i=1}^N \int_{\kappa_1}^{\kappa_2} \frac{ \partial (\Im G_{ii})}{ \partial h_{ba} } \dd x  =-F''(\X) \sum_{i=1}^N\int_{\kappa_1}^{\kappa_2}  \widetilde\Im (G_{ib} G_{ai}) \dd x=-F''(\X)  \Dim G_{ab}.
\end{align}
We first apply cumulant expansions to the complex-valued off-diagonal entries $h_{ab}$, \ie let $a \neq b$ in the summations in (\ref{theta1}).  Then by direct computations and using Assumption~\ref{assump} (b.), the second order terms in the cumulant expansions corresponding to $s_{ab}^{(1,1)}(t) \equiv 1$ are canceled with the second term on the right of (\ref{theta1}) with $a\neq b$. The third and fourth order terms in the cumulant expansions, corresponding to $p+q+1\in\{3,4\}$, are given in (\ref{K_term}). We stop the cumulant expansion at $l=4$ and the corresponding truncation error $R_5=\sum_{a\neq b} R^{(ab)}_{5}$ is estimated as follows. 

We have from (\ref{cumulant_error}) that 
\begin{align}\label{R_5}
	|R^{(ab)}_{5}| &\leq  C \E[ |h_{ab}|^{5}] \E \Big[ \max_{p+q=4} \Big\{ \sup_{|w| \leq N^{-1/2+\gamma}} \Big| \frac{\partial^{p+q}}{\partial h^p_{ba} \partial h^q_{ab}}f_{ab} \Big(H^{(ab)}+w E^{(ba)}+\bar{w} E^{(ab)} \Big)\Big| \Big\} \Big]\nonumber\\
	&+C  \E \Big[ |h_{ab}|^{5} 1_{|h_{ab}|> N^{-1/2+\gamma}}\Big] \E \Big[\max_{p+q=4} \Big\{ \sup_{w \in \C} \Big| \frac{\partial^{p+q}}{\partial h^p_{ba} \partial h^q_{ab}}f_{ab} \Big(H^{(ab)}+w E^{(ba)}+\bar{w} E^{(ab)} \Big)\Big| \Big\} \Big],
	\end{align}
	with a fixed small $\gamma>0$, and where we use the notation $E^{(ab)}:=(\delta_{ab})_{i,j=1}^N$, $H^{(ab)}:=H-h_{ab}E^{(ab)}-h_{ba}E^{(ba)}$, as well as
	\begin{align}\label{fab}
	f_{ab}(H):=F'(\X)  \Dim G_{ba}.
	\end{align}
	Using the second resolvent identity, we can write
	\begin{align}\label{resolvent_expansion}
	G^{H^{(ab)}}_{ij}=G^{H}_{ij}+ \Big(G^{H^{(ab)}} ( h_{ab}E^{(ab)}+h_{ba}E^{(ba)} ) G^{H}\Big)_{ij}.
	\end{align}
	From the local law in~\eqref{G}, we have $ \max_{i \neq j}|G^{H}_{ij}| \prec \Psi$ and $\max_{i}|G^{H}_{ii}| \prec 1$. In addition, we have $|h_{ij}| \prec \frac{1}{\sqrt{N}}$ from the moment condition (\ref{moment_condition}). Therefore, we have from (\ref{resolvent_expansion}) that $\max_{i \neq j}|G^{H^{(ab)}}_{ij}| \prec \Psi$ and $\max_{i}|G^{H^{(ab)}}_{ii}| \prec 1$. Similarly, we have
	\begin{align}\label{resolvent_expansion2}
	G^{H^{(ab)}+w E^{(ab)}+\bar{w} E^{(ba)}}_{ij}=G^{H^{(ab)}}_{ij}- \Big(G^{H^{(ab)}+w E^{(ab)}+\bar{w} E^{(ba)}} ( w E^{(ab)}+\bar{w} E^{(ba)} ) G^{H^{(ab)}}\Big)_{ij},
	\end{align}
	and thus
	\begin{align}
	\sup_{|w|<N^{-1/2+\gamma}}\Big\{ \max_{i, j} \Big|G^{H^{(ab)}+w E^{(ab)}+\bar{w} E^{(ba)}}_{ij} \Big|\Big\} \prec 1.
	\end{align}
	Combining with (\ref{dH}), (\ref{int_1}), and the fact that $F$ in (\ref{fab}) has bounded derivatives, we obtain that
	$$\sup_{|w| < N^{-1/2+\gamma}} \Big| \frac{\partial^{p+q}}{\partial h^p_{ba} \partial h^q_{ab}}f_{ab} \Big(H^{(ab)}+w E^{(ba)}+\bar{w} E^{(ab)} \Big)\Big| \prec 1.$$
	Together with $\E|h_{ij}|^5\le CN^{-5/2}$ under Assumption~\ref{assump} and Lemma \ref{dominant}, the first term on the right side of (\ref{R_5}) is bounded by $O_{\prec}(N^{-5/2})$. Note that for $z=E+\ii \eta$ with $\eta \geq N^{-1+\epsilon}$, we have the deterministic upper bound for $\max_{i,j} |G_{ij}| \leq \|G\|_{2} \leq \frac{1}{\eta}=O(N^{1-\epsilon})$. So the conditions of statement (3) of Lemma \ref{dominant} are satisfied, and we can directly bound the expectation of the first term on the right side of~\eqref{R_5}.
	
	We next estimate the second term on the right side of (\ref{R_5}). Using the deterministic bound $\max_{i,j} |G_{ij}|=O(N^{1-\epsilon})$, we have from (\ref{dH}), (\ref{int_1}) and the fact that $F$ in (\ref{fab}) has bounded derivatives that
	$$\max_{p+q=4} \Big\{ \sup_{w \in \C} \Big| \frac{\partial^{p+q}}{\partial h^p_{ba} \partial h^q_{ab}}f_{ab} \Big(H^{(ab)}+w E^{(ba)}+\bar{w} E^{(ab)} \Big)\Big| \Big\} =O(N^{5-5\epsilon}).$$
	Combining with the moment condition (\ref{moment_condition}) and H\"{o}lder's inequality, the second term on the right side of (\ref{R_5}) can also be bounded by $O_{\prec}(N^{-5/2})$. Thus the truncation error $R_5$ in the cumulant expansions satisfies $|R_5|=O_{\prec}(N^{-1/2})$.
	Throughout the paper, we will use similar arguments as above to estimate the error terms stemming from cutting cumulant expansions at some fixed order without specifically mentioning it.

Concerning the terms involving the diagonal entries $h_{aa}$ in (\ref{theta1}), we apply the cumulant expansion for real-valued random variables in Lemma \ref{cumulant} and stop at the second order $l=2$. The resulting second order term in combination with the second sum in (\ref{theta1}) with $a \equiv b$ is given by $E_2$ in (\ref{E_1}) and the truncation error is similarly bounded by $O_{\prec}(N^{-1/2})$. We have hence finished the proof of Lemma~\ref{lemma_first}.
\end{proof}

Having established Lemma~\ref{lemma_first}, we next estimate the expectation of the drift term $\E[\Theta]$ in~\eqref{derivative_gene} in the next proposition, whose proof is postponed to Section~\ref{sec:fourth}.
\begin{proposition}\label{prop_theta}
The drift term $\E[\Theta]$ in (\ref{derivative_gene}) has the following bound:
\begin{align}\label{E_Theta}
|\E[\Theta(t,z_1,z_2)]|\leq N^{-1/3+c \epsilon},
\end{align}
uniformly in $t\geq 0$ and $z_1,z_2$ given in (\ref{le zs}), for a numerical constant $c>0$ that does not depend on $\epsilon$ and sufficiently large $N \geq N_0(\epsilon,C_0)$.
\end{proposition}

In order to finish the proof of Theorem~\ref{green_comparison}, we now choose $T:=8 \log N$ and integrate (\ref{dF}) over $[0,T]$. Then taking the expectation, the diffusion term vanishes (see Remark \ref{diffusion_mart}) and the drift term is bounded using (\ref{E_Theta}). We hence find by writing out $\X$ in (\ref{X}) that
\begin{align}\label{sandwich1}
\Big|\E\Big[ F \Big( N\int_{\kappa_1}^{\kappa_2} \Im m_N (0, 2+x+\i \eta) \dd x \Big) \Big]-\E\Big[ F \Big( N\int_{\kappa_1}^{\kappa_2} \Im m_N (T, 2+x+\i \eta) \dd x \Big) \Big] \Big|=O(N^{-\frac{1}{3}+c \epsilon} \log N).
\end{align}
Using the inequality $\|A\|_{\mathrm{max}} \leq \|A\|_{2} \leq N \|A\|_{\mathrm{max}}$, the second resolvent identity, that $\|G(E+\ii \eta)\|_{2} \leq \frac{1}{\eta}$, and (\ref{sum_H}), one shows that $G(T,z)$ is sufficiently close to the Green function of the GUE, \ie
\begin{align}\label{approxxxx}
\|G(T,z) -G^{\mathrm{GUE}}(z)\|_{\mathrm{max}} \leq \|G(T,z)(\mathrm{GUE}-H(T))G^{\mathrm{GUE}}(z)\|_2 \leq \frac{N}{\eta^2} \|(\mathrm{GUE}-H(T))\|_{\mathrm{max}} \prec \frac{1}{N^3\eta^2}.
\end{align}
Since $F$ is a smooth function with uniformly bounded derivatives, we have
 \begin{equation}\label{sandwich2}
 \Big| F \Big( N\int_{\kappa_1}^{\kappa_2} \Im m_N (T, 2+x+\i \eta) \dd x \Big) -F \Big(  N\int_{\kappa_1}^{\kappa_2} \Im m_N^{\mathrm{GUE}} (2+x+\i \eta) \dd x \Big) \Big| \prec \frac{N^{\epsilon}}{N^{8/3} \eta^2}.
 \end{equation}
 Combining (\ref{sandwich1}) and (\ref{sandwich2}), we conclude the proof of Theorem \ref{green_comparison}.
\end{proof}

\begin{remark}
In the traditional approach to the Green function comparison theorem~\cite{rigidity} a Lindeberg type replacement strategy is used. In~\eqref{sum_H} we use a continuous flow  to interpolate between Wigner matrices and the invariant ensembles. This is notationally easier than the Lindeberg replacement, especially when we do recursive comparisons to estimate the contributions from the fourth order cumulants in Section~\ref{sec:match}.
\end{remark}

\section{A special case: estimates on $\E[\Im m_N]$}\label{toy}
In this section, we prove the simplest version of the Green function comparison theorem, Theorem~\ref{green_comparison}, when $F(x)=x$. It then suffices to compare the expected normalized trace of the Green function of a Wigner matrix $\E[m_N(z)]$ with $\E^{\mathrm{GUE}}[m_N(z)]$. The ideas in this section will also be used to prove Proposition~\ref{prop_theta}, which is a key ingredient to establish the Green function comparison theorem for a general function~$F$. The proof for general functions $F$ will rely on the estimate (\ref{img_wigner}) in Proposition \ref{GCT_mn} below as an input.

\begin{proposition}\label{GCT_mn}
Let $H_N$ be a complex Wigner matrix satisfying Assumption \ref{assump} and recall the time dependent matrix $H(t)$ in (\ref{ou_process}). For any $\epsilon>0$ and $C_0>0$, define the domain of the spectral parameter $z$ near the upper edge, 
\begin{align}\label{S_edge}
S_{\mathrm{edge}} \equiv S_{\mathrm{edge}}(\epsilon,C_0):=\{ z=E+\ii \eta \in S: |E-2| \leq C_0  N^{-2/3+\epsilon}, N^{-1+\epsilon} \leq \eta \leq N^{-2/3+\epsilon}\}\,,
\end{align}
with $S$ given in (\ref{ddd}). Then for any $\tau>0$, we have 
\begin{align}\label{compare_mn}
\Big| \E[m_N(t,z)]- \E^{\mathrm{GUE}}[m_N(z)] \Big| \leq N^{-1/3+\tau},
\end{align}
uniformly in $z \in S_{\mathrm{edge}}$ and $t \geq 0$, for sufficiently large $N \geq N_0(C_0, \epsilon, \tau)$. Furthermore, there exists some $C>0$ independent of $\epsilon$, such that
\begin{align}\label{img_wigner}
\E[\Im m_N(t, z)] \leq C N^{-1/3+\epsilon},
\end{align}
uniformly in $z \in S_{\mathrm{edge}}$ and $t \geq 0$, for sufficiently large $N \geq N'_0(C_0, \epsilon)$.
\end{proposition}

In the rest of this section we prove Proposition~\ref{GCT_mn}; its proof is split into several parts organized in subsections.

\subsection{Interpolation between a Wigner matrix and the GUE}\label{sec:interpolate}
Following the proof of Lemma~\ref{lemma_first} in Section~\ref{sec: proof_main_result}, we start by applying Ito's lemma to the time dependent normalized trace of the Green function, $m_N(t,z)$, from~\eqref{le time dependent G}.We find using~(\ref{derivative2}) that
\begin{align}\label{diff_eq_1}
\dd (m_N(t,z))=& - \frac{1}{N^{3/2}} \sum_{v, a, b=1}^N G_{va}G_{bv}  \dd \beta_{ab}+\frac{1}{2N} \sum_{v,a,b=1}^N \Big( h_{ab} G_{va} G_{bv}+\frac{1}{N}  G_{vb} G_{bv} G_{aa}+\frac{1}{N}  G_{va} G_{av} G_{bb} \Big) \dd t\nonumber\\
:=&\dd M_0+ \Theta_0 \dd t,
\end{align}
with diffusion term $\dd M_0$ and drift term $\Theta_0\dd t \equiv \Theta_0(t,z)\dd t$; here we use the subscript $0$ to indicate that we are considering the simple case $F(x)=x$. The diffusion term $\dd M_0$ yields a martingale after integration; see Remark \ref{diffusion_mart}. Taking the expectation of the drift term and applying the cumulant expansions in Lemma \ref{cumulant}, we have the analogue of (\ref{derivative_gene}),
\begin{align}\label{step0}
\E[\Theta_0]=&\frac{1}{2N^2} \sum_{v, a=1}^N ( s^{(2)}_{aa}-1) \E \Big[ \frac{ \partial  (G_{av} G_{va})}{\partial h_{aa} } \Big]+\frac{1}{2N} \sum_{\substack{v, a,b=1\\ a\neq b}}^N \sum_{p+q+1=3}^{4} \frac{1}{p! q!}  \frac{s^{(p,q+1)}_{ab}}{N^{\frac{p+q+1}{2}}}  \E \Big[ \frac{ \partial^{p+q}  (G_{bv} G_{va})}{\partial h^p_{ba} \partial h^q_{ab}} \Big]+O_{\prec}\big(\frac{1}{\sqrt{N}}\big)\nonumber\\
=&-\frac{1}{2N^2} \sum_{v, a=1}^N ( s^{(2)}_{aa}-1) \E \Big[ \frac{ \partial^2  G_{vv}}{\partial h^2_{aa} } \Big]- \sum_{p+q+1=3}^{4}  \frac{1}{2 p! q! N^{\frac{p+q+3}{2}}} \sum_{\substack{v, a,b=1\\ a\neq b}}^N s^{(p,q+1)}_{ab} \E \Big[ \frac{ \partial^{p+q+1} G_{vv}}{\partial h^p_{ba} \partial h^{q+1}_{ab}} \Big]+O_{\prec}\big(\frac{1}{\sqrt{N}}\big)\,,
\end{align}
where the error stems from the truncation of the cumulant expansions at fourth order. Recalling the arguments in Section~\ref{sec: proof_main_result}, in order to establish Proposition~\ref{GCT_mn} it suffices to show that for any $\tau>0$,
\begin{align}\label{claim}
|\E[\Theta_0(t,z)]| \leq N^{-1/3+\tau},
\end{align}
uniformly in $z \in S_{\mathrm{edge}}(\epsilon,C_0)$ and $t \geq 0$, for sufficiently large $N \geq N_0(C_0, \epsilon, \tau)$. 

Admitting (\ref{claim}), for $T=8 \log N$ and any $0 \leq t'\leq T$, we integrate (\ref{diff_eq_1}) over $[t',T]$ and take the expectation to get 
\begin{align}\label{sandwich3}
\Big| \E \Big[m_N(t',z) \Big]-\E\Big[ m_N(T, z) \Big] \Big| =O\Big(N^{-1/3+\tau} \log N\Big).
\end{align}
Combining with (\ref{approxxxx}), we obtain the comparison estimate in (\ref{compare_mn}) between the GUE and the time dependent $H(t)$ in (\ref{ou_process}) staring from the Wigner matrix $H$. The bound (\ref{img_wigner}) will follow directly from the comparison result (\ref{compare_mn}) and the corresponding estimate for the GUE in Lemma \ref{lemma_trace} below. 

In the remaining part of this section, we will hence prove (\ref{claim}). For that it suffices to estimate the terms on the right side of~(\ref{step0}).

\subsection{Third and fourth order terms with unmatched indices}

Using the differential rule for the Green function entries in (\ref{dH}), each term in the cumulant expansion (\ref{step0}) can be written out in terms of an averaged product of the Green function entries. The first group of terms on the right side of (\ref{step0}) is given by
$$-\frac{2}{N^2} \sum_{v, a=1}^N ( s^{(2)}_{aa}-1) \E \Big[ G_{va}G_{av} G_{aa}\Big]\,.$$
In the second group of terms on the right side of (\ref{step0}), one example of a third order term with $p=1, q=1$ and one example of a fourth order term with $p=2,q=1$  are given by,
$$\sqrt{N} \frac{1}{N^{3}} \sum_{v, a,b}   \frac{s^{(1,2)}_{ab}}{2} \E \Big[ G_{va} G_{bv} G_{aa} G_{bb}  \Big], \qquad -\frac{1}{N^{3}} \sum_{v, a,b} \frac{ s^{(2,2)}_{ab}}{4} \E \Big[ G_{va} G_{av}  G_{aa} G_{bb} G_{bb}\Big].$$
We remark that the third order terms with $p+q+1=3$ are averaged products of Green function entries with an additional leading factor $\sqrt{N}$.

To study these averaged products of the Green function entries in (\ref{step0}), we introduce the general form in (\ref{product}) below. We will use the letters $v_j$ to denote the free summation indices running from 1 to $N$, and the letters $x_i,y_i$ as the row and column indices of the Green function entries. In order to avoid confusion, we clarify that $x_i=y_i=v_j$ means that both $x_i$ and $y_i$ represent the same summation index~$v_j$. Further we write $x_i \neq y_i$ if~$x_i$ and~$y_i$ represent two distinct summation indices, say $v_j$ and $v_{j'}$. They could have the same value as the summation indices $v_j$ and $v_{j'}$ run from $1$ to $N$.

We are now ready to introduce the general form of averaged products of the Green function entries:
\begin{align}\label{product}
\frac{1}{N^{m}}  \sum_{v_1=1}^N\cdots \sum_{v_m=1}^N   c_{v_1, \ldots, v_m} \Big(  \prod^n_{i=1} G_{x_{i} y_{i}} (t,z)\Big)=:\frac{1}{N^{\# \mathcal{I}}}\sum_{\mathcal{I}} c_{\mathcal{I}}\Big(\prod^n_{i=1} G_{x_i y_i}(t,z)\Big), \qquad t \in \R^+,  z \in \C^+,
\end{align}
for $m,n \in \N$, where $\mathcal{I}:=\{v_j\}_{j=1}^m$ is a free summation index set which may include $a,b,v$ from (\ref{step0}), $m:=\#\{\mathcal{I}\}$ is the number of the free summation indices, and the coefficients $\{c_{\mathcal{I}}:=c_{v_1, \ldots, v_m}\}$ are uniformly bounded complex numbers. Moreover, $n$ is the number of Green function entries in the product, and each row index $x_i$ and column index $y_i$ $(1 \leq i\leq n)$ of the Green function entries represent some element in the free summation index set~$\mathcal{I}$.

We further define the degree of such a term in (\ref{product}) to be the number of off-diagonal terms in the product of the Green function entries, \ie
\begin{align}\label{degree_0}
d:=\#\{ 1 \leq i\leq n: x_i \neq y_i\}.
\end{align}
In particular, we have $0 \leq d \leq n$. We use $\mathcal{Q}_d \equiv \mathcal{Q}_d(t,z)$ to denote the collection of the averaged products of the Green function entries of the form in (\ref{product}) of degree $d$. For any $Q_d \equiv Q_d(t,z) \in \mathcal{Q}_d$, it is clear from the local law in (\ref{G}) that 
\begin{align}\label{localaw_0}
|Q_d(t,z)| \prec \Psi^d+\frac{1}{N},
\end{align}
uniformly in $z \in S$ given in (\ref{ddd}) and $t \geq 0$. We will often omit the parameters $z$ and $t$ for notational simplicity. The last error $N^{-1}$ is from the coincidence of distinct summation indices.

Now we first look at the third order terms in the cumulant expansion (\ref{step0}) with $p+q+1=3$. Using the differential rule for the Green function entries in (\ref{dH}), all the third order terms with $p+q+1=3$ can be written out in the form in (\ref{product}), with an extra factor $\sqrt{N}$ in front.  We observe that these terms are unmatched, see Definition \ref{unmatch_def} below, since the indices $a,b$ both appear an odd number of times in the product of the Green function entries. 

In a similarly way, the fourth order terms in the cumulant expansion (\ref{step0}) with $p+q+1=4$, except the ones corresponding to $p=2,q=1$, are also unmatched terms of the form in (\ref{product}) from Definition~\ref{unmatch_def}, since the number of times the index $a~( \mbox{or } b)$ appears in the row index set does not agree with the number of times it appears in the column index set of the product of Green function entries.

\begin{definition}[Terms with unmatched indices]\label{unmatch_def}
Given any $Q_d \in \mathcal{Q}_d$ of the form in (\ref{product}) of degree $d$, let $\nu^{(r)}_j$, $\nu^{(c)}_j$, be the number of times the free summation index $v_j \in \mathcal{I}$ appears as the row, respectively column, index in the product of the Green function entries, \ie
\begin{align}\label{nu_rc}
\nu^{(r)}_j:=\#\{1\leq i \leq n: x_i=v_j\},\qquad \nu^{(c)}_j:=\#\{1\leq i \leq n: y_i=v_j\}, \quad 1\leq j\leq m.
\end{align}
We define the set of the unmatched summation indices as 
$$\mathcal{I}^o := \{1 \leq j \leq m: \nu^{(r)}_j \neq \nu^{(c)}_j \} \subset \mathcal{I}.$$ 
If $\mathcal{I}^o$ is empty, \ie all the free summation indices appear the same number of times in the row index set $\{x_i\}$ and the row column index set $\{y_i\}$, then we say that $Q_d$ is matched. Otherwise, we say $Q_d$ is an unmatched term, denoted by $Q_d^o$. The collection of the unmatched terms of the form in (\ref{product}) of degree $d$ is denoted by~$\mathcal{Q}^o_d \subset \mathcal{Q}_d$. 

Given any unmatched term $Q_d^o \in \mathcal{Q}^o_d$, we define the unmatched index set for both row and column as
\begin{align}\label{rc_index}
\mathcal{R}^o := \{1 \leq j \leq m: \nu^{(r)}_j > \nu^{(c)}_j \} \subset \mathcal{I}^o;\qquad \mathcal{C}^o := \{1 \leq j \leq m: \nu^{(r)}_j < \nu^{(c)}_j \} \subset \mathcal{I}^o.
\end{align}
Neither of $\mathcal{R}^o$ and $\mathcal{C}^o$ is empty. Moreover, $\mathcal{R}^o \cap \mathcal{C}^o$ is empty, and $\mathcal{R}^o \cup \mathcal{C}^o=\mathcal{I}^o.$
\end{definition}

Next, we give two examples of unmatched terms, which appear as fourth order terms in (\ref{step0}),
\begin{align}
- \frac{1}{N^{3}} \sum_{v, a,b} \frac{ s^{(1,3)}_{ab}}{4} \E \Big[ G_{va} G_{bv} G_{ba} G_{aa} G_{bb} \Big] \in \mathcal{Q}^{o}_{3}; \qquad - \frac{1}{N^{3}} \sum_{v, a,b}  \frac{s^{(0,4)}_{ab} }{12}\E \Big[ G_{va}G_{bv}  G_{ba} G_{ba}G_{ba} \Big] \in \mathcal{Q}^{o}_{5}; 
\end{align}
and two examples of the unmatched terms from the third order terms on the right side of (\ref{step0}),
\begin{align}
\frac{1}{N^{3}} \sum_{v, a,b}   \frac{s^{(1,2)}_{ab}}{2} \E \Big[ G_{va} G_{bv} G_{aa} G_{bb}  \Big] \in \mathcal{Q}^{o}_2; \qquad \frac{1}{ N^{3}} \sum_{v, a,b}  \frac{s^{(0,3)}_{ab}}{4} \E \Big[ G_{va} G_{bv} G_{ba} G_{ba}\Big] \in \mathcal{Q}^{o}_4, 
\end{align}
up to a factor of $\sqrt{N}$.

The following proposition states that the expectations of the unmatched terms are much smaller than their naive size obtained by the power counting from the local law as in (\ref{localaw_0}). The proof is postponed to Section \ref{sec:unmatch}.
\begin{proposition}\label{unmatch_lemma}
Consider any unmatched term $Q^o_d \in \mathcal{Q}_d^o$ of degree $d$ with fixed $n$ (the number of Green function entries in the product) given in (\ref{product}). For any fixed $D \in \N$, we have
\begin{align}\label{unmatch_lemma equation}\E[Q^o_d(t,z)]=O_{\prec}\Big(\frac{1}{N}+\Psi^D\Big)\,,
\end{align}
uniformly in $z \in S$ given in (\ref{ddd}) and $t \geq 0$.
\end{proposition}

\begin{remark}\label{remark1 about different z}
In the observable $Q^o_d(t,z)$ in~\eqref{unmatch_lemma equation} the Green function entries from~\eqref{product} are all chosen at the same spectral parameter $z\in S$. Our proofs can be extended to the setting where the Green function entries are evaluated at different spectral parameters in the domain $S$ with the estimate in~\eqref{unmatch_lemma equation} holding true. As we do not require this generalization to prove Proposition~\ref{GCT_mn} we do not pursue this direction here.
\end{remark}

Therefore, using Proposition \ref{unmatch_lemma}, the third order terms in the cumulant expansion (\ref{step0}) are all bounded as $O_{\prec}(N^{-1/2}+\sqrt{N} \Psi^D)$. Moreover all the fourth order terms in the cumulant expansion (\ref{step0}), except the one corresponding to $p=2,q=1$, are bounded by $O_{\prec}(N^{-1}+\Psi^D)$. By choosing $D \geq \frac{1}{\epsilon}$ with $\epsilon>0$ as in (\ref{control_eq}), we hence obtain from (\ref{step0}) that
\begin{align}\label{step}
\E[\Theta_0]=&-\frac{1}{2N^2} \sum_{v, a=1}^N( s^{(2)}_{aa}-1) \E \Big[ \frac{ \partial^2  G_{vv} }{\partial h^2_{aa} } \Big]-\frac{1}{4N^3}\sum_{\substack{v, a,b=1\\ a\neq b}}^N s^{(2,2)}_{ab} \E \Big[ \frac{ \partial^4 G_{vv} }{\partial h^2_{ba} \partial h^2_{ab} } \Big]+O_{\prec}(\frac{1}{\sqrt{N}}).
\end{align}
The remaining terms on the right side of (\ref{step}) are matched under Definition \ref{unmatch_def}. It is thus sufficient to estimate these matched terms, as presented in the next subsection.

\subsection{Terms with matched indices}
Applying the differentiation rule (\ref{dH}) to the right side of (\ref{step}), the index $v$ appears once as a row index and once as a column index of the Green function entries of the resulting terms on the right side of (\ref{step}). In addition, the indices $a,b$ from~\eqref{step} will take a special role and appear twice as a row index and twice as a column index of the Green function entries.  After differentiation by (\ref{dH}), we write out these products of Green function entries and observe that they are of the following form which we call type-AB terms. 

\begin{definition}[Type-AB terms, type-A terms, type-0 terms]\label{def_type_AB}
 For arbitrary $m,n\in\N$, we consider averaged products of Green functions of the form
\begin{align}\label{form}
\frac{1}{N^{m+2}}  \sum_{v_1=1}^N\cdots \sum_{v_m=1}^N  \sum_{a=1}^N\sum_{b=1}^N c_{a,b, v_1, \ldots, v_m} \Big(  \prod_{i=1}^n G_{x_{i} y_{i}} (t,z)\Big)=:\frac{1}{N^{\# \mathcal I+2}} \sum_{\mathcal{I},a,b} c_{a,b, \mathcal{I}} \Big( \prod_{i=1}^n G_{x_{i} y_{i}}\Big) \,,
\end{align}
for $t\in\R^+,~z\in\C^+$, where each $x_i$ and $y_i$ represent the free summation indices $a$, $b$ or $v_j$ $(1 \leq j \leq m)$. Here the coefficients $\{c_{a,b,\mathcal{I}} := c_{a,b, v_1, \ldots, v_m}\}$ are uniformly bounded complex numbers. Note that the form in (\ref{form}) is a special case of the form given in (\ref{product}) with the two indices $a$ and $b$ singled out. The degree, denoted by $d$, of such a term is defined as in (\ref{degree_0}) by counting the number of the off-diagonal Green function entries. Recall $\nu_j^{(r)}$, $\nu_j^{(c)}$ defined in~(\ref{nu_rc}). We further define similarly 
$$\nu^{(r)}_{a}:=\#\{i: x_i=a\},~ \nu^{(c)}_{a}:=\#\{i: y_i=a\}, ~ \nu^{(r)}_{b}:=\#\{i: x_i=b\}, ~ \nu^{(c)}_{b}:=\#\{i: y_i=b\}, $$
for the special indices $a$, $b$.

A {\it type-AB term}, denoted by $P_d^{AB}$, is of the form in (\ref{form}) with each $v_j$ appearing once in the row index set $\{x_i\}$ and once in the column index set $\{y_i\}$ in the product of the Green function entries, \ie $\nu_j^{(r)}=\nu_j^{(c)}=1$. The indices $a$ and $b$ both appear the same number of times (more than once) in the row index set $\{x_i\}$ and column index set $\{y_i\}$ in the product of the Green function entries, \ie $\nu_a^{(r)}=\nu_a^{(c)} \geq 2$ and $\nu_b^{(r)}=\nu_b^{(c)} \geq 2$. We denote by $\mathcal{P}_d^{AB} \equiv \mathcal{P}_d^{AB}(t,z)$ the collection of the type-AB terms of degree~$d$. We remark that type-AB terms are matched in the sense of Definition ~\ref{unmatch_def}.

A {\it type-A term}, denoted by $P_d^{A}$, is of the form in (\ref{form}) with $\nu_a^{(r)}=\nu_a^{(c)} \geq 2$, and $\nu_b^{(r)}=\nu_b^{(c)}=\nu_j^{(r)}=\nu_j^{(c)}=1$ for $1 \leq j \leq m$. We denote the collection of the type-A terms of degree $d$ by $\mathcal{P}_d^{A} \equiv \mathcal{P}^{A}_d(t,z)$.

Finally, a {\it type-0 term}, denoted by $P_d$, is of the form in (\ref{form}) with all the free summation indices appearing once in the row index set $\{x_i\}$ and once in the column index set $\{y_i\}$ in the product of the Green function entries, \ie $\nu_a^{(r)}=\nu_a^{(c)}=\nu_b^{(r)}=\nu_b^{(c)}=\nu_j^{(r)}=\nu_j^{(c)}=1$ for $1 \leq j \leq m$. We denote the collection of the type-0 terms of degree~$d$ by $\mathcal{P}_d \equiv \mathcal{P}_d(t,z)$.

We remark that the index $b$ does no longer play a special role in type-A terms, as well as the indices $a$ and $b$ are not special in type-0 terms. We keep them in the notation in order to emphasize the inheritance from the form in~(\ref{form}).

\end{definition}

Next, we give two examples for type-AB terms, which are generated from the fourth order expansion terms in (\ref{step}) corresponding to the $(2,2)$-cumulants,
\begin{equation*}
-\frac{1}{4N^3} \sum_{v,a,b} s^{(2,2)}_{ab} \Big( G_{va} G_{aa} G_{av} G_{bb} G_{bb} \Big) \in \mathcal{P}^{AB}_{2}; \qquad -\frac{1}{4N^3} \sum_{v,a,b} s^{(2,2)}_{ab} \Big( G_{va} G_{ab} G_{bv} G_{aa} G_{bb} \Big) \in \mathcal{P}^{AB}_{3}\,;
\end{equation*}
and an example of a type-A term, which is from the second order terms of diagonal entries in the cumulant expansion~(\ref{step}),
\begin{equation*}
-\frac{1}{2N^2} \sum_{v, a}( s^{(2)}_{aa}-1) \Big( G_{va} G_{aa} G_{av} \Big) \in \mathcal{P}^{A}_2,
\end{equation*}
where the index $b$ no longer takes the special role. 

 In the following, we only consider special type-AB terms with both indices $a$ and~$b$ appearing in the product of the Green function entries four times (\ie $\nu_a^{(r)}=\nu_a^{(c)}=\nu_b^{(r)}=\nu_b^{(c)}=2$) and the corresponding type-A terms. For the general case, see Remark~\ref{remark_more_than_4}.

The next proposition claims that, in expectation, any type-AB term as well as any type-A term of degree $d$ can be expanded into linear combinations of type-0 terms of degrees at least $d$ up to negligible error. The proof of Proposition~\ref{lemma_expand_type} is presented in Subsection \ref{toy_fourth}.

\begin{proposition}\label{lemma_expand_type}
Consider any type-AB term $P_d^{AB} \in \mathcal{P}^{AB}_d$ of the form in (\ref{product}) of degree $d$ with fixed $n \in \N$, and $\nu_a^{(r)}=\nu_a^{(c)}=\nu_b^{(r)}=\nu_b^{(c)}=2$. Then for any fixed $D \in \N $, we have
\begin{align}\label{reduce}
\E [  P_d^{AB}(t, z)]=\sum_{\substack{P_{d'} \in \mathcal{P}_{d'} \\ d \leq d' < D}} \E [ P_{d'}(t,z)]+O_{\prec}\Big(\frac{1}{\sqrt{N}}+\Psi^D\Big),
\end{align}
uniformly in $z \in S$ (see (\ref{ddd})), $t\in \R^+$, where we use $\sum_{P_{d'} \in \mathcal{P}_{d'}, d \leq d' < D}\E [ P_{d'}(t,z)]$ to denote a sum of finitely many type-0 terms of the form in (\ref{form}) of degrees $d'$ satisfying $d \leq d' < D$. Moreover, the number of type-0 terms in the sum above is bounded by $(6(n+8D))^{2D}$ and the number of the Green function entries in each type-0 term is bounded by $n+8D$.

Similarly, for any type-A term $P_d^{A} \in \mathcal{P}_d^A$ of the form in (\ref{form}) with $\nu_a^{(r)}=\nu_a^{(c)}=2$, we have
\begin{align}
\E [  P_d^{A}(t,z)]=\sum_{\substack{P_{d'} \in \mathcal{P}_{d'} \\ d \leq d' < D}} \E [ P_{d'}(t,z)]+O_{\prec}\Big(\frac{1}{\sqrt{N}}+\Psi^D\Big)\,,
\end{align}
uniformly in $z \in S$ and $t\in \R^+$. The number of the type-0 terms in the sum above is at most~$(6(n+4D))^{D}$, and the number of the Green function entries in each type-0 term is at most $n+4D$.
\end{proposition}
\begin{remark}\label{remark_more_than_4}
The above expansions also hold true if we consider a slightly generalized setup when both indices $a$ and $b$ appear arbitrary even number of times, not limited to $\nu_a^{(r)}=\nu_a^{(c)}=\nu_b^{(r)}=\nu_b^{(c)}=2$. Then the number of expansions generated on the right side also depends on the values $\nu_a^{(r)}(=\nu_a^{(c)})$ and $\nu_b^{(r)}(=\nu_b^{(c)})$; see also Remark~\ref{moreab}. Furthermore, in the above all the Green function entries are taken at the same spectral parameter $z\in S$. The expansion results may be generalized to the setting when the Green functions are taken at different spectral parameters in the domain $S$, c.f.\ Remark~\ref{remark1 about different z}.
\end{remark}

Armed with Proposition \ref{lemma_expand_type}, we return to (\ref{step}). Recalling Definition \ref{def_type_AB} and using (\ref{dH}), the second group of terms on the right side of (\ref{step}) can be written out as type-AB terms of the form in (\ref{form}) of degrees satisfying $d \geq 2$, where the number of Green function entries in each type-AB term is $n=5$, the summation index set $\mathcal{I}=\{v\}$ and the coefficients $c_{a,b,v}=s^{(2,2)}_{ab}$. Similarly, the first group of terms on the right side of (\ref{step}) can be written as a type-A term with degree $d=2$ and the number of Green function entries $n=3$. Therefore, from Proposition \ref{lemma_expand_type}, we can expand (\ref{step}) as a sum of finitely many type-0 terms of degrees at least two, \ie
\begin{align}\label{final}
\E[\Theta_0(t,z)]=\sum_{ \substack{P_{d} \in \mathcal{P}_{d} \\ 2 \leq d \leq D-1}} \E [ P_d(t, z) ]+O_{\prec}\Big(\frac{1}{\sqrt{N}}+\Psi^D\Big)\,,
\end{align}
uniformly in $z \in S$ and $t\in \R^+$, where the number of type-0 terms in the sum above can be bounded by $(CD)^{cD}$, for some numerical constants $C,c$.

Having expanded $\E[\Theta_0(t,z)]$ into type-0 terms, we next estimate the size of type-0 terms of the form in (\ref{form}) of degree $d \geq 2$ in the edge scaling, \ie when the spectral parameter $z$ is chosen in the domain $S_{\mathrm{edge}}$ defined in~\eqref{S_edge}. The proof of Lemma~\ref{lemma_trace_wigner} is presented in Subsection \ref{toy_type_zero}.
\begin{lemma}\label{lemma_trace_wigner}
For any type-0 term $P_d \in \mathcal{P}_d$ of the form in (\ref{form}) of degree $d \geq 2$ with fixed $n \in \N$, we have
\begin{align}\label{tracek_wigner}
|\E [P_d(t,z)]| =O_{\prec}(N^{-1/3}),
\end{align}
uniformly in $z \in S_{\mathrm{edge}}$ given by (\ref{S_edge}) and $t \geq 0$.
\end{lemma}

We hence obtain the estimate of $\E[\Theta_0(t,z)]$ in (\ref{claim}) by combining (\ref{final}) and (\ref{tracek_wigner}), and by choosing~$D \geq \frac{1}{\epsilon}$ and using the upper bound in (\ref{control_eq}). This yields the proof of Proposition \ref{GCT_mn}.

\section{Product of Green function entries with matched indices}\label{sec:match}

In this section, we prove Proposition~\ref{lemma_expand_type} and Lemma~\ref{lemma_trace_wigner}. Before diving into their proofs, we outline in the next subsection the intuition stemming from the GUE.

\subsection{Intuition from the GUE}\label{subsec:gue}
In this subsection, we focus on the special case of the GUE. The idea of eliminating the indices appearing more than twice and reducing type-AB to type-0 terms as in Proposition~\ref{lemma_expand_type} stems from explicit computations for the GUE based on the Weingarten calculus for Haar unitary matrices. To simplify the arguments, we only consider the following example of a type-AB term of the form in (\ref{form}),
\begin{align}\label{example_ab}
\frac{1}{N^2}\sum_{a,b} (G_{aa}(z))^2 (G_{bb}(z))^2 \in \mathcal{P}_0^{AB}.
\end{align}
Thanks to the unitary conjugation invariance, we know that the eigenvalues $(\lambda_i)$ and the corresponding orthonormal eigenvectors $(\u_i)$ of a GUE matrix are independent, and that the collection of eigenvectors $U:=(\u_1,\cdots, \u_N)$ is distributed according to Haar measure on the unitary group $U(N)$. 

Further, using the spectral decomposition
\begin{align}\label{le spectral decomp}
G(z)=\frac{1}{H-z}=\sum_{j=1}^N \frac{\u_j \u_j^{*}}{\lambda_{j}-z}\,,\qquad\qquad z\in S\,,
\end{align}
we write the expectation of (\ref{example_ab}) as
\begin{align}\label{sum_gue}
\frac{1}{N^2}&\sum_{a,b} \E [(G_{aa}(z))^2 (G_{bb}(z))^2]=\frac{1}{N^2}\sum_{a,b} \sum_{j,k,p,q}\E \Big[\frac{\u_j(a) \overline{\u_j(a)}\u_k(a) \overline{\u_k(a)}\u_p(b) \overline{\u_p(b)}\u_q(b) \overline{\u_q(b)}}{(\lambda_{j}-z)(\lambda_{k}-z)(\lambda_{p}-z)(\lambda_{q}-z)}\Big]\nonumber\\
=&\frac{1}{N^2}\sum_{a,b} \sum_{j,k,p,q}\E \Big[ \frac{1}{(\lambda_{j}-z)(\lambda_{k}-z)(\lambda_{p}-z)(\lambda_{q}-z)}\Big] \times \E[U_{aj} U_{ak} U_{bp} U_{bq}  \overline{U_{aj}} \overline{U_{ak}} \overline{U_{bp}}\overline{U_{bq}}]\,.
\end{align}

In order to estimate the expectations of the eigenvectors above, we use the following result for the Weingarten calculus on the unitary groups~\cite{weingarten1, weingarten}.

\begin{lemma}[Corollary 2.4, Proposition 2.6 in~\cite{weingarten}]\label{wein_lemma}
Let $U=(U_{ij})_{i,j=1}^N$ be a Haar unitary random matrix of size $N$. Let $n\in \N$ and denote by $S_n$ the symmetric group of order $n$. Then, for arbitrary column and row indices $i_k,i'_k,j_k,j_k'\in\llbracket 1,N\rrbracket$, $1\le k\le n$, we have
\begin{align}\label{wein_temp}
\E [U_{i_1j_1} \cdots U_{i_n j_n} \overline{U_{i'_1j'_1}} \cdots \overline{U_{i'_n j'_n}} ]=\sum_{\alpha, \beta \in S_n} \delta_{i_1, i'_{\alpha(1)}} \cdots \delta_{i_n, i'_{\alpha(n)}}\delta_{j_1, j'_{\beta(1)}} \cdots \delta_{j_n, j'_{\beta(n)}} \mathrm{Wg}(N, \alpha^{-1} \beta)\,,
\end{align}
where $\mathrm{Wg}(N, \gamma)$ is the Weingarten function given by
\begin{align}
\mathrm{Wg}(N, \gamma):=\E [U_{11} \cdots U_{nn} \overline{U_{1 \gamma(1)}} \cdots \overline{U_{n, \gamma(n)}} ]\,,\qquad\qquad \gamma\in S_n\,.
\end{align}
In the limit of large $N$, the Weingarten function $\mathrm{Wg}(N, \gamma)$ has the following asymptotic behavior: Let $\{c_i\}_{i=1}^{\# (\gamma)}$ denotes the cycles of $\gamma\in S_n$, with $\# (\gamma)$ the total number of cycles. Then
\begin{align}\label{wein_fun}
\mathrm{Wg}(N, \gamma)
=&N^{\#(\gamma) -2 n}\prod_{i=1}^{\# (\gamma)} (-1)^{|c_i|-1} \mathrm{Cat}(|c_i|-1)+O(N^{\#(\gamma)-2n-2})\,,
\end{align}
where $|c_i|$ denotes the length of the cycle $c_i$ and $\mathrm{Cat}(k)=\frac{(2k)!}{k!(k+1)!}$ is the $k$-th Catalan number.
\end{lemma} 

Now we are ready to evaluate, for large $N$, $\E[U_{aj} U_{ak} U_{bp} U_{bq}  \overline{U_{aj}} \overline{U_{ak}} \overline{U_{bp}}\overline{U_{bq}}]$ from~\eqref{sum_gue} using Lemma~\ref{wein_lemma} with $n=4$. We may assume that $a \neq b$, as the case $a=b$ only contributes $O(N^{-1})$ to the expectation of~(\ref{example_ab}) uniformly for $z\in S$, using the local law (\ref{G}) and Lemma \ref{prop_msc}. We set $n=4$, $i_1=i_2=i_1'=i_2'=a$, $i_3=i_4=i_3'=i_4'=b$, $j_1=j_1'=j$, $j_2=j_2'=k$, $j_3=j_3'=p$, and $j_4=j_4'=q$. Since $\max_{\gamma \in S_n} \#(\gamma)=4$, the leading term in (\ref{wein_temp}), corresponding to $\mathrm{Wg}(N,\gamma)$ with $\gamma=\one$ ($\alpha^{-1}\beta=\one$), is of size $O(\frac{1}{N^4})$ from (\ref{wein_fun}) and the rest terms are bounded by $O(\frac{1}{N^5})$. Moreover, the coefficient in front of $\mathrm{Wg}(N,\one)$ is given by the number of permutations $\sigma \in S_4$ such that
\begin{align}\label{permu_relation}
i_l=i'_{\sigma(l)}, \qquad j_l=j'_{\sigma(l)}, \qquad l=1,2,3,4.
\end{align}
 We then separate into the following five cases: 1.) all indices $j,k,p,q$ are distinct, 2.) only two of them coincide while the other two are distinct, 3.) two pairs of them coincide, 4.) three of them coincide and the rest one is different, and 5.) all the indices are the same. As $a \neq b$, the number of permutations satisfying (\ref{permu_relation}) is given by $1$, $8$, $6$, $8$ and $4$, respectively. Therefore, for $a \neq b$, we obtain 
\begin{align}\label{le new weingarten term}
\E  [(G_{aa}(z))^2& (G_{bb}(z))^2]=\frac{1}{N^4} \sum_{\substack{j,k,p,q\\\textrm{all distinct}}}  \E \Big[ \frac{1}{(\lambda_{j}-z)(\lambda_{k}-z)(\lambda_{p}-z)(\lambda_{q}-z)}\Big] \Big(1+O\big(\frac{1}{N}\big)\Big)\nonumber\\
& +\frac{8}{N^4} \sum_{\substack{j,p,q\\\textrm{all distinct}}}  \E \Big[\frac{1}{(\lambda_{j}-z)^2(\lambda_{p}-z)(\lambda_{q}-z)}\Big]\Big(1+O\big(\frac{1}{N}\big)\Big)\nonumber\\
&+ \frac{6}{N^4} \sum_{j \neq q} \E \Big[ \frac{1}{(\lambda_{j}-z)^2(\lambda_{q}-z)^2}\Big]\Big(1+O\big(\frac{1}{N}\big)\Big)\nonumber\\
&+\frac{8}{N^4} \sum_{j \neq q} \E \Big[ \frac{1}{(\lambda_{j}-z)^3(\lambda_{q}-z)}\Big]\Big(1+O\big(\frac{1}{N}\big)\Big)+\frac{4}{N^4} \sum_{j} \E \Big[ \frac{1}{(\lambda_{j}-z)^4}\Big]\Big(1+O\big(\frac{1}{N}\big)\Big).
\end{align}

For example, by direct computation, the first term on the right side of (\ref{le new weingarten term}) can be written using the spectral decomposition~\eqref{le spectral decomp} as
\begin{align}
\frac{1}{N^4} \sum_{\substack{j,k,p,q\\\textrm{all distinct}}}  \E& \Big[ \frac{1}{(\lambda_{j}-z)(\lambda_{k}-z)(\lambda_{p}-z)(\lambda_{q}-z)}\Big] =\frac{1}{N^4}\E [(\Tr{G})^4] -\frac{6}{N^4}\E[(\Tr G^2)(\Tr{G})^2]\nonumber\\
&\qquad \qquad+\frac{8}{N^4}\E[(\Tr G^3)(\Tr{G})]-\frac{6}{N^4}\E [\Tr G^4]+\frac{3}{N^4}\E[(\Tr G^2)(\Tr{G^2})].
\end{align}
Observe that the resulting terms on the right side are type-0 terms under Definition \ref{def_type_AB}. 
We further write the other terms on the right side of~\eqref{le new weingarten term} similarly by type-0 terms using the spectral decomposition. To sum up, averaging over $a,b$ and adding the subleading diagonal terms, (\ref{sum_gue}) eventually becomes, 
\begin{align}
\frac{1}{N^2}\sum_{a,b} \E [(G_{aa})^2 (G_{bb})^2]=&\frac{1}{N^4}\E [(\Tr{G})^4] +\frac{2}{N^4}\E[(\Tr G^2)(\Tr{G})^2]+\frac{1}{N^4}\E[(\Tr G^2)(\Tr{G^2})]+O(N^{-1}),\nonumber
\end{align}
uniformly in $z\in S$, after exact cancellations between the terms.

In this way, we have eliminated one pair of $a$-indices and $b$-indices from the type-AB term (\ref{example_ab}) and shown that they can be written as linear combinations of type-0 terms, which involves only products of traces. 

For Wigner matrices, the above does not apply anymore as the eigenvectors are no longer exactly Haar distributed on $U(N)$, further the expectation in~\eqref{sum_gue} does not factorize. Yet successively applying cumulant expansions, we can reduce type-AB terms to sums of type-A terms up to negligible error, and then finally reduce type-A terms to sums of type-0 terms. This procedure is explained in the next subsection.

\subsection{Proof of Proposition \ref{lemma_expand_type}}\label{toy_fourth}
In this subsection, we give the proof of Proposition \ref{lemma_expand_type} for arbitrary Wigner matrices using cumulant expansions.
\begin{proof}[Proof of Proposition \ref{lemma_expand_type}]
We consider a type-AB term of the form in (\ref{form}) with both indices $a$ and $b$ appearing twice as a row index and twice as a column index in the product of the Green function entries. There are two steps as follows. We first expand the type-AB term as a linear combination of type-A terms by eliminating one pair of the index~$b$. Then in a second step we expand the resulting type-A terms as linear combinations of type-0 terms by further eliminating a pair of the index $a$.

{\bf Step 1: Reduction to type-A terms.}
Given a type-AB term, we will eliminate one pair of the index $b$ using the relation
\begin{align}\label{resolvent_identity}
G_{ij}=\delta_{ij} \underline{G}+G_{ij} \underline{HG}-\underline{G}(HG)_{ij} ,
\end{align}
and then applying cumulant expansions. The identity may be checked directly from the definition of the Green function. In~\eqref{resolvent_identity} we use the notation $\ud{A}:=\frac{1}{N}\Tr A$, for any $A \in \C^{N \times N}$, to denote the normalized trace. Similar ideas were used in~\cite{He+Knowles,sparse}.

Consider now a type-AB term $P_d^{AB} \in \mathcal{P}_d^{AB}$ of the form in (\ref{form}). We split into the following two cases.

{\bf Case 1:} If there exists some $i$ such that $x_i=y_i=b$, \ie there is a factor $G_{bb}$ in the product of Green function entries, we may then assume $i=1$. Applying (\ref{resolvent_identity}) to $G_{bb}$ and performing cumulant expansions for the resulting terms $\underline{HG}$ and $(HG)_{bb}$, we obtain
\begin{align}
\E [ P_d^{AB}]=&\frac{1}{N^{\# \mathcal I+2}} \sum_{\mathcal{I},a,b} c_{a,b, \mathcal{I}} \E\Big[ (\underline{G}+G_{bb} \underline{HG}-\underline{G} (HG)_{bb} )\prod_{2 \leq i \leq n} G_{x_{i} y_{i}}\Big]\nonumber\\
=&\frac{1}{N^{\# \mathcal I+2}} \sum_{\mathcal{I},a,b}  c_{a,b, \mathcal{I}}\E\Big[ \ud{G} \prod_{2 \leq i \leq n} G_{x_{i} y_{i}} \Big]+\frac{1}{N^{\# \mathcal I+4}} \sum_{\mathcal{I},a,b,j,k}  c_{a,b, \mathcal{I}} \E\Big[ \frac{\partial G_{bb} G_{jk} \prod_{2 \leq i \leq n} G_{x_{i} y_{i}} }{\partial h_{jk}}\Big]\nonumber\\
&-\frac{1}{N^{\# \mathcal I+4}} \sum_{\mathcal{I},a,b,j,k}  c_{a,b, \mathcal{I}}\E\Big[ \frac{\partial  G_{jj} G_{kb} \prod_{2 \leq i \leq n} G_{x_{i} y_{i}} }{\partial h_{kb}}\Big]+O_{\prec}\big(\frac{1}{\sqrt{N}}\big)\,,
\end{align}
where the error $O_{\prec}(\frac{1}{\sqrt{N}})$ is from the truncation of the cumulant expansions. Using (\ref{dH}), the first order of the second group of terms above corresponding to $\frac{\partial}{\partial h_{jk}}G_{jk}$ is precisely canceled by that of the third group of terms corresponding to $\frac{\partial }{\partial h_{kb}}G_{kb}$. Then we write
\begin{align}\label{expand_G_bb}
\E [ P_d^{AB}]=&\frac{1}{N^{\# \mathcal I+2}} \sum_{\mathcal{I},a,b}  c_{a,b, \mathcal{I}} \E\Big[\ud{G} \prod_{2 \leq i \leq n} G_{x_{i} y_{i}} \Big]-\frac{1}{N^{\# \mathcal I+4}} \sum_{\mathcal{I},a,b,j,k}  c_{a,b, \mathcal{I}} \E\Big[ \frac{\partial  G_{bb} \prod_{2 \leq i \leq n} G_{x_{i} y_{i}} }{\partial h_{jk}}G_{jk}  \Big]\nonumber\\
&+\frac{1}{N^{\# \mathcal I+4}} \sum_{\mathcal{I},a,b,j,k}  c_{a,b, \mathcal{I}}\E\Big[ \frac{\partial G_{jj}  \prod_{2 \leq i \leq n} G_{x_{i} y_{i}} }{\partial h_{kb}}G_{kb}\Big]+O_{\prec}(\frac{1}{\sqrt{N}}).
\end{align}

The first term on the right side above is obtained by replacing $G_{bb}$ by the normalized trace $\ud G$ in the expression of~$P_d^{AB}$. In this way we have eliminated one pair of the index $b$. Since the index $b$ originally appeared twice as a row index and twice as a column index in the product of the Green function entries, the first term has become a type-A term of degree $d$. Moreover, from~(\ref{dH}) and the fact that $j,k$ are fresh indices, the other terms on the right side of (\ref{expand_G_bb}) can be written out as a sum of $2n$ type-AB terms of the form in (\ref{form}), where the corresponding free summation index set is $\mathcal{I}'=\{\mathcal{I},j,k\}$, $m'=\#\mathcal{I}'=m+2$, and the number of Green function entries is $n'=n+2$. 

We next study the degrees of these terms in detail. In the second group of summation in (\ref{expand_G_bb}), if $\partial/\partial h_{jk}$ acts on $G_{bb}$, then the degree of the resulting term is increased by three, since $j$ and $k$ are fresh indices. If $\partial/\partial h_{jk}$ acts on $G_{x_i y_i}~(2\leq i\leq n)$, then the degree is increased by at least two for the same reason. In the last group of summation in (\ref{expand_G_bb}), if $\partial/\partial h_{kb}$ acts on $G_{jj}$, then the degree is increase by three. When $\partial/\partial h_{jk}$ acts on $G_{x_i y_i}~(2\leq i\leq n)$, we split the discussion into three cases: 1) if $G_{x_i y_i}$ is diagonal and $x_i=y_i \neq b$, then the degree is increased by three; 2) if $G_{x_i y_i}$ is off-diagonal with $y_i \neq b$, then the degree is increased by two; 3) if $G_{x_i y_i}$ is off-diagonal with $y_i= b$, then the degree is increased by one.

Hence the degrees, denoted by $d'$, of all the terms on the right side of (\ref{expand_G_bb}) except the first one, satisfy $d' \geq d+1$. We use $\sum_{\substack{P_{d'}^{AB} \in \mathcal{P}^{AB}_{d'}, d' \geq d+1}} \E[P_{d'}^{AB}]$ to denote the finite sum of these terms, \ie we write
\begin{align}\label{p1}
\E [ P_d^{AB}]=\frac{1}{N^{\# \mathcal I+2}} \sum_{\mathcal{I},a,b}  c_{a,b, \mathcal{I}} \E\Big[\ud{G} \prod^{n}_{i=2} G_{x_{i} y_{i}} \Big]+\sum_{\substack{P_{d'}^{AB} \in \mathcal{P}^{AB}_{d'}\\ d' \geq d+1}} \E [  P^{AB}_{d'}]+O_{\prec}\big(\frac{1}{\sqrt{N}}\big)\,.
\end{align}
Therefore, the combination of the identity~\eqref{resolvent_identity} and the cumulant expansion gives a cancellation to first order, and the only leading term left is obtained by replacing a factor $G_{bb}$ with the normalized trace $\ud G$ of the product of Green function entries in the expression of the original $P_d^{AB}$.

{\bf Case 2: } If there is no $i$ such that $x_i=y_i=b$, \ie there is no factor as $G_{bb}$ in the product of Green function entries in (\ref{form}), we may then assume that $x_1=b$ and $y_1 \neq b$. Since the index $b$ appears exactly twice in $\{y_i\}_{i=2}^n$, we may assume that $y_2=y_3=b$ and $x_2 \neq b$ and $x_3 \neq b$. Then there is no $b$ in the remaining column index set $\{y_i\}_{i=4}^n$. Using the identity (\ref{resolvent_identity}) on $G_{by_1}$ and applying cumulant expansions, we find
\begin{align}\label{expand_G_bb1}
\E [ P_d^{AB}]=&\frac{1}{N^{\# \mathcal I+2}} \sum_{\mathcal{I},a,b} c_{a,b, \mathcal{I}} \E \Big[ ( G_{by_1} \underline{HG}-\underline{G} (HG)_{by_1} )G_{x_2b} G_{x_3 b} \prod^{n}_{i=4} G_{x_{i} y_{i}}\Big]\nonumber\\
&+\frac{1}{N^{\# \mathcal I+2}} \sum_{\mathcal{I},a,b} c_{a,b, \mathcal{I}} \E \Big[ \delta_{by_1} \ud{G}  \prod^{n}_{i=2} G_{x_{i} y_{i}}\Big]\nonumber\\
=&-\frac{1}{N^{\# \mathcal I+4}} \sum_{\mathcal{I},a,b,j,k}  c_{a,b, \mathcal{I}} \E\Big[ \frac{\partial  G_{by_1} G_{x_2b} G_{x_3 b} \prod^{n}_{i=4} G_{x_{i} y_{i}} }{\partial h_{jk}}G_{jk}  \Big]\nonumber\\
&+\frac{1}{N^{\# \mathcal I+4}} \sum_{\mathcal{I},a,b,j,k}  c_{a,b, \mathcal{I}}\E\Big[ \frac{\partial G_{jj}  G_{x_2b} G_{x_3 b}  \prod^{n}_{i=4} G_{x_{i} y_{i}} }{\partial h_{kb}}G_{ky_1}\Big]+O_{\prec}(\frac{1}{\sqrt{N}}),
\end{align}
where in the second step, we observe a cancellation to first order similarly as in (\ref{expand_G_bb}), and the last error $O_{\prec}(N^{-1/2})$ is from the truncation of the cumulant expansions at the third order, while the contribution from the diagonal case $b \equiv y_1$, \ie the second line of (\ref{expand_G_bb1}), can be bounded by $O_{\prec}(N^{-1})$ using the local law in (\ref{G}). From (\ref{dH}), the right side of~\eqref{expand_G_bb1} can again be written as a sum of $2n$ type-AB terms of the form in (\ref{form}) with $\mathcal{I}'=\{\mathcal{I}, j,k\}$, $m'=m+2$, and $n'=n+2$. Since $j,k$ are fresh indices, the resulting type-AB terms have degrees $d' \geq d+1$ (the finite sum of such terms is denoted by $\sum_{{P_{d'}^{AB} \in \mathcal{P}^{AB}_{d'},d' \geq d+1}} \E [  P^{AB}_{d'}]$), except the following two terms corresponding to taking $\frac{\partial }{\partial h_{kb}}$ of a Green function entry whose column index coincides with $b$, \ie
\begin{align}\label{le x1}
\frac{1}{N^{\# \mathcal I+4}} \sum_{\mathcal{I},a,b,j,k}  c_{a,b, \mathcal{I}}\E\Big[ G_{jj} G_{x_2k} G_{bb} G_{x_3 b}  \prod^{n}_{i=4} G_{x_{i} y_{i}} G_{ky_1}\Big]
\end{align}
and 
\begin{align}\label{le x2}
\frac{1}{N^{\# \mathcal I+4}} \sum_{\mathcal{I},a,b,j,k}  c_{a,b, \mathcal{I}}\E\Big[ G_{jj}  G_{x_2 b}  G_{x_3k} G_{bb} \prod^{n}_{i=4} G_{x_{i} y_{i}} G_{ky_1}\Big].
\end{align}
Compared with the original term $P^{AB}_d$, one observes that the terms in~\eqref{le x1} and~\eqref{le x2} are obtained by replacing one pair of the index $b$ by a fresh index $k$ and adding a factor $G_{bb}$ for the replaced index $b$. These terms are again type-AB terms in $\mathcal{P}_d^{AB}$ with a factor $G_{bb}$ in the product of Green function entries considered in Case 1. Using (\ref{p1}) on these terms and combining with (\ref{expand_G_bb1}), we hence obtain
\begin{align}\label{le x3}
\E[P^{AB}_{d}]=&\frac{1}{N^{\# \mathcal I+4}} \sum_{\mathcal{I},a,b,j,k}  c_{a,b, \mathcal{I}}\E\Big[ G_{jj} \ud{G} G_{ky_1} G_{x_2k}  G_{x_3 b}  \prod^{n}_{i=4} G_{x_{i} y_{i}} \Big]\nonumber\\
&+\frac{1}{N^{\# \mathcal I+4}} \sum_{\mathcal{I},a,b,j,k}  c_{a,b, \mathcal{I}}\E\Big[ G_{jj} \ud{G}  G_{ky_1} G_{x_2 b}  G_{x_3k} \prod^{n}_{i=4} G_{x_{i} y_{i}} \Big]\nonumber\\
&+\sum_{\substack{P_{d'}^{AB} \in \mathcal{P}^{AB}_{d'}\\ d' \geq d+1}} \E [  P^{AB}_{d'}]+\sum_{\substack{P_{d''}^{AB} \in \mathcal{P}^{AB}_{d''}\\ d'' \geq d+1}} \E [  P^{AB}_{d''}]+O_{\prec}\big(\frac{1}{\sqrt{N}}\big)\,,
\end{align}
where the first two lines above are type-A terms in $\mathcal{P}^{A}_d$, obtained from the original term $P^{AB}_d$ by replacing a pair of index~$b$, \ie $(x_1,y_2)$ or $(x_1,y_3)$ by a fresh index $k$ and multiplied by $(\ud{G})^2$. The first group of sum on the last line of (\ref{le x3}) comes from (\ref{expand_G_bb1}) excluding two terms (\ref{le x1}) and (\ref{le x2}), and the number of the type-AB terms in the sum is at most $2n-2$. The second group of sum on the last line of (\ref{le x3}) is obtained from expanding (\ref{le x1}) and (\ref{le x2}) by (\ref{p1}). The corresponding type-AB terms are of the form in (\ref{form}) with $m''=m'+2$ and $n'' = n'+2$, and the number of the terms in the sum is at most $4n'$.

Combining with Case 1, for any type-AB term $P_d^{AB} \in \mathcal{P}_d^{AB}$, we rewrite~ (\ref{p1}) and \eqref{le x3} in the short form
\begin{align}\label{step_1}
\E[P_d^{AB}]=\sum_{\substack{P^{A}_{d} \in \mathcal{P}^{A}_{d}}} \E[P^{A}_{d}]+\sum_{\substack{P_{d'}^{AB} \in \mathcal{P}^{AB}_{d'}\\ d' \geq d+1}} \E [  P^{AB}_{d'}]+O_{\prec}\big(\frac{1}{\sqrt{N}}\big)\,,
\end{align}
where the summations above denote a sum of at most two type-A terms of degree $d$ and a sum of at most $(6n+8)$ type-AB terms of degree not less than $d+1$. The number of the Green function entries in the product (see (\ref{form})) of each term is at most $n+4$.

\begin{remark}\label{moreab}
In general, if the number of the index $b$ appearing in the Green function entries of $P^{AB}_d$ is not limited to four, \ie $\nu_{b}^{(r)}=\nu_{b}^{(c)} =s \geq 3$, then the terms in the first group of sum on the right side of~(\ref{step_1}) are of the form in~(\ref{form}) with~$\nu_{b}^{(r)}=\nu_{b}^{(c)}=s-1 \geq 2$. Moreover, the number of such terms in the first group of the sum is at most $s$. We can repeat the expansion procedure in (\ref{step_1}) for $s$ times until $\nu_{b}^{(r)}=\nu_{b}^{(c)}=1$. We then end up with at most $s!$ type-A terms in $\mathcal{P}^A_d$, and at most $6s^s(n+4s)$ type-AB terms of degrees not less than $d+1$ generated in the above expansion procedures.

\end{remark}
 
Iterating the expansion procedure (\ref{step_1}) $D-d$ times, the resulting type-AB terms have degrees at least $D$. Using the local law in (\ref{G}), we expand an arbitrary type-AB term $P^{AB}_d\in \mathcal{P}^{AB}_d$ as a finite sum of type-A terms of degrees at least~$d$, up to negligible error. We hence arrive at
\begin{align}\label{step11}
\E [ P_d^{AB}]=\sum_{d \leq d' < D} \sum_{P^{A}_{d'} \in \mathcal{P}^A_{d'}} \E [ P^{A}_{d'}]+O_{\prec}\big(\frac{1}{\sqrt{N}}+\Psi^D\big)\,.
\end{align}
The number of the Green function entries in the product of each type-A term above is bounded by $n+4D$, and the number of these type-A terms is bounded by $ (6(n+4D))^D$.

{\bf Step 2: Reduction to type-0 terms.} For the expanded type-A terms on the right side of (\ref{step11}), we follow the idea in Step 1 to expand the resulting type-A terms as linear combinations of type-0 terms by further eliminating one pair of the index $a$. 

Given a type-A term $P_d^{A} \in \mathcal{P}_d^{A}$ of the form in (\ref{form}), we split into two cases: 1) there exists a factor $G_{aa}$ in the product of Green function entries; 2) there is no factor $G_{aa}$ in the product of the Green function entries. We utilize similar arguments as in  Case~1 and Case~2 of Step~1 above and obtain the analogue of (\ref{step_1}), namely that
\begin{align}\label{step_2}
\E[P_d^{A}]=\sum_{\substack{P_{d} \in \mathcal{P}_{d}}} P_{d}+\sum_{\substack{P_{d''}^{A} \in \mathcal{P}^{A}_{d''}\\ d'' \geq d+1}}\E [  P^{A}_{d''}]+O_{\prec}\big(\frac{1}{\sqrt{N}}\big)\,,
\end{align}
where the summations above denote a sum of at most two type-0 terms of degree $d$ and a sum of at most $(6n+8)$ type-A terms of degrees at least $d+1$. The number of the Green function entries in the product of each term is bound by $n+4$.

Iterating the above expansion for $D-d$ times, we then expand an arbitrary type-A term $P_d^{A} \in \mathcal{P}_d^{A} $ as a sum of at most $(6(n+4D))^D$ type-0 terms of degree $d'$ satisfying $d \leq d' < D$, up to negligible error. As the analogue of (\ref{step11}), we write
\begin{align}\label{step_22}
\E [  P_d^{A}]=\sum_{d \leq d' < D} \sum_{P_{d'} \in \mathcal{P}_{d'}}  \E [ P_{d'}]+O_{\prec}\big(\frac{1}{\sqrt{N}}+\Psi^D\big)\,,
\end{align}
where the number of the Green function entries in the product of each type-0 term above is bounded by $n+4D$.

Combining with the first step (\ref{step11}), we finish the proof of Proposition \ref{lemma_expand_type}.
\end{proof}

 \begin{remark}
 The expansion procedures in the proof of Proposition~\ref{lemma_expand_type} are not unique in the sense that at each step the resulting expansion will depend on the choice of the Green function entry we pick to be replaced by the identity (\ref{resolvent_identity}) and perform cumulant expansions. However in view of Lemma~\ref{lemma_trace_wigner} this is not pertinent to the proof of Proposition~\ref{GCT_mn}. This arbitrariness can be used to derive relations among Green function correlation functions. We finally remark that the errors $O_{\prec}(N^{-1/2})$ in~(\ref{step11}) and~(\ref{step_22}) stemming from truncating the cumulant expansions at second order, can be improved to $O_{\prec}(N^{-1})$ because of the negligibility of third order terms; see Proposition \ref{unmatch_lemma}.
 \end{remark}

\subsection{Proof of Lemma \ref{lemma_trace_wigner}}\label{toy_type_zero}
In the subsection, we estimate the expectations of the type-0 terms and prove Lemma~\ref{lemma_trace_wigner}. We start with the following lemma for the GUE. 
\begin{lemma}\label{lemma_trace} Let $H$ belong to the GUE. For any $\epsilon>0$ and $C_0>0$, recall the domain $S_{\mathrm{edge}}\equiv S_{\mathrm{edge}}(\epsilon,C_0)$ defined in~\eqref{S_edge}. Then there exists a constant $C$ independent of $\epsilon$ such that
\begin{equation}\label{traceimg}
\frac{1}{N} \E^{\mathrm{GUE}} \Big[ \Im \Tr G (z)\Big] \leq C N^{-1/3+\epsilon}\,,
\end{equation}
holds uniformly for all $z \in S_{\mathrm{edge}}$, for sufficiently large $N \geq N_0(C_0, \epsilon)$. Furthermore, for any $\tau>0$, all the type-0 terms $P_d \in \mathcal{P}_d$ ($d \geq 2$) of the form in (\ref{form})  have the upper bound
\begin{align}\label{tracek}
|\E^{\mathrm{GUE}} [P_d(z)]| \leq N^{-1/3+\tau}\,,
\end{align}
uniformly for all $z\in S_{\mathrm{edge}}$, for sufficiently large $N \geq N'_0(C_0, \epsilon, \tau)$.
\end{lemma}
The proof of Lemma~\ref{lemma_trace} is postponed to Subsection~\ref{toy_type_zero_gue}. Using the above lemma for the GUE and the comparison method, we are now ready to prove Lemma \ref{lemma_trace_wigner} for arbitrary Wigner matrices.

\begin{proof}[Proof of Lemma \ref{lemma_trace_wigner}]
Consider any type-0 term $P_d \in \mathcal{P}_d$ of the form in (\ref{form}) of degree $d \geq 2$.  If $d \geq D$ for some large $D$, then by the local law in (\ref{G}), $|\E[P_d]|=O_{\prec}(\Psi^D+N^{-1})$. Else, if $d$ is smaller, we estimate $\E[P_d]$ using the comparison method iteratively and the corresponding estimates for the GUE in~(\ref{tracek}).

We start the iteration by denoting the type-0 term $P_d$ of the form in (\ref{form}) as $P_d\equiv P^{(1)}_{d_1}$, where the superscript $(1)$ and degree $d\equiv d_1$ will be used to indicate the iteration step. We hence consider a term of the form
\begin{align}\label{P_d_s_1}
P^{(1)}_{d_1} \equiv P^{(1)}_{d_1} (t,z): \qquad  \frac{1}{N^{\# \mathcal I_1+2}} \sum_{\mathcal{I}_1,a_1,b_1} c_{a_1,b_1,\mathcal{I}_1} \Big(  \prod^{n_1}_{i=1} G_{x_{i} y_{i}} (t,z)\Big)\,,\qquad t\in \R^{+},~ z\in S_{\mathrm{edge}}\,,
\end{align}
with $n_1=\# \mathcal I_1+2$, where each summation index in $\{a_1,b_1,\mathcal{I}_1\}$ appears exactly once in the row index set $\{x_i\}$ and exactly once in the column index set $\{y_{i}\}$. In the following, we often omit the parameters $t,z$ and the errors below are always bounded uniformly in $z \in S_{\mathrm{edge}}$ and $t \geq 0$. 

We next derive the stochastic differential equation for the type-0 term $P^{(1)}_{d_1}$ under the Ornstein--Uhlenbeck flow in~\eqref{ou_process}, similarly to~(\ref{diff_eq_1}). In general, for any $\{x_i,y_i\}_{i=1}^n$ with some $n \in \N$, using Ito's formula and the stochastic differential equation for the Green function entries in~(\ref{derivative2}), we have
\begin{align}\label{general_diff}
\dd \Big( \prod_{i=1}^n G_{x_i y_i}\Big)=&\sum_{j=1}^n \prod_{i \neq j} G_{x_i y_i} \dd G_{x_j y_j}+\frac{1}{2}\sum_{j,k=1}^n \prod_{i \neq j,k} G_{x_i y_i} \dd G_{x_j y_j}\dd G_{x_k y_k}\nonumber\\
=&- \frac{1}{\sqrt{N}} \sum_{a, b=1}^N  \sum_{j=1}^n  G_{x_j a}G_{by_j}   \prod_{i \neq j} G_{x_i y_i} \dd \beta_{ab}\nonumber\\
 &+ \frac{1}{2} \sum_{a, b=1}^N  \sum_{j=1}^n \Big( h_{ab} G_{x_j a} G_{by_j}+\frac{1}{N}  G_{x_j b} G_{by_j} G_{aa}+\frac{1}{N}  G_{x_ja} G_{ay_j} G_{bb} \Big) \prod_{i \neq j} G_{x_i y_i} \dd t \nonumber\\
 &+\frac{1}{2N}  \sum_{a, b=1}^N  \sum_{j,k=1}^n G_{x_j a}G_{b y_j} G_{x_k b}G_{a y_k} \prod_{i \neq j,k} G_{x_i y_i}\dd t:=\dd \widehat M +\widehat \Theta \,\dd t\,,
\end{align}
with diffusion term $\dd\widehat M$ and drift term $\widehat \Theta\,\dd t$. Applying cumulant expansions to the drift term, we observe cancellations of the second order expansions as in (\ref{step0}) and obtain that
\begin{align}\label{rhs}
\E [\widehat \Theta]=&\frac{1}{2N} \sum_{a=1}^N   (s^{(2)}_{aa}-1) \sum_{j=1}^n  \E \Big[ \frac{\partial (G_{x_j a} G_{ay_j}  \prod_{i \neq j} G_{x_i y_i})}{\partial h_{aa} }\Big]\nonumber\\
&+  \frac{1}{2}\sum_{\substack{a,b=1\\a \neq b}}^N \sum_{p+q+1=3}^4 \frac{s^{(p,q+1)}_{ab}}{p! q! N^{\frac{p+q+1}{2}}}  \sum_{j=1}^n  \E \Big[ \frac{\partial^{p+q} (G_{x_j a} G_{by_j}  \prod_{i \neq j} G_{x_i y_i})}{\partial h^p_{ba} \partial h^q_{ab}}\Big]+O_{\prec}(\frac{1}{\sqrt{N}})\nonumber\\
=&-\frac{1}{2N} \sum_{a=1}^N   (s^{(2)}_{aa}-1) \E \Big[ \frac{\partial^2 (\prod_{i=1}^n G_{x_i y_i})}{\partial h^2_{aa}}\Big] - \sum_{p+q+1=3}^4 \frac{1}{2 p! q! N^{\frac{p+q+1}{2}}}  \sum_{\substack{a,b=1\\ a \neq b}}^N s^{(p,q+1)}_{ab}  \E \Big[ \frac{\partial^{p+q+1} (\prod_{i=1}^n G_{x_i y_i})}{\partial h^p_{ba} \partial h^{q+1}_{ab}}\Big]\nonumber\\ &\qquad\qquad+O_{\prec}(\frac{1}{\sqrt{N}}).
\end{align}
From (\ref{general_diff}) and (\ref{rhs}), we find that $P^{(1)}_{d_1}$ in (\ref{P_d_s_1}) satisfies the stochastic differential equation
\begin{equation}\label{late2}
\dd (P^{(1)}_{d_1})=\dd M^{(1)}_{d_1}+\Theta^{(1)}_{d_1} \dd t,
\end{equation}
where the diffusion term $\dd M^{(1)}_{d_1}$ yields a martingale after integration (see Remark \ref{diffusion_mart}) and the drift term $\Theta^{(1)}_{d_1} \dd t$ satisfies the following analogue of~(\ref{step0}),
\begin{align}\label{late_0}
\E[\Theta^{(1)}_{d_1}]=&-\frac{1}{2N} \sum_{a_2=1}^N ( s^{(2)}_{a_2a_2}-1) \E \Big[ \frac{ \partial^2  (P^{(1)}_{d_1}) }{\partial h^2_{a_2a_2} } \Big]- \sum_{p+q+1=3}^{4}\frac{1}{2 p! q! N^{\frac{p+q+1}{2}}}  \sum_{\substack{a_2,b_2=1\\a_2 \neq b_2}}^N   s^{(p,q+1)}_{a_2b_2} \E \Big[ \frac{ \partial^{p+q+1} (P^{(1)}_{d_1}) }{\partial h^p_{b_2a_2} \partial h^{q+1}_{a_2b_2}} \Big]\nonumber\\ &\qquad +O_{\prec}\big(\frac{1}{\sqrt{N}}\big)\,,
\end{align}
where $a_2,b_2$ are fresh summation indices, as $a,b$ in (\ref{rhs}). The subscript $2$ is used to indicate the iteration step and distinguish from $a_1,b_1$ in (\ref{P_d_s_1}). 

From (\ref{dH}), all the third order terms for $p+q+1=3$ in the cumulant expansion above can be written out in the form in (\ref{product}), with an extra factor $\sqrt{N}$ in front. Since the fresh indices $a_2,b_2$ both appear an odd number of times in the product of the Green function entries, they are unmatched from Definition~\ref{unmatch_def}. Using Proposition~\ref{unmatch_lemma}, these term are bounded by $O_{\prec}(N^{-1/2}+\sqrt{N} \Psi^D)$.

The fourth order terms in the cumulant expansion with $p+q+1=4$ in~\eqref{late_0}, with the exception of those corresponding to $p=2,q=1$, are also unmatched terms of the form in (\ref{product}), since the number of times the index $a_2~( \mbox{or } b_2)$ appears in the row index set $\{x_i\}$ does not agree with the number of times it appears in the column index set $\{y_i\}$. Using Proposition \ref{unmatch_lemma}, these term are bounded by $O_{\prec}(N^{-1}+\Psi^D)$. 

By choosing $D \geq \frac{1}{\epsilon}$ with $\epsilon$ as in (\ref{control_eq}), we hence obtain the following analogue of (\ref{step})
\begin{align}\label{late}
\E[\Theta^{(1)}_{d_1}]=&-\frac{1}{2N} \sum_{a_2=1}^N (s^{(2)}_{a_2a_2}-1) \E \Big[ \frac{ \partial^2 (P^{(1)}_{d_1}) }{\partial h^2_{a_2a_2} } \Big]-\frac{1}{4N^2} \sum_{\substack{a_2,b_2=1\\a_2\neq b_2}}^N s^{(2,2)}_{a_2b_2} \E \Big[ \frac{ \partial^4 (P^{(1)}_{d_1}) }{\partial h^2_{b_2a_2} \partial h^2_{a_2b_2} } \Big]+O_{\prec}(N^{-1/2})\,.
\end{align}
It then suffices to estimate the remaining matched terms above. Using (\ref{dH}) and (\ref{P_d_s_1}), the second group of terms on the right side of~\eqref{late} can be written out in the form:
\begin{align}\label{type_ab_p_s}
  \frac{1}{N^{\# \mathcal I_1+4}} \sum_{\mathcal{I}_1,a_1,b_1, a_2,b_2} c_{a_1,b_1,a_2,b_2,\mathcal{I}_1}  \Big( \prod^{n_1+4}_{i=1} G_{x_{i} y_{i}}\Big)\,,
\end{align}
where the coefficients $\{c_{a_1,b_1,a_2,b_2,\mathcal{I}_1}\}$ are determined by $\{c_{a_1,b_1,\mathcal{I}_1}\}$ and $\{s^{(2,2)}_{a_2,b_2}\}$, and each summation index in $\{a_1,b_1,\mathcal{I}_1\}$ appears once in the row index set $\{x_i\}$ and once in the column index set $\{y_{i}\}$. Moreover, both indices $a_2,b_2$ appear exactly twice in the row index set $\{x_i\}$ and exactly twice in the column index set $\{y_i\}$. We define the degree of the form in (\ref{type_ab_p_s}) as in (\ref{degree_0}) by counting the number of off-diagonal Green function entries. Recall the definition of the type-AB, type-A and type-0 terms from Definition \ref{def_type_AB}. The definitions can be adapted naturally with respect to the fresh indices $a_2$ and $b_2$, for the form given in (\ref{type_ab_p_s}). 

Thus the second group of terms on the right side of (\ref{late}) are $n_1(n_1+1)(n_1+2)(n_1+3)$ type-AB terms considered in Proposition \ref{lemma_expand_type} of degrees not less than $d_1 +1$, from (\ref{dH}) and the fact that $a_2,b_2$ are fresh indices. Similarly, the first group of terms on the right side of~\eqref{late} are $n_1(n_1+1)$ type-A terms of degrees not less than $d_1 +1$. Using Proposition \ref{lemma_expand_type}, we expand each of these terms as a sum of finitely many type-0 terms of degrees at least $d_1+1$, which are in the form:
\begin{align}\label{general_pd}
\mathcal{P}^{(2)}_{d_2}: \qquad \frac{1}{N^{\# \mathcal I_2+4}} \sum_{\mathcal{I}_2,a_1,b_1, a_2,b_2} c_{a_1,b_1,a_2,b_2,\mathcal{I}_2}  \Big(  \prod^{n_2}_{i=1} G_{x_{i} y_{i}} \Big)\,,
 \end{align}
where $\mathcal{I}_2$ is a set of free summation indices, the coefficients $\{c_{a_1,b_1,a_2,b_2,\mathcal{I}_2}\}$ are uniformly bounded complex numbers, and each index in $\{a_2,b_2, a_1,b_1, \mathcal{I}_2\}$ appears once in $\{x_i\}$ and once in $\{y_{i}\}$. In particular, $n_2=\# \mathcal I_2+4$.  The degree of such a term, denoted by $d_2$, is given as in (\ref{degree_0}).  The collection of the type-0 terms of the form in (\ref{general_pd}) of degree $d_2$ is denoted by $\mathcal{P}^{(2)}_{d_2}$.  Here we use the subscript~$2$ to indicate the iteration step. Note that the form in (\ref{general_pd}) is a special case of the form given in (\ref{product}) and the indices $a_1, b_1,a_2,b_2$ do not take special roles. We keep them in the notation to emphasize the inheritance from (\ref{type_ab_p_s}).  Then from Proposition \ref{lemma_expand_type}, we expand (\ref{late}) and write for short
\begin{align}\label{late3}
\E[\Theta^{(1)}_{d_1}]= \sum_{\substack{ P^{(2)}_{d_{2}} \in \mathcal{P}^{(2)}_{d_{2}}\\d_1+1\leq d_{2} < D}} \E[ P^{(2)}_{d_{2}}]+O_{\prec}\big(N^{-1/2}+\Psi^D\big)\,,
\end{align}
uniformly in $t \geq 0$ and $z \in S_{\mathrm{edge}}$, where the summation above is over finitely many type-0 terms of the form in (\ref{general_pd}), and the number of these terms is determined by $D$ and $n_1$. 

We now return to the stochastic differential equation for $P^{(1)}_{d_{1}}$ in (\ref{late2}). Integrating (\ref{late2}) over $ [t', T]$ for any $0\leq t' \leq T=8 \log N$ and taking the expectation similarly to (\ref{sandwich3}), we find from (\ref{late3}) that
\begin{align}\label{previous_d_1}
\E[P^{(1)}_{d_{1}}(T,z)]-\E[P^{(1)}_{d_{1}}(t',z)]=&\sum_{ \substack{ P^{(2)}_{d_{2}} \in \mathcal{P}^{(2)}_{d_{2}} \\d_1+1\leq d_{2} < D}} \int_{t'}^{T} \E[ P^{(2)}_{d_{2}}(t,z)] \dd t +O_{\prec}\big(\log N (N^{-1/2}+\Psi^D)\big)\,.
\end{align}
Using the local law in (\ref{G}), (\ref{approxxxx}) and (\ref{tracek}), $\E[P^{(1)}_{d_1}(T,z)]$ is sufficiently close (up to an error $O(N^{-1})$) to $\E^{\mathrm{GUE}}[P^{(1)}_{d_1}(z)]$, which can be bounded by $O_{\prec}(N^{-1/3})$. Hence it suffices to estimate $\E[ P^{(2)}_{d_{2}}(t,z)]$ on the right side of (\ref{previous_d_1}), for $t \in [0,T]$, $z \in S_{\mathrm{edge}}$ in~(\ref{S_edge}).

Given any $P^{(2)}_{d_{2}} \in \mathcal{P}^{(2)}_{d_{2}}~(d_2 \geq d_1+1)$ of the form in (\ref{general_pd}), if $d_1=D-1$, we find $|\E[ P^{(2)}_{d_{2}}(t,z)]|=O_\prec(\Psi^D+N^{-1})$ using the local law in (\ref{G}). We then obtain from (\ref{previous_d_1}) that
\begin{align}\label{temp_step1}
|\E[P^{(1)}_{d_{1}}(t',z)]|=O_{\prec}\Big(\log N (N^{-1/2}+\Psi^D)+N^{-1/3}\Big),
\end{align}
uniformly in $t'\in [0,T]$ and $z \in S_{\mathrm{edge}}$.

Else, if $d_1 \leq D-2$, we repeat the above arguments for the resulting type-0 terms $P^{(2)}_{d_2} \in \mathcal{P}^{(2)}_{d_{2}}$ $(d_2 \geq d_1+1)$ on the right side of (\ref{previous_d_1}) as in (\ref{late2}). Using (\ref{general_diff}) and (\ref{rhs}), we then create two fresh summation indices, denoted by $a_3,b_3$, to derive the evolution under the Ornstein--Uhlenbeck flow of any $P^{(2)}_{d_2} \in \mathcal{P}^{(2)}_{d_2}$. Similarly as in (\ref{late}), the expectation of the corresponding drift terms is given by
 \begin{align}
\E[\Theta^{(2)}_{d_2}]=&-\frac{1}{2N} \sum_{a_3=1}^N (s^{(2)}_{a_3a_3}-1) \E \Big[ \frac{ \partial^2 (P^{(2)}_{d_2}) }{\partial h^2_{a_3a_3} } \Big]-\frac{1}{4N^2} \sum_{\substack{a_3,b_3=1\\ a_3 \neq b_3}}^N s^{(2,2)}_{a_3b_3} \E \Big[ \frac{ \partial^4 (P^{(2)}_{d_2}) }{\partial h^2_{b_3a_3} \partial h^2_{a_3b_3} } \Big]+O_{\prec}(N^{-1/2})\,.
\end{align}

From Definition \ref{def_type_AB}, the right side above can be written out as linear combinations of type-A terms and type-AB terms, with respect to fresh summation indices $a_3$ and $b_3$, of degrees not less than $d_2 +1$. Using Proposition \ref{lemma_expand_type}, these terms can further be expanded by the type-0 terms of degrees at least $d_2+1$. In this way, we obtain an estimate similar to (\ref{temp_step1}) for $d_1=D-2$.

Next, we discuss the iterative mechanism to extend to any small $d_1 \geq 2$. In general, for any $s \geq 1$, we define a type-0 term in the $s$-th iteration step to be in the form of
\begin{align}\label{P_ds}
\mathcal{P}^{(s)}_{d_s}: \qquad  \frac{1}{N^{\# \mathcal I_s+2s}} \sum_{\mathcal{I}_s, a_1,b_1, \ldots, a_s, b_s} c_{a_1,b_1, \ldots, a_s,b_s,\mathcal{I}_s} \E \Big[  \prod^{n_s}_{i=1} G_{x_{i} y_{i}} (t,z)\Big],
 \end{align}
 where $\mathcal{I}_s$ is a set of free summation indices, the coefficients $\{c_{a_1,b_1, \ldots, a_s,b_s,\mathcal{I}_s}\}$ are uniformly bounded complex numbers, and each free summation index in $\{a_1,b_1, \ldots, a_s, b_s, \mathcal{I}_s\}$ appears once in $\{x_i\}$ and once in $\{y_{i}\}$. In particular, we have $n_s=\# \mathcal I_s+2s$. The degree, denoted by $d_s$, of such a term in (\ref{P_ds}) is given as in (\ref{degree_0}) by counting the number of off-diagonal Green function entries. We denote by $\mathcal{P}^{(s)}_{d_s}$ the collection of the type-0 terms in the $s$-th step of the form in (\ref{P_ds}) of degree $d_s$. Note that the form in (\ref{P_ds}) is a special case of the form given in (\ref{product}), in order to emphasize the $s$-th iteration step and the dependence on $\{a_s, b_s\}$.

We then derive the stochastic evolution for any $P^{(s)}_{d_s} \in \mathcal{P}^{(s)}_{d_s}$ $(s\geq 1)$, using (\ref{general_diff}) and (\ref{rhs}) similarly as in (\ref{late2}) and (\ref{late3}). That is, 
 \begin{equation}\label{late22}
\dd (P^{(s)}_{d_s})=\dd M^{(s)}_{d_s}+\Theta^{(s)}_{d_s} \dd t\,,
\end{equation}
where $\dd M^{(s)}_{d}$ yields a martingale after integration, and $\E[\Theta^{(s)}_{d}]$ satisfies
\begin{align}\label{rhs2}
\E[\Theta^{(s)}_{d_s}(t,z)]=\sum_{ \substack { P^{(s+1)}_{d_{s+1}} \in \mathcal{P}^{(s+1)}_{d_{s+1}} \\ d_s+1\leq d_{s+1}< D}} \E[ P^{(s+1)}_{d_{s+1}}(t,z)]+O_{\prec}(N^{-1/2}+\Psi^D)\,,
\end{align}
uniformly in $t \geq 0$ and $z \in S_{\mathrm{edge}}$, where the sums in~\eqref{rhs2} are over finitely many type-0 terms in the $(s+1)$-th step given in (\ref{P_ds}) and the number of such terms is determined by $D$ and $n_s$. Moreover, the number of Green function entries in the product of each type-0 term is finite and determined by $D$,~$n_s$.

We run the dynamics of $P^{(s)}_{d_s}$ in (\ref{late22}) up to $T=8 \log N$ as chosen previously. We next estimate the size of $\E[P^{(s)}_{d_s}(t,z)]$ at the terminal time $T$ for any $P^{(s)}_{d_s} \in \mathcal{P}^{(s)}_{d_s}$, with $s \geq 1$ and $d_s \geq 2 $. Indeed, from (\ref{approxxxx}) and the local law in (\ref{G}), we have
\begin{align}\label{approx_gue_1}
\big|\E[P^{(s)}_{d_s}(T,z)]-\E^{\mathrm{GUE}}[P^{(s)}_{d_s}(z)]\big| =O(N^{-1})\,.
\end{align}
Together with the estimate (\ref{tracek}) for the GUE, we obtain that, for any $s \geq 1$ and $d_s \geq 2$, 
\begin{align}\label{approx_gue}
\big|\E[P^{(s)}_{d_s}(T,z)]\big|=O_{\prec}(N^{-1/3})\,.
\end{align}

Next, we return to the stochastic differential equation of $P^{(s)}_{d_s}$ in (\ref{late22}). Integrating (\ref{late22}) over $ [t', T]$ for any $0\leq t' \leq T$ and taking the expectation as in (\ref{previous_d_1}), we have from~(\ref{rhs2}) and~(\ref{approx_gue}) that
\begin{align}\label{temp_eq}
\E[P^{(s)}_{d_s}(t',z)]=&\sum_{ \substack{ P^{(s+1)}_{d_{s+1}} \in \mathcal{P}^{(s+1)}_{d_{s+1}} \\d_s+1 \leq d_{s+1} <D}}  \int_{t'}^{T} \E[ P^{(s+1)}_{d_{s+1}}(t,z)] \dd t +O_{\prec}\Big(\log N (N^{-1/2}+\Psi^D)+N^{-1/3}\Big).
\end{align}

Now, we are ready to iterate using ~\eqref{temp_eq}. In the first step, we start by $P_{d_1}^{(1)}(t,z)$ in (\ref{P_d_s_1}) and have~(\ref{temp_eq}) for $s=1$. The number of the terms $P^{(2)}_{d_2} \in \mathcal{P}^{(2)}_{d_2}$ with $d_2 \geq d_1+1$ on the right side of (\ref{temp_eq}) is finite and depends on $n_1$ and $D$.  Then we further estimate these type-0 terms $P^{(2)}_{d_2}$ using (\ref{temp_eq}) for $s=2$ as the second step. The resulting type-0 terms $P^{(3)}_{d_3} \in \mathcal{P}^{(2)}_{d_3}$ with $d_3 \geq d_2+1 \geq d_1+2$ will be estimated again using (\ref{temp_eq}) for $s=3$ as the third step. Since in each step of using (\ref{temp_eq}), the degrees of the corresponding type-0 terms $P^{(s+1)}_{d_{s+1}} \in \mathcal{P}^{(s+1)}_{d_{s+1}} $ on the right side of (\ref{temp_eq}) are increased by at least one, we have $d_{s+1} \geq d_1+s$. We hence stop at step $s=s_0:=D-d_1$. For any $P^{(s_0)}_{d_{s_0}} \in \mathcal{P}^{(s_0)}_{d_{s_0}}$ with $d_{s_0} \geq D-1$, the resulting terms $P^{(s_0+1)}_{d_{s_0+1}} \in \mathcal{P}^{(s_0+1)}_{d_{s_0+1}}$ on the right side of (\ref{temp_eq}) have degrees $d_{s_0+1} \geq D$. The number of these terms is finite and depends on $D,n_1$. Using the local law in (\ref{G}), all these terms can be bounded by $O_{\prec}(\Psi^D+N^{-1})$. This implies that the finite sum of these terms after integration over $[t',T]$ can be absorbed into the error term on the right side of (\ref{temp_eq}). That is, for any $P^{(s_0)}_{d_{s_0}} \in \mathcal{P}^{(s_0)}_{d_{s_0}}$ with $d_{s_0} \geq D-1$, 
$$\big|\E[P^{(s_0)}_{d_{s_0}}(t',z)]\big|=O_{\prec}\Big(\log N (N^{-1/2}+\Psi^D)+N^{-1/3}\Big)\,.$$

We hence plug the above estimate back to the previous step, \ie (\ref{temp_eq}) for $s=s_0-1$. We then obtain a similar estimate for any $P^{(s_0-1)}_{d_{s_0-1}} \in \mathcal{P}^{(s_0-1)}_{d_{s_0-1}}$ with $d_{s_0-1} \geq D-2$,
$$\big|\E[P^{(s_0-1)}_{d_{s_0-1}}(t',z)]\big|=O_{\prec}\Big(\log^2 N (N^{-1/2}+\Psi^D)+N^{-1/3} \log N\Big)\,.$$

Repeating the above process until $s=1$, we hence obtain that, for $d_1 \geq 2$,
$$\big|\E[P^{(1)}_{d_{1}}(t,z)]\big|=O_{\prec}\Big((N^{-1/3}+ \Psi^D) \log^{D} N\Big),$$
uniformly in $t \in[0,T]$ and $z \in S_{\mathrm{edge}}$. By choosing $D \geq \frac{1}{\epsilon}$ with $\epsilon>0$ as in (\ref{control_eq}), we prove~(\ref{tracek_wigner}) for~$t \in [0,T]$. If $t \geq T$, a similar estimate can be obtained by using (\ref{approx_gue_1}) and (\ref{approx_gue}). We have hence finished the proof of Lemma \ref{lemma_trace_wigner}.
\end{proof}

\subsection{Proof of Lemma \ref{lemma_trace}} \label{toy_type_zero_gue}
We end this section with the proof of Lemma \ref{lemma_trace} considering the GUE.
\begin{proof}[Proof of Lemma \ref{lemma_trace}]
Using the spectral decomposition~\eqref{le spectral decomp}, we write
\begin{align}\label{imagine}
\frac{1}{N} \E^{\mathrm{GUE}} \Big[ \Im \Tr G (z)\Big]=\frac{N \eta}{N^2} \E^{\mathrm{GUE}} \Big[ \sum_{j=1}^N \frac{1}{|\lambda_j-z|^2}\Big]\,,\qquad z\in S_{\mathrm{edge}}\,.
\end{align}
Then it suffices to estimate the following linear eigenvalue statistics, which can be written from (\ref{determinant}), (\ref{wt_kernel}) and then (\ref{edge}) as 
\begin{align}\label{sum_0}
\frac{1}{N^2} \E^{\mathrm{GUE}} \Big[\sum_{i=1}^{N} \frac{1}{|\lambda_i-z|^{2}}\Big]=\frac{1}{N^2} \int_{\R} \frac{\wt K_{N}(x,x)}{|x-2-\kappa-\ii \eta|^2} \,\dd x=\frac{1}{N^\frac{2}{3}}\int_{\R} \frac{ K^{\mathrm{edge}}_{N}( x, x)}{| x-N^{2/3} \kappa-\ii N^{2/3} \eta|^2}\, \dd  x\,,
\end{align}
where $z=2+ \kappa+\ii \eta\in S_{\mathrm{edge}}$, with $|\kappa| \leq C_0N^{-2/3+\epsilon}$ and $N^{-1+\epsilon} \leq \eta \leq N^{-2/3+\epsilon}$. 

To control the integral on the right side of~\eqref{sum_0}, we choose a fixed $L_0<0$ (see~Lemma~\ref{lemma_airy_bound} and  Theorem \ref{kernel_diff}) and split the real line in the parts,  $(-\infty, -N^{2/3}]$, $(-N^{2/3},L_0]$ and $(L_0\,\infty)$.

For the integration domain $(-\infty,-N^{2/3}]$, we find that
\begin{align}\label{sum_1}
\frac{1}{N^\frac{2}{3}}\int_{x<-N^{2/3}} \frac{ K^{\mathrm{edge}}_{N}(x,x)}{|x-N^{2/3} \kappa-\ii N^{2/3} \eta|^2}\, \dd x=O(N^{-1})\,,
\end{align}
using the trace identity (\ref{eq_1}) for the kernel $K_N$ and that $|\kappa| \leq C_0N^{-2/3+\epsilon}$.

Moreover, from Theorem \ref{kernel_diff} and Lemma \ref{lemma_airy_bound}, we have on $(L_0,\infty)$, that
\begin{align}\label{sum_2}
\frac{1}{N^\frac{2}{3}}\int_{x>L_0} \frac{ K^{\mathrm{edge}}_{N}(x,x)}{|x-N^{2/3}\kappa-\ii N^{2/3} \eta|^2} \,\dd x=&\frac{1}{N^\frac{2}{3}}\int_{x>L_0} \frac{ K_{\mathrm{airy}}(x,x)+O(N^{-2/3})}{|x-N^{2/3}\kappa-\ii N^{2/3} \eta|^2}\, \dd x=O\Big(  \frac{1}{N^{\frac{4}{3}} \eta}\Big)\,.
\end{align}

It hence suffices to focus on the regime $(-N^{2/3}, L_0]$. Recall from~\eqref{kernel_xy} and~\eqref{edge} that
\begin{equation}
K_N(x,x)=\sum_{k=0}^{N-1} \phi^2_k(x); \qquad K^{\mathrm{edge}}_{N}(x,x)=\frac{1}{N^{1/6}} K_N\Big( 2\sqrt{N}+\frac{x}{N^{1/6}},2\sqrt{N}+\frac{x}{N^{1/6}} \Big)\,.
\end{equation}
From (\ref{hermite_fn}) and (\ref{kernel_xx}), the derivative of $K_N(x,x)$ is given by
$$K'_N(x,x)=-\sqrt{N}\phi_{N-1}(x) \phi_{N}(x)\,.$$
The Hermite functions satisfy, for all $k$,
\begin{align}\label{upper_bound_hermite}
\sup_{x \in \R} |\phi_k(x)| \leq Ck^{-1/12}\,.
\end{align} 
for some constant $C$ independent of $k$, as was proved in \cite{upper_hermite}.
Therefore, the derivative of the edge kernel $K^{\mathrm{edge}}_{N}(x,x)$ is given by
\begin{align}\label{deri_edge}
\Big( K^{\mathrm{edge}}_{N}(x,x)\Big)'=\frac{1}{N^{1/3}} K'_N\Big( 2\sqrt{N}+\frac{x}{N^{1/6}},2\sqrt{N}+\frac{x}{N^{1/6}} \Big)=O(1)\,.
\end{align}
For any $x\in(-N^{2/3},L_0]$, we have from (\ref{deri_edge}) and Lemma \ref{lemma_airy_bound} that
\begin{align}\label{derivative_K_edge}
K^{\mathrm{edge}}_{N}(x,x)=K^{\mathrm{edge}}_{N}(L_0,L_0)-\int_{x}^{L_0} \big(K^{\mathrm{edge}}_{N}(x,x)\big)' \dd x \leq C'(1+|x|).
\end{align}
Therefore, we obtain from (\ref{derivative_K_edge}) that
\begin{align}\label{sum_3}
\frac{1}{N^\frac{2}{3}}\int_{-N^{2/3} < x < L_0} &\frac{ K^{\mathrm{edge}}_{N}(x,x) }{| x-N^{2/3}\kappa-\ii N^{2/3} \eta|^2}\, \dd x \leq   \nonumber \frac{C'}{N^\frac{2}{3}}\int_{-N^{2/3} < x < L_0} \frac{ 1+|x| }{(x-N^{2/3}\kappa)^2+(N^{2/3} \eta)^2}\, \dd x\nonumber\\
=&O\Big(N^{-2/3}\log N+\frac{N^\epsilon}{N^\frac{4}{3}\eta} \Big)=O\Big(  \frac{1}{N^{\frac{4}{3}-\epsilon} \eta}\Big)\,,
\end{align}
where we used that $|\kappa| \leq C_0N^{-2/3+\epsilon}$.

Plugging (\ref{sum_1}), (\ref{sum_2}) and (\ref{sum_3}) into (\ref{sum_0}), there exists some constant $C$ independent of $\epsilon$ such that
\begin{align}
\frac{1}{N^2} \E^{\mathrm{GUE}} \Big[ \sum_{j=1}^N \frac{1}{|\lambda_j-z|^2}\Big] \leq \frac{CN^{\epsilon}}{N^{\frac{4}{3}} \eta}\,,
\end{align}
uniformly in $z \in S_{\mathrm{edge}}$, for sufficiently large $N \geq N_0(\epsilon,C_0)$. In combination with (\ref{imagine}), we hence have proved~(\ref{traceimg}). 

Finally, we consider any type-0 term $P_d(z) \in \mathcal{P}_d(z)$ of the form in (\ref{form}) of degree $d \geq 2$ for the GUE. For notational simplicity, we no longer emphasize the indices $a$,$b$ and write
\begin{align}\label{cluster}
P_d(z) =\frac{1}{N^{n}} \sum_{v_1=1}^N \cdots   \sum_{v_n=1}^N c_{v_1, \ldots,v_n} \Big(  \prod^{n}_{i=1} G_{x_{i} y_{i}} (z)\Big),
\end{align}
with $n \geq 2$, where each summation index $v_j$ $(1\leq j\leq n)$ appears once in the row index set $\{x_i\}_{i=1}^n$ and once in the column index set $\{y_i\}_{i=1}^n$ and the coefficients $\{c_{v_1, \ldots,v_n}\}$ are uniformly bounded complex numbers. For any $1\leq j \leq n$, if there exists $1\leq i\leq n$ such that $x_i=y_i=v_j$, then we say that $v_j$ is isolated. For any $1 \leq j \neq j' \leq n$, if there exists $1 \leq i \leq n$ such that either $x_i=v_j$, $y_i=v_{j'}$ or $y_i=v_j$, $x_i=v_{j'}$, then we say that $v_j$ and $v_{j'}$ are connected indices.  Because the degree of (\ref{cluster}) is at least two, there exists at least one cluster of connected indices containing at least two elements. We may assume that $v_1, \ldots, v_{n_0} $ $(2 \leq n_0 \leq n)$ form a cluster of connected indices. Using the local law in (\ref{G}), we have
$$|P_d(z)| \prec \frac{1}{N^{n_0}} \sum_{v_1=1}^N \cdots   \sum_{v_{n_0}=1}^N \big|G_{v_1v_2} G_{v_2v_3}  \cdots G_{v_{n_0}v_1}(z)\big|\,.$$
If $n_0=2$, from Young's inequality and the Ward identity
\begin{align}\label{ward}
\frac{1}{N^2} \sum_{i,j} |G_{ij}(z)|^2 =\frac{\Im m_N(z)}{N \eta}\,,\qquad z=E+\ii \eta \in \C^+,
\end{align}
which follows from the spectral decomposition~\eqref{le spectral decomp}, we then obtain 
\begin{align}
|P_d(z)| \prec & \frac{1}{N^2}\sum_{v_1,v_2} \big|G_{v_1v_2}(z) G_{v_2v_1}(z)\big| \leq \frac{1}{2 N^2}\sum_{v_1,v_2} \big( |G_{v_1v_2}(z)|^2+|G_{v_2v_1}(z)|^2 \big) = \frac{\Im m_N(z)}{ N \eta }\,.
\end{align}
For $n_0 \geq 3$, we have similarly from the local law (\ref{G}) that
\begin{align}\label{traces}
|P_d(z)| \prec & \Psi^{n_0-2} \frac{1}{N^3}\sum_{v_1,v_2,v_3} \big|G_{v_1v_2}(z) G_{v_2v_3}(z)\Big| \nonumber\\
\leq &\Psi^{n_0-2} \frac{1}{2 N^3}\sum_{v_1,v_2,v_3} \big( |G_{v_1v_2}(z)|^2+|G_{v_2v_3}(z)|^2 \big)=O\Big( \frac{\Im m_N(z)}{(N \eta)^{n_0-1}}\Big),
\end{align}
where in the last two steps we use Young's inequality, the Ward identity~\eqref{ward}, and that $\Psi(z)=O(\frac{1}{N\eta})$ for any $z \in S_{\mathrm{edge}}$. Therefore, combining with the estimate (\ref{traceimg}) for the expectation of $\Im m_N(z)$, the properties of stochastic domination in Lemma~\ref{dominant}, and that $\eta \geq N^{-1+\epsilon}$, we have, for any $\tau>0$,
$$\E^{\mathrm{GUE}}\big[|P_d^{(s)}(z)|\big] \leq  N^{-1/3+\tau}, \qquad d \geq 2,$$
uniformly in $z \in S_{\mathrm{edge}}$, for sufficiently large $N \geq N'_0(C_0, \epsilon, \tau)$. This completes the proof of~(\ref{tracek}), and hence the proof of Lemma~\ref{lemma_trace}.
\end{proof}

\section{Product of Green function entries with unmatched indices}\label{sec:unmatch}

In this section, we prove Proposition \ref{unmatch_lemma}. Before stating the proof for Wigner matrices, we first consider the GUE for the intuition why expectations of unmatched terms are much smaller than the naive size obtained using power counting and the local law as in (\ref{localaw_0}).

\subsection{Intuition from the GUE}\label{sec:unmatch_gue}
In this subsection, we focus on the special case of the GUE, as in Subsection \ref{subsec:gue}. Consider any $Q^o_d \in \mathcal{Q}^o_d$ of the form (\ref{product}). Using the spectral decomposition~\eqref{le spectral decomp} and the unitary invariance of the GUE similarly as in (\ref{sum_gue}), we write the expectation of the unmatched $Q^o_d$ as
\begin{align}\label{temp_q}
\E[Q_d^o]
=&\frac{1}{N^{\# \mathcal{I}}}\sum_{\mathcal{I}}  c_{\mathcal{I}} \sum_{j_1, \ldots, j_n=1}^N\E \Big[\prod_{i=1}^n \frac{1}{(\lambda_{j_i}-z)}\Big] \times \E \Big[\prod_{i=1}^n \u_{j_i}(x_i) \overline{\u_{j_{i}}(y_i)}\Big]\,,
\end{align}
with $(\lambda_j)$ the eigenvalues and the corresponding normalized eigenvectors $(\u_j)$, and each $x_i$, $y_i$ represent some free summation index in $\mathcal{I}$. In order to estimate the expectations of the eigenvectors, we recall the Weingarten calculus formula in Lemma \ref{wein_lemma}. Under Definition~\ref{unmatch_def} for unmatched indices, if the values of the free summation indices in $\mathcal{I}$ are distinct, then $\delta_{x_1, y_{\sigma(1)}} \cdots \delta_{x_n, y_{\sigma(n)}}=0$, for any permutation $\sigma \in S_n$. Thus from~(\ref{wein_temp}), for any $1 \leq j_1, \cdots, j_n \leq N$, we have
$$\E \Big[\prod_{i=1}^n \u_{j_i}(x_i) \overline{\u_{j_{i}}(y_i)}\Big]=0.$$
The non-vanishing contributions come from the diagonal cases when the values of some free summation indices in $\mathcal{I}$ coincide. Because of the averaged form of $Q_d^o$ in~(\ref{product}) and the local law in~(\ref{G}) one works out that, for any $z \in S$ and $t\ge0$, 
\begin{align}\label{Wein}
\E[Q_d^{o}]=O(N^{-1})\,.
\end{align}

For Wigner matrices, the above argument does not apply anymore. We hence use similar expansions as in Subsection \ref{toy_fourth} to extend to arbitrary Wigner matrices. Before we give the proof of Proposition~\ref{unmatch_lemma}, we start by considering an example of the unmatched term in $\mathcal{Q}^o_d$ to illustrate the mechanism.

\subsection{Example of an unmatched term} \label{sec:unmatch_toy}
We look at the following example of an unmatched term
\begin{align}\label{example1}
 \frac{1}{N^2} \sum_{a,b} G_{ab} G_{ba} G_{ab} \in \mathcal{Q}^o_3,
\end{align}
with $a \in \mathcal{R}^o$ and $b \in \mathcal{C}^o$; see (\ref{rc_index}) in Definition \ref{unmatch_def}. Using the local law in (\ref{G}), the expectation of this term can be naively bounded by $O_{\prec}(\Psi^3+N^{-1})$. The idea to improve this bound is similar to the proof of Proposition \ref{lemma_expand_type}. Note that the combination of the identity~\eqref{resolvent_identity} and the cumulant expansion gives a cancellation to the leading order. Thus we can improve the upper bound to $O_{\prec}(\Psi^{4}+\frac{\Psi^3}{\sqrt{N}}+N^{-1})$. We next discuss the details.

Using the identity (\ref{resolvent_identity}) on the off-diagonal entry $G_{ab}$ with unmatched $a$ as the row index and applying cumulant expansions, we have
\begin{align}\label{example_step1}
\frac{1}{N^2} \sum_{a,b} \E [G_{ab} G_{ba} G_{ab} ]=&\frac{1}{N^2} \sum_{a \neq b} \E \Big[ \Big( G_{ab} \ud{HG} -\ud{G} (HG)_{ab}\Big) G_{ba} G_{ab} \Big]+\frac{1}{N^2} \sum_{a=1}^N \E [(G_{aa})^3]\nonumber\\
=&\frac{1}{N^4}\sum_{a,b,j,k} \E \Big[ \frac{ \partial G_{ab} G_{jk} G_{ba} G_{ab} }{\partial h_{jk}}\Big]-\frac{1}{N^4}\sum_{a,b,j,k} \E \Big[ \frac{ \partial  G_{jj} G_{kb} G_{ba} G_{ab}}{\partial h_{k a}}\Big]\nonumber\\
&+\frac{1}{\sqrt{N}} \frac{1}{N^4} \sum_{p+q+1=3} \frac{1}{p! q!}   \sum_{a,b,j,k}  s^{(p,q+1)}_{jk} \E \Big[ \frac{ \partial^2 G_{ab} G_{jk} G_{ba} G_{ab} }{\partial h^p_{jk} \partial h^q_{kj}}\Big]\nonumber\\
&-\frac{1}{\sqrt{N}} \frac{1}{N^4} \sum_{p+q+1=3} \frac{1}{p! q!}  \sum_{a,b,j,k}  s^{(p,q+1)}_{ak} \E \Big[ \frac{ \partial^2 G_{jj} G_{kb} G_{ba} G_{ab}}{\partial h^p_{ka} \partial h^q_{ak} }\Big]+O_{\prec}\big(\frac{1}{N}\big)\,,
\end{align}
where the last error term comes from the truncation of the cumulant expansions at the third order and the diagonal case $a=b$. 

Using (\ref{dH}) and that $j,k$ are fresh summation indices, all the third order expansions for $\{p+q+1=3\}$ can be written out using the terms of the form in (\ref{product}) of degree at least three, with an additional factor $\frac{1}{\sqrt{N}}$ in front. 
Since both the fresh indices $j,k$ appear in the product of the Green function entries for an odd number of times, the resulting terms are unmatched from Definition \ref{unmatch_def}. From the local law in~(\ref{G}), they are bounded by~$O_{\prec}\Big(\frac{\Psi^3}{\sqrt{N}}+\frac{1}{N^{3/2}} \Big)$.

Now we return to the second order terms in the cumulant expansions in (\ref{example_step1}), \ie
\begin{align}\label{second_temp}
\frac{1}{N^4}\sum_{a,b,j,k} \E \Big[ \frac{ \partial G_{ab} G_{jk} G_{ba} G_{ab} }{\partial h_{jk}}\Big]-\frac{1}{N^4}\sum_{a,b,j,k} \E \Big[ \frac{ \partial G_{jj} G_{kb} G_{ba} G_{ab}}{\partial h_{ka}}\Big]\,.
\end{align}
Using (\ref{dH}), the fresh indices $j,k$ are then matched and the index $a$ remains to be an unmatched row index. The key observation here is that the leading sub-term from the first term above, corresponding to taking $\frac{\partial }{\partial h_{jk}}$ of $G_{jk}$, will be canceled precisely by the leading sub-term from the second term above, resulting from taking $\frac{\partial }{\partial h_{ka}}$ of $G_{kb}$. We hence rewrite (\ref{second_temp}) as
\begin{align}
\frac{1}{N^4}\sum_{a,b,j,k} \E \Big[ \frac{ \partial G_{ab}  G_{ba} G_{ab}}{\partial h_{jk}} G_{jk} \Big]-\frac{1}{N^4}\sum_{a,b,j,k} \E \Big[ \frac{ \partial   G_{jj}G_{ba} G_{ab}}{\partial h_{ka}} G_{kb} \Big]\,.
\end{align}
The degrees of the resulting terms from the first part above are five as $j,k$ are fresh indices. Similarly, the ones from the second part have degrees at least four, except one sub-term from taking $\frac{\partial}{\partial h_{ka}}$ of $G_{ba}$, whose column index coincides with the unmatched row index $a$:
$$\frac{1}{N^4}\sum_{a,b,j,k} \E \Big[G_{jj} G_{bk} G_{aa}   G_{ab}  G_{kb} \Big]\,.$$
Compared with the original term in (\ref{example1}), one replaces one pair of the index $a$ by a fresh index $k$ and adds a factor $G_{aa}$ for the replaced index $a$. The good news is that this leading term of degree three remains unmatched with an unmatched row index~$a$. We then expand it further as in (\ref{example_step1}) and obtain that
\begin{align}\label{example_step2}
\frac{1}{N^4}\sum_{a,b,j,k}& \E \Big[G_{jj} G_{aa}   G_{ab} G_{bk}   G_{kb} \Big]=\frac{1}{N^6}\sum_{a,b, j, k, j',k'} \E \Big[ \frac{ \partial G_{jj} G_{aa}   G_{ab}  G_{bk} G_{kb}}{\partial h_{j'k'}} G_{j'k'} \Big]\nonumber\\
-&\frac{1}{N^6}\sum_{a,b, j,k,j',k'} \E \Big[ \frac{ \partial  G_{jj} G_{aa}   G_{j'j'} G_{bk} G_{kb} }{\partial h_{k' a}} G_{k'b} \Big]+\{\mbox{third order terms}\}+O_{\prec}\Big(\frac{1}{N}\Big)\,,
\end{align}
with $j',k'$ another two fresh summation indices. Here, the third order terms are also unmatched terms of the form in (\ref{product}) of degree at least three with an extra $\frac{1}{\sqrt{N}}$ in front, similarly as in (\ref{example_step1}). From (\ref{dH}), the resulting terms from the first part on the right side of (\ref{example_step2}) have degrees at least five. As for the second part above, even though the column index of the diagonal entry $G_{aa}$ coincides with the unmatched row index $a$, the resulting terms have degrees at least four.

In this way, we improve the upper bound of the unmatched term given in (\ref{example1}) to
$$\Big|\frac{1}{N^2} \sum_{a,b} \E \Big[G_{ab} G_{ba} G_{ab} \Big] \Big| \prec \Psi^{4}+\frac{\Psi^3}{\sqrt{N}}+N^{-1}\,.$$
Indeed, we expand this unmatched term as
\begin{align}\label{example100}
\frac{1}{N^2} \sum_{a,b} \E \Big[G_{ab} G_{ba} G_{ab}]=\sum_{\substack{ Q^o_{d_1'} \in \mathcal{Q}^o_{d_1'} \\ d_1' \geq 4}}  \E[Q^o_{d_1'}]+\frac{1}{\sqrt{N}}\sum_{\substack{ Q^o_{d_2'} \in \mathcal{Q}^o_{d_2'} \\ d_2' \geq 3}} \E[ Q^o_{d_2'}]+O_{\prec}(N^{-1})\,,
\end{align}
where we write $\sum_{Q^o_{d_1'} \in \mathcal{Q}^o_{d_1'} , d_1' \geq 4} Q^o_{d_1'}$ as a sum of finitely many unmatched terms of the form in (\ref{product}) of degrees increased by at least one, which comes from the second order expansions. Moreover, we write $\frac{1}{\sqrt{N}} \sum_{Q^o_{d_2'} \in \mathcal{Q}^o_{d_2'} , d_2' \geq 3} Q^o_{d_2'}$ as a finite sum of unmatched terms of the form in (\ref{product}) with an extra factor $\frac{1}{\sqrt{N}}$ in front, which corresponds to the third order expansions. The last error term $O_{\prec}(N^{-1})$ is from the truncation of the cumulant expansion and the diagonal cases. By repeating the above expansion procedure in (\ref{example100}) for arbitrary $D$ times, we improve the upper bound to $O_{\prec}\big( \Psi^{D}+\frac{\Psi^{D-1}}{\sqrt{N}}+N^{-1}\big)$. The full proof is presented in the following section.

\subsection{Proof of Proposition \ref{unmatch_lemma}}\label{sec:unmatch_proof}

In this section, we give the proof of Proposition \ref{unmatch_lemma} for Wigner matrices using the cumulant expansions as explained above.

\begin{proof}[Proof of Proposition \ref{unmatch_lemma}]
Consider an arbitrary unmatched term $Q_d^{o} \in \mathcal{Q}_d^o$ of the form (\ref{product}). Because it is equivalent to expand a Green function entry $G_{xy}$ in the row index $x$ or column index $y$, we focus on the unmatched row indices in the following.  

We may assume that the index $v_1$ belongs to the unmatched row index set $\mathcal{R}^o$ (which cannot be empty) from Definition \ref{unmatch_def}. Then there exists an off-diagonal factor in the product of Green function entries with $v_1$ as the row index. Without loss of generality, we set $x_1=v_1$, and $y_1 \neq v_1$. Using~(\ref{resolvent_identity}) on the off-diagonal entry $G_{v_1y_1}$ and applying cumulant expansions similarly as in (\ref{example_step1}), we have
\begin{align}\label{third_unmatch}
\E[Q_d^o]=&\frac{1}{N^{\# \mathcal{I}}} \sum_{\mathcal{I}} c_{\mathcal{I}} \E \Big[ G_{v_1 y_1} \prod^n_{i=2} G_{x_i y_i} \Big]=\frac{1}{N^{\# \mathcal{I}}}\sum_{\mathcal{I}}c_{\mathcal{I}} \E \Big[ \delta_{v_1 y_1}G_{v_1 y_1} \prod^n_{i=2} G_{x_i y_i} \Big]\nonumber\\
&+\frac{1}{N^{2+\# \mathcal{I}}} \sum_{\mathcal{I}}c_{\mathcal{I}} \sum_{j,k}\E \Big[ \frac{ \partial G_{v_1 y_1} G_{jk} \prod^n_{i=2} G_{x_i y_i} }{\partial h_{jk}}\Big]-\frac{1}{N^{2+\# \mathcal{I}}}\sum_{\mathcal{I}}c_{\mathcal{I}} \sum_{j,k} \E \Big[ \frac{ \partial G_{jj} G_{ky_1}  \prod^n_{i=2} G_{x_i y_i} }{\partial h_{k v_1}}\Big]\nonumber\\
&+\frac{1}{N^{2+\# \mathcal{I}}} \frac{1}{\sqrt{N}} \sum_{\mathcal{I}}c_{\mathcal{I}} \sum_{p+q+1=3} \frac{1}{p!q!}  \sum_{j,k} s^{(p,q+1)}_{jk} \E \Big[ \frac{ \partial^2 G_{x_1y_1} G_{jk} \prod^n_{i=2} G_{x_i y_i} }{\partial h^p_{jk}\partial h^q_{kj}}\Big]\nonumber\\
&-\frac{1}{ N^{2+\# \mathcal{I}}} \frac{1}{\sqrt{N}} \sum_{\mathcal{I}}c_{\mathcal{I}}\sum_{p+q+1=3} \frac{1}{p!q!}   \sum_{j,k} s^{(p,q+1)}_{v_1k} \E \Big[ \frac{ \partial^2 G_{jj}G_{ky_1}  \prod^n_{i=2} G_{x_i y_i} }{\partial h^p_{kv_1}\partial h^q_{v_1k}}\Big]+O_{\prec}(\frac{1}{N})\nonumber\\
=&\frac{1}{N^{2+\# \mathcal{I}}}\sum_{\mathcal{I}}c_{\mathcal{I}} \sum_{j,k} \E \Big[ \frac{ \partial G_{v_1 y_1} \prod^n_{i=2} G_{x_i y_i} }{\partial h_{jk}}G_{jk}  \Big]-\frac{1}{N^{2+\# \mathcal{I}}}\sum_{\mathcal{I}}c_{\mathcal{I}} \sum_{j,k} \E \Big[ \frac{ \partial G_{jj} \prod^n_{i=2} G_{x_i y_i} }{\partial h_{kv_1}}G_{k y_1} \Big]\nonumber\\
&+\{\mbox{ third order terms for } p+q+1=3\}+O_{\prec}\big(\frac{1}{N}\big)\,,
\end{align}
where $j,k$ are fresh summation indices, the last error $O_{\prec}(\frac{1}{N})$ is from the truncation of the cumulant expansions at the third order and the diagonal case $v_1 \equiv y_1$. 

We first look at the third order expansions for $p+q+1=3$, which are much smaller because we gain an extra $\frac{1}{\sqrt{N}}$ from the third order cumulants. Since both $j,k$ are fresh indices, it is straightforward to check from (\ref{dH}) that the resulting terms are also of the form in (\ref{product}) with an extra $\frac{1}{\sqrt{N}}$ in front. Their degrees, denoted by $d'$, satisfy $d' \geq d$, the corresponding free summation index set is $\mathcal{I}'=\{\mathcal{I}, j,k\}$ and the number of Green function entries is $n'=n+3$. In addition, the number of such terms is at most $6(n+3)^2$. Comparing these terms with the original $Q_d^o$, we add in total an odd number of $j$'s (or $k$'s) into the original row index set and column index set of the product of the Green function entries. Then all these terms are unmatched terms from Definition \ref{unmatch_def}. We use $\frac{1}{\sqrt{N}} \sum_{Q^o_{d'} \in \mathcal{Q}^o_{d'}; d' \geq d} \E [Q^o_{d'}]$ to denote the finite sum of these unmatched terms from the third order expansions.

Next, we estimate the second order expansion terms, \ie the second but last line on the right side of~(\ref{third_unmatch}). Using~(\ref{dH}) we write them as a sum of at most $2n$ terms of the form in~(\ref{product}) with $\mathcal{I}'=\{\mathcal{I}, j,k\}$ and $n'=n+2$. The degrees of these terms are estimated as follows.

For the first group of terms in the second but last line of (\ref{third_unmatch}), comparing with the original $Q^o_d$, we have added one fresh index~$j$ and one fresh index~$k$ into both the original row index set and column index set. Then $j$ and $k$ are both matched indices. Moreover, $v_1$ from $G_{v_1y_1}$ remains an unmatched row index. After taking $\frac{\partial}{\partial h_{jk}}$ by (\ref{dH}), the degrees are then increased by at least two. 

Similarly, we compare the second group in the second but last line of (\ref{third_unmatch}) with the original $Q^o_d$. We find again that both $j$ and $k$ are matched, and the index $v_1$ is still an unmatched row index. However, the degrees of the resulting terms from taking $\frac{\partial}{\partial h_{kv_1}}$ may not be increased. This is because the column index of some Green function entry $G_{x_iy_i} (2 \leq i \leq n)$ may coincide with the unmatched row index $v_1$. The number of such Green function entries with $v_1$ as column index is given by $\nu_1^{c} (\leq n)$ from Definition~\ref{unmatch_def}. So we split the discussion into three cases.

{\bf Case 1:} If $y_i \neq v_1$, then after taking $\frac{\partial}{\partial h_{kv_1}}$ of $G_{x_iy_i}$, the degree of the resulting term is increased by at least one. 

{\bf Case 2:}  If $y_i=x_i = v_1$, then after taking $\frac{\partial}{\partial h_{kv_1}}$ of $G_{x_iy_i}$, the degree is then increased by exactly one.

{\bf Case 3:}  If $y_i=v_1$, but $x_i \neq v_1$, then, for simplicity, we may assume that $y_2=v_1$ and $x_2 \neq v_1$. From Definition~\ref{unmatch_def} for unmatched indices, there exists some $3 \leq i' \leq n$ such that $x_{i'}=v_1$ and $y_{i'} \neq v_1$, because else $v_1$ cannot be an unmatched row index of the original $Q_d^o$. We may assume $x_3=v_1$ and $y_3 \neq v_1$. Then the corresponding term after taking $\frac{\partial}{\partial h_{kv_1}}$ of $G_{x_2,v_1}$ becomes
\begin{align}\label{term_1}
(*):=\frac{1}{N^{2+\# \mathcal{I}}}\sum_{\mathcal{I},j,k} c_{\mathcal{I}} \E \Big[ G_{jj}G_{v_1v_1} G_{k y_1} G_{x_2k}  G_{v_1,y_3}   \prod^n_{i=4} G_{x_i y_i} \Big]\,,
\end{align}
with $y_1 \neq v_1$, $x_2 \neq v_1$, and $y_3 \neq v_1$, and the degree of this term is still $d$. Compared with the original $Q_d^o$, we have replaced one pair of the index $v_1$, \ie the row index of $G_{x_1y_1}$ and the column index of $G_{x_2y_2}$, by the fresh index $k$. Further we get an additional diagonal Green function entry $G_{v_1v_1}$ for the replaced pair of index $v_1$. Since the index $v_1$ from $G_{v_1y_3}$ remains an unmatched row index, we can further expand the term in (\ref{term_1}) using the unmatched row index $v_1$, as in (\ref{third_unmatch}). We write
\begin{align}\label{pian2}
(*)=&-\frac{1}{N^{4+\# \mathcal{I}}}\sum_{\mathcal{I},j,k, j',k'} c_{\mathcal{I}}\E \Big[ \frac{ \partial G_{jj}G_{v_1v_1} G_{v_1 y_3} G_{x_2k}  G_{ky_1} \Big( \prod^n_{i=4} G_{x_i y_i} \Big) }{\partial h_{j'k'}}G_{j'k'}  \Big]\nonumber\\
&+\frac{1}{N^{4+\# \mathcal{I}}}\sum_{\mathcal{I},j,k,j',k'} c_{\mathcal{I}} \E \Big[ \frac{ \partial G_{jj}G_{v_1v_1}  G_{j' j'} G_{x_2k}  G_{k y_1} \Big( \prod^n_{i=4} G_{x_i y_i} \Big)}{\partial h_{k' v_1}}G_{k' y_3} \Big]\nonumber\\
&+\{\mbox{ third order expansions for }p+q+1=3 \}+O_{\prec}\big(\frac{1}{N}\big)\,.
\end{align}

Similar as (\ref{third_unmatch}), the third order expansions contains at most $6(n+5)^2$ unmatched terms of the form in (\ref{product}) with an additional factor $\frac{1}{\sqrt{N}}$ in front, of degrees $d'' \geq d$, with $\mathcal{I}''=\{\mathcal{I},j,k,j',k'\}$ and $n''=n+5$.  We next estimate the second order expansions on the right side of (\ref{pian2}). From (\ref{dH}), they become a sum of at most $2n$ terms of the form in (\ref{product}), with $\mathcal{I}''=\{\mathcal{I},j,k,j',k'\}$ and $n''=n+4$. 

If for any $4 \leq i \leq n $, either $y_i \neq v_1$ or $x_i=y_i=v_1$ holds, as considered in Cases 1 and 2 above, then the degrees of these resulting terms are increased by at least one, \ie $d'' \geq d+1$. 

Else we may assume that $y_4=v_1$ and $ x_4 \neq v_1$. The resulting leading term of degree $d$, as the analogue of (\ref{term_1}), is obtained from replacing one pair of the index $v_1$, \ie the row index of $G_{x_3y_3}$ and the column index of $G_{x_4y_4}$, by the fresh index $k'$ and adding an additional diagonal Green function entry $G_{v_1v_1}$.  Moreover, there exists some $5 \leq i'' \leq n$ such that $x_{i''}=v_1$ and $y_{i''} \neq v_1$ to make sure $v_1$ is an unmatched row index of the original $Q_d^o$ in~\eqref{third_unmatch}, as explained at the beginning of Case 3. We may assume $i''=5$ for simplicity. Then the index $v_1$ from $G_{v_1y_5}$ is again unmatched. We can expand this leading term of degree~$d$ for the third time by applying (\ref{resolvent_identity}) on $G_{v_1y_5}$ and applying cumulant expansions, similarly as in (\ref{pian2}).

We continue this procedure of expanding in the unmatched row index $v_1$ repeatedly for $s$ times, until there is no off-diagonal Green function entry with column index $y_i=v_1$ in the remaining product of the Green function entries $\prod_{i=2s}^{n} G_{x_iy_i}$. Then from Case 1 and Case 2 above, the resulting terms have degrees increased by at least one. The number of iteration $s$ is at most $\nu^{(c)}_1 (\leq n)$, where $\nu^{(c)}_1$ defined in~(\ref{v_ab}) is the number of times the unmatched row index~$v_1$ appears in the column index set of the original~$Q_d^{o}$.

In this way, we expand the original unmatched $Q_d^o$ in terms of finitely many unmatched terms in the form (\ref{product}) of degrees at least $d+1$, as well as the third order cumulant expansion terms generated in the iterations, plus an error $O_{\prec}(N^{-1})$ from the truncation of the cumulant expansion and the diagonal cases. In summary, for any unmatched $Q_d^o \in \mathcal{Q}_d^o$, we write the following expansions for short:
\begin{align}\label{third_unmatch2}
\E[Q_d^o]=\sum_{\substack{ Q^o_{d_1'} \in \mathcal{Q}^o_{d_1'} \\ d_1' \geq d+1}}  \E[Q^o_{d_1'}]+\frac{1}{\sqrt{N}}\sum_{\substack{ Q^o_{d_2'} \in \mathcal{Q}^o_{d_2'} \\ d_2' \geq d}} \E[ Q^o_{d_2'}]+O_{\prec}(\frac{1}{N})\,,
\end{align}
where the number of unmatched terms in the summations above is bounded by $(Cn)^{cn}$, and the number of the Green function entries in the product of each the unmatched term is bounded by $Cn$ for some numerical constants~$C,c>0$.

We finally iterate the expansion in (\ref{third_unmatch2}) for $D-d$ times. Then the unmatched terms in the first summation have degrees at least $D$, and the unmatched terms with $\frac{1}{\sqrt{N}}$ in the second summation have degrees at least $D-1$. Note that the total number of the terms generated in the iteration of the expansions is bounded by $\big((C^Dn)^{c^D n} \big)^D$, and the number of the Green function entries in the product of each term is bounded by $C^{D}n$. We hence obtain from the local law in ~(\ref{G}) that
\begin{align}\label{third_unmatch3}
\E[Q_d^o]=O_{\prec}\big(\Psi^D+\frac{\Psi^{D-1}}{\sqrt{N}}+\frac{1}{N}\big)=O_{\prec}\big(\Psi^D+\frac{1}{N}\big)\,.
\end{align} 
We hence have finished the proof of Proposition \ref{unmatch_lemma}.
\end{proof}

\section{Proof of Proposition \ref{prop_theta}}\label{sec:fourth}
In this section, we prove Proposition \ref{prop_theta}, which is a key ingredient in the proof the Green function comparison theorem, Theorem~\ref{green_comparison}. The special case of Proposition \ref{prop_theta} considering $F(x)=x$ was stated in~(\ref{claim}), which leads to the corresponding Green function comparison theorem for $F(x)=x$ in Proposition~\ref{GCT_mn}. The proof of Proposition \ref{prop_theta} relies on the analogues of Proposition~\ref{lemma_expand_type} (expansion in type-0 terms) and Proposition~\ref{unmatch_lemma} (the negligibility of unmatched terms), as well as the estimate (\ref{img_wigner}) obtained in Proposition \ref{GCT_mn} to bound the resulting type-0 terms.

\begin{proof}[Proof of Proposition \ref{prop_theta}]

We extend the ideas from the proofs of (\ref{claim}) to the setup of Proposition \ref{prop_theta}. Recall $\E[\Theta(t,z_1,z_2)]$ from~\eqref{derivative_gene}, \ie
\begin{align}\label{derivative_gene1}
\E[\Theta(t,z_1,z_2)]\equiv\E [\Theta] =& 
\sum_{\substack{ p+q+1=3\\ p,q \in \N}}^{4}K_{p,q+1}+E_2 +O_{\prec}(N^{-1/2})\,,
\end{align}
with $K_{p,q+1}$ given in~\eqref{K_term} and $E_2$ given in~\eqref{E_1}.

Using the differentiation rules (\ref{dH}) and (\ref{int_1}), each term on the right side of~(\ref{derivative_gene1}) can be written out in terms of an average product of Green function entries with $\Dim$ acting on it and  multiplied by derivatives of $F$. We give one example of a third order term with $p=1,q=1$,
$$\sqrt{N} \frac{1}{N^{3}} \sum_{v, a,b}   \frac{s^{(1,2)}_{ab}}{2} \E \Big[ F'(\X) \Dim \big( G_{va} G_{bv} G_{aa} G_{bb} \big) \Big]\,,$$
and one example of a fourth order terms with $p=2,q=1$,  
$$-\frac{1}{N^{3}} \sum_{v, a,b}   \frac{s^{(2,2)}_{ab}}{4} \E \Big[ F''(\X) \Dim \big(  G_{aa} G_{bb} \big) \Dim \big(  G_{aa} G_{bb} \big) \Big]\,.$$
We point out that the third order terms with $p+q+1=3$ have an additional leading factor $\sqrt{N}$.

To estimate these averaged products of Green function entries multiplied by derivatives of $F$, we introduce the following form of terms generalizing the definition in \eqref{product}:
\begin{align}\label{form_F}
\mathcal{\wt Q}(t,z_1,z_2): \qquad \frac{1}{N^{m}}  \sum_{v_1=1}^N\cdots \sum_{v_m=1}^N  c_{ v_1, \ldots, v_m}
  \E \Big[ F^{(\alpha)}(\X) \prod_{i=1}^{i_0} \Delta \widetilde \Im \Big( \prod_{l=1}^{n_i} G_{x^{(i)}_{l} y^{(i)}_l}  \Big)\Big],
\end{align}
with $\alpha,m, i_0,n_i\in \N$, $F^{(\alpha)}$ be the $\alpha$-th derivative of a smooth function $F$ which has uniformly bounded derivatives, $\Dim: \R^+ \times (\C \setminus \R)^2 \rightarrow \C$ defined in (\ref{dim}), where $\mathcal{I}:=\{v_j\}_{j=1}^m$ is a free summation index set, and the $v_j$'s may also represent $a,b$ from (\ref{K_term}) and (\ref{E_1}). The coefficients $\{c_{\mathcal{I}}:=c_{v_1, \ldots, v_m}\}$ are uniformly bounded complex numbers, and each $x^{(i)}_{l}$ and $y^{(i)}_{l}$ represent some element in the free summation index set $\mathcal{I}$. The total number of the Green function entries in (\ref{form_F}) is then given by 
\begin{align}\label{n_number}
n:=\sum_{i=1}^{i_0} n_i.
\end{align}
We further define the degree of a term in the form~\eqref{form_F} by counting the number of off-diagonal Green function entries, \ie
\begin{align}\label{degree}
d:=\sum_{i=1}^{i_0} \# \big\{ 1 \leq l \leq n_i : x^{(i)}_{l} \neq y^{(i)}_{l}\big\}\,.
\end{align}
In particular, we have $0 \leq d \leq n$. The collection of the terms in the form (\ref{form_F}) of degree $d$ is denoted by $\mathcal{\wt Q}_d \equiv \mathcal{\wt Q}_d(t,z_1,z_2)$. From the definition of $\Dim$ in (\ref{dim}), the local law in (\ref{G}) and the fact that $F$ has bounded derivatives, we have, for any term $\wt Q_d \equiv \wt Q_d(t,z_1,z_2)  \in \mathcal{\wt Q}_d$,
$$|\wt Q_d(t,z_1,z_2)|=O_{\prec}\big(\Psi^d+\frac{1}{N}\big)\,,$$
uniformly in $t\in \R^+$, and $z_1,z_2 \in S$ given in (\ref{ddd}). In the following, we often omit the parameters $t,z_1,z_2$ for notational simplicity. 

\subsection{Unmatched terms $K_{p,q+1}$ in (\ref{K_term})}
In this subsection, we follow the idea in Section \ref{sec:unmatch} to show the negligibility of the terms $K_{p,q+1}$ given in (\ref{K_term}) with unmatched indices as defined next, \cf Proposition~\ref{unmatch_lemma}. Recall Definition \ref{unmatch_def} for unmatched terms of the form in (\ref{product}). 
\begin{definition}
Given any $\wt Q_d \in  \mathcal{\wt Q}_d$ of the form in (\ref{form_F}), let $\nu^{(r)}_j$, $\nu^{(c)}_j$, be the number of times the free summation index $v_j \in \mathcal{I}$ appears in the the row index set $\{x^{(i)}_{l} \}$ and the column index set $\{y^{(i)}_{l} \}$ of the Green function entries, \ie
\begin{align}\label{v_ab}
\nu^{(r)}_j:=\sum_{ i=1}^{i_0} \#\{ 1 \leq l \leq n_i: x^{(i)}_l=v_j \},\qquad \nu^{(c)}_j:=\sum_{ i=1}^{i_0} \#\{1 \leq l \leq n_i: y^{(i)}_l=v_j\}.
\end{align}
Definition \ref{unmatch_def} for unmatched terms can be adapted naturally to the general form given in (\ref{form_F}). Define the set of unmatched summation indices as 
$$\mathcal{I}^o := \{1 \leq j \leq m: \nu^{(r)}_j \neq \nu^{(c)}_j \} \subset \mathcal{I}.$$ 
If $\mathcal{I}^o$ is not empty, then we say $\wt Q_d$ is an unmatched term, denoted by $\wt Q^o_d$. We denote by $\wt{\mathcal{Q}}_d^{o} \subset \wt{\mathcal{Q}}_d$ the collection of unmatched terms in the form (\ref{form_F}) of degree $d$.
\end{definition} 

The combination of the identity (\ref{resolvent_identity}) and the cumulant expansion formula Lemma \ref{cumulant} used previously in the proof of Proposition \ref{unmatch_lemma} still applies similarly to the form in (\ref{form_F}), using that $\{h_{ij}\}$ commute with $\Dim$~given in (\ref{dim}), the differentiation rules (\ref{dH}) and (\ref{int_1}), and the assumption that the function $F$ has bounded derivatives. Therefore, for fixed $D\ge1$ and any unmatched term $\wt Q^o_d  \in \wt{\mathcal{Q}}^o_d$ of the form in~(\ref{form_F}) with fixed $n$ given in (\ref{n_number}),
\begin{align}\label{wt_Q_d}
\E[\wt Q^o_d(t,z_1,z_2)]=O_{\prec}\big(\frac{1}{N}+\Psi^D\big)\,,
\end{align}
holds uniformly in $t\in \R^+$ and $z_1,z_2 \in S$, as in Proposition \ref{unmatch_lemma}.

Now we return to the right side of (\ref{derivative_gene1}). Using (\ref{dH}) and (\ref{int_1}), all the third order expansion terms $K_{p,q+1}$ in (\ref{K_term}) for $p+q+1=3$ can be written out as a sum of finitely many unmatched terms of the form in~(\ref{form_F}) with an extra factor~$\sqrt{N}$ in front, since both the indices $a$ and $b$ appear an odd number of times in the product of the Green function entries.  We hence have from (\ref{wt_Q_d}) that
\begin{equation}\label{L}
|K_{2,1}+K_{1,2}+K_{0,3}|=O_{\prec}(N^{-1/2}+\sqrt{N} \Psi^D)\,.
\end{equation}

Similarly, the fourth order expansion terms $K_{p,q+1}$, $p+q+1=4$, in (\ref{K_term}), with the exception of $K_{2,2}$, can also be written as a finite sum of unmatched terms of the form in (\ref{form_F}), since the number of times the index $a$ (or $b$) appears in the row index set $\{x^{(i)}_{l} \}$ does not agree with the number of times it appears in the column index set $\{y^{(i)}_{l} \}$. We then find from (\ref{wt_Q_d}) that
\begin{align}\label{Kother}
|K_{3,1}+K_{1,3}+K_{0,4}|=O_{\prec}\big(N^{-1}+\Psi^D\big)\,.
\end{align}

It hence suffices to estimate the remaining matched terms $K_{2,2}$ and $E_2$ on the right side of~(\ref{derivative_gene1}) as follows. We first consider $K_{2,2}$ given in (\ref{K_term}), $E_2$ in (\ref{E_1}) can then be estimated similarly. The proof contains two parts: 1) expanding matched terms into type-0 terms defined as below (\cf Proposition~\ref{lemma_expand_type}); 2) estimating the resulting type-0 terms whose degrees are at least two (\cf Lemma~\ref{lemma_trace_wigner}) and the rest type-0 terms of degree zero using the estimate (\ref{img_wigner}) in the edge scaling.

\subsection{Expanding $K_{2,2}$}\label{sec:expand}
We start by $K_{2,2}$ given in (\ref{K_term}), corresponding to the (2,2)-cumulants. Using the differentiation rules (\ref{dH}) and (\ref{int_1}), we first write $K_{2,2}$ as the following sum
\begin{align}\label{K_22}
K_{2,2}=\sum_{k=1}^{8}I_k,
\end{align}
with
\begin{align}\label{general_derivative}
I_1:=&-\frac{1}{2 N^{2} } \sum_{a \neq b}s^{(2,2)}_{ab}   \E \Big[  F'(\X) \Dim \Big( (G_{aa})^2 ( G_{bb})^2 \Big) \Big]; \nonumber\\
I_2:=&-\frac{1}{ N^{2} } \sum_{a \neq b}  s^{(2,2)}_{ab} \E \Big[   F'(\X) \Dim \Big(G_{ab} G_{ba}G_{aa}G_{bb}\Big)\Big];\nonumber\\
I_3:=&-\frac{2}{ N^{2} }   \sum_{a \neq b} s^{(2,2)}_{ab} \E \Big[ F''(\X)  \Dim ( G_{ab})  \Dim \Big( G_{aa} G_{bb} G_{ba} \Big) \Big]; \nonumber\\
 I_4:=&-\frac{1}{2 N^{2} }   \sum_{a \neq b}s^{(2,2)}_{ab} \E \Big[ F''(\X) \Big(    \Dim (G_{aa} G_{bb}) \Big)^2 \Big];\nonumber\\
I_5:=&-\frac{1}{N^{2} }  \sum_{a \neq b} s^{(2,2)}_{ab} \E \Big[ F'''(\X) \Dim (G_{ab}) \Dim (G_{ba}) \Dim \Big(  G_{aa} G_{bb}) \Big) \Big]; \nonumber\\
I_6:=&-\frac{1}{4 N^{2} }  \sum_{a \neq b} s^{(2,2)}_{ab}  \E \Big[ F''(\X)  \Dim \Big( (G_{ab})^2 \Big) \Dim \Big( (G_{ba})^2 \Big)\Big];\nonumber\\
I_7:=&-\frac{1}{2N^{2} } \sum_{a \neq b}  s^{(2,2)}_{ab} \E \Big[F'''(\X) \Big(\Dim (G_{ab}) \Big)^2 \Dim \Big( (G_{ba})^2 \Big)\Big];\nonumber\\
I_8:=&-\frac{1}{4 N^{2} }   \sum_{a \neq b} s^{(2,2)}_{ab}\E \Big[ F''''(\X) \Big( \Dim (G_{ab}) \Big)^2 \Big( \Dim (G_{ba}) \Big)^2  \Big],
\end{align}
where $s^{(2,2)}_{ab}$ $(a \neq b)$ are the (2,2)-cumulants of the rescaled entries $\sqrt{N} h_{ab}$ given in (\ref{cumulant_pq}).

Observe that for the terms given in (\ref{general_derivative}), both indices $a$ and $b$ appear exactly twice as the row index and exactly twice as the column index of a Green function entry. We hence consider the special case of the form in (\ref{form_F}) with the two indices $a,b$ singled out, namely,
\begin{align}\label{form_F_2}
\frac{1}{N^{\#\mathcal{I}+2}}  \sum_{a,b,\mathcal{I}} c_{ a,b, \mathcal{I}} \E \Big[ F^{(\alpha)}(\X) \prod_{i=1}^{i_0} \Delta \widetilde \Im \Big( \prod_{l=1}^{n_i} G_{x^{(i)}_{l} y^{(i)}_l}  \Big)\Big],
\end{align}
where each $x^{(i)}_{l}$ and $y^{(i)}_{l}$ represent $a$, $b$ or some element in the free summation index set $\mathcal{I}=\{v_j\}_{j=1}^m$, and $\{c_{ a,b, \mathcal{I}}\}$ are uniformly bounded complex numbers. The number of Green function entries in the product, denoted by $n$, is given as in (\ref{n_number}). The degree, denoted by~$d$, is given as in~(\ref{degree}) by counting the number of off-diagonal Green function entries in the product.

\begin{definition}\label{def_T_ab}
Given any term of the form in (\ref{form_F_2}), Definition \ref{def_type_AB} for the type-AB, type-A and Type-0 terms of the form in~(\ref{form}) can be adapted naturally. Recall $\nu_j^{(r)},\nu_j^{(c)}$ given in (\ref{v_ab}) for any free summation index $v_j \in \mathcal{I}$. We further define similarly for the special summation indices $a$ and $b$, \ie
\begin{align}
\nu^{(r)}_a:=\sum_{ i=1}^{i_0} \#\{ 1 \leq l \leq n_i: x^{(i)}_l=a \},\qquad \nu^{(c)}_a:=\sum_{ i=1}^{i_0} \#\{1 \leq l \leq n_i: y^{(i)}_l=a\};\nonumber\\
\nu^{(r)}_b:=\sum_{ i=1}^{i_0} \#\{ 1 \leq l \leq n_i: x^{(i)}_l=b \},\qquad \nu^{(c)}_b:=\sum_{ i=1}^{i_0} \#\{1 \leq l \leq n_i: y^{(i)}_l=b\}. \nonumber
\end{align}
If the following two conditions are satisfied,
\begin{enumerate}
\item all the free summation indices in $\{\mathcal{I}\}$ appear once in the row index set $\{x^{(i)}_{l}\}$ and once in the column index set $\{y^{(i)}_{l}\}$ of the Green function entries, \ie $\nu^{(r)}_j=\nu^{(c)}_j=1$ $(1\leq j \leq m)$;
\item  both the special indices $a$ and $b$ appear twice in the row index set $\{x^{(i)}_{l}\}$ and twice in the column index set $\{y^{(i)}_{l}\}$ of the Green function entries, \ie $\nu^{(r)}_a=\nu^{(c)}_a=\nu^{(r)}_b=\nu^{(c)}_b=2$,
\end{enumerate} then such a term is a type-AB term. We denote a type-AB term in the form (\ref{form_F_2}) of degree~$d$ by $T^{AB}_d \equiv T^{AB}_d(t,z_1,z_2)$.  The collection of all the type-AB terms of degree $d$ is denoted by $\mathcal{T}_d^{AB} \equiv \mathcal{T}_d^{AB}(t,z_1,z_2)$.

A type-A term in the form (\ref{form_F_2}) of degree $d$, denoted by $T_d^{A}$, has $\nu^{(r)}_a=\nu^{(c)}_a=2$, and $\nu^{(r)}_b=\nu^{(c)}_b=\nu^{(r)}_j=\nu^{(c)}_j=1$ $(1\leq j \leq m)$. Moreover, a type-0 term, denoted by $T_d$, is of the form (\ref{form_F_2}) of degree $d$ with $\nu^{(r)}_a=\nu^{(c)}_a=\nu^{(r)}_b=\nu^{(c)}_b=\nu^{(r)}_j=\nu^{(c)}_j=1$ $(1\leq j \leq m)$. In addition, the collections of the type-A terms and the type-0 terms of the form in~(\ref{form_F}) of degree $d$ are denoted by $\mathcal{T}_d^{A} \equiv \mathcal{T}_d^{A}(t,z_1,z_2)$ and $\mathcal{T}_d \equiv \mathcal{T}_d(t,z_1,z_2)$, respectively. We finally remark that the index $b$ in a type-A term, as well as both indices $a,b$ in a type-0 term, do not take special roles. We keep them in the notation in order to emphasize the inheritance from the form (\ref{form_F_2}). 
\end{definition}

Under Definition \ref{def_T_ab}, we observe that all the terms given in (\ref{general_derivative}) are type-AB terms in the form~(\ref{form_F_2}) with $\mathcal{I}=\emptyset$ and the coefficients given by $c_{a,b}=s^{(2,2)}_{ab} \delta_{a \neq b}$. In particular, we have that $I_1,I_4 \in \mathcal{T}^{AB}_0$, $I_2,I_3,I_5 \in \mathcal{T}^{AB}_2$, and $I_6, I_7, I_8 \in \mathcal{T}^{AB}_4$. In the following, we use, as in the proof of Proposition~\ref{lemma_expand_type}, the combination of the identity (\ref{resolvent_identity}) and cumulant expansion formula Lemma \ref{cumulant} to eliminate one pair of the index $b$ and also one pair of the index $a$, and thus expand the type-AB terms as linear combinations of type-0 terms up to negligible error.

\begin{lemma}\label{expand_K_1_lemma}
For any fixed $D \in \N$, we have
\begin{align}\label{K_1}
K_{2,2}=&-\frac{s_4}{2} \Big\{ \E \big[  F'(\X) \big( \Dim  ( \ud{G})^4 \big) \big]+\E \big[ F''(\X) \big( \Dim  (\ud{G})^2  \big)^2 \big]\Big\}+ \sum_{\substack{T_d \in \mathcal{T}_d\\2 \leq d <D}}T_d+O_{\prec}\big(\frac{1}{\sqrt{N}}+\Psi^D\big),
\end{align}
uniformly in $t \in \R^+$, and $z_1,z_2 \in S$ given in (\ref{ddd}), with
\begin{align}\label{s_4}
s_4 \equiv s_4(t):=\frac{1}{N^2} \sum_{a \neq b } s^{(2,2)}_{ab}(t),
\end{align}
where $s^{(2,2)}_{ab}(t)$ are the (2,2)-cumulants defined in (\ref{cumulant_pq}) of the time-dependent scaled off-diagonal entries $\sqrt{N} h_{ab}$ given in (\ref{ou_process}). In addition, the number of type-0 terms appearing in the sum in~\eqref{K_1} can be bounded by $(CD)^{cD}$, for some numerical constants $C,c>0$. 
\end{lemma}

\begin{proof}
We first consider $I_1 \in \mathcal{T}^{AB}_0$ given in (\ref{general_derivative}) and expand it into a sum of finitely many type-0 terms. The expansion procedure consists of two steps: 1) eliminating one pair of the index $b$ and expanding $I_1$ in terms of type-A terms; 2) further eliminating one pair of the index $a$ in the resulting type-A terms from 1) and then expanding them in terms of type-0 terms.

Recall the definition of $\Dim $ in (\ref{dim}). Replacing $G_{bb}$ by the relation (\ref{resolvent_identity}) and using the cumulant expansion formula in Lemma \ref{cumulant}, since $\{h_{ij}\}$ commute with $\Dim$, we have
\begin{align}\label{temp0}
I_1=&-\frac{1}{2N^2} \sum_{a,b} s^{(2,2)}_{ab} \E \Big[  F'(\X) \Dim \Big( (G_{aa})^2 G_{bb} \Big(\underline{G}+G_{bb} \underline{HG}-\underline{G}(HG)_{bb} \Big)\Big) \Big]\nonumber\\
=&-\frac{1}{2N^2} \sum_{a,b} s^{(2,2)}_{ab} \E \Big[ F'(\X) \Dim \Big(  (G_{aa})^2  G_{bb} \ud{G} \Big)\Big]-\frac{1}{2N^4} \sum_{a,b,j,k} s^{(2,2)}_{ab}\E \Big[ \frac{ \partial F'(\X)\Dim \Big( (G_{aa})^2 ( G_{bb})^2 G_{jk} \Big)}{\partial h_{jk}}\Big]
\nonumber\\
&+\frac{1}{2N^4}\sum_{a,b,j,k} s^{(2,2)}_{ab}\E \Big[  \frac{\partial F'(\X) \Dim \Big( (G_{aa})^2  G_{bb} G_{jj} G_{kb} \Big)}{\partial h_{kb}}\Big]+O_{\prec}\big(\frac{1}{\sqrt{N}}\big)\,,
\end{align}
where the error is from the truncation of the cumulant expansion, as in the proof of Lemma \ref{lemma_first}. The first term on the right side of~\eqref{temp0} is a type-A term in $\mathcal{T}^{A}_0$ of the form (\ref{form_F_2}) obtained by replacing~$G_{bb}$ with~$\ud{G}$ in the product of the Green function entries. We observe as in (\ref{expand_G_bb}), the leading sub-term from the second term above, corresponding to taking $\frac{\partial}{\partial h_{jk}}$ of $G_{jk}$, is exactly canceled by the leading sub-term from the third term resulting from taking $\frac{\partial}{\partial h_{kb}}$ of $G_{kb}$. Thus using the differentiation rules~(\ref{dH}) and~(\ref{int_1}), the second and third term on the right side of (\ref{temp0}) can be written as a sum of at most ten type-AB terms of the form in (\ref{form_F_2}) with degrees $d' \geq 2$, the number of Green function entries $n'=6$, and $\mathcal{I}'=\{j,k\}$. We denote the finite sum as $\sum_{T^{AB}_{d'} \in \mathcal{T}^{AB}_{d'}; d' \geq 2} T^{AB}_{d'}$, and write
\begin{align}\label{I_1_expand}
I_1=-\frac{1}{2N^2} \sum_{a,b} s^{(2,2)}_{ab} \E \Big[ F'(\X) \Dim \Big(  (G_{aa})^2  G_{bb} \ud{G} \Big)\Big]+\sum_{T^{AB}_{d'} \in \mathcal{T}^{AB}_{d'}; d' \geq 2} T^{AB}_{d'}+O_{\prec}\big(\frac{1}{\sqrt{N}}\big)\,.
\end{align}

Next, we further replace $G_{bb}$ in the first terms on the right side of~\eqref{I_1_expand} by $\ud{G}$ using~(\ref{resolvent_identity}) and the cumulant expansion formula as in (\ref{temp0}) to obtain
\begin{align}\label{temp_ppp}
-\frac{1}{2N^2} \sum_{a,b} s^{(2,2)}_{ab} &\E \Big[ F'(\X) \Dim \Big( (G_{aa})^2  G_{bb} \ud{G} \Big)\Big]=-\frac{1}{2N^2} \sum_{a,b} s^{(2,2)}_{ab} \E \Big[ F'(\X) \Dim \Big( (G_{aa})^2  ( \ud{G} )^2 \Big)\Big]\nonumber\\
&-\frac{1}{2N^4} \sum_{a,b,j,k} s^{(2,2)}_{ab}\E \Big[ \frac{ \partial F'(\X)\Dim \Big( (G_{aa})^2 G_{bb} \ud{G} G_{jk} \Big)}{\partial h_{jk}}\Big]
\nonumber\\
&+\frac{1}{2N^4}\sum_{a,b,j,k} s^{(2,2)}_{ab}\E \Big[  \frac{\partial F'(\X) \Dim \Big( (G_{aa})^2  \ud{G} G_{jj} G_{kb} \Big)}{\partial h_{kb}}\Big]+O_{\prec}\big(\frac{1}{\sqrt{N}}\big)\,.
\end{align}
Observe similarly to above that the leading sub-term from the second term will be canceled exactly by the leading sub-term from the third term.  The remaining sub-terms form a sum of at most ten type-A terms of degrees at least two, denoted as $\sum_{{T^{A}_{d'} \in \mathcal{T}^{A}_{d'}; d' \geq 2}} T^{A}_{d'}$. Combining with (\ref{I_1_expand}), we have
\begin{align}\label{I_1_expand_2}
I_1=-\frac{1}{2N^2} \sum_{a,b} s^{(2,2)}_{ab} \E \Big[ F'(\X) \Dim \Big( (G_{aa})^2  ( \ud{G} )^2 \Big)\Big]+\sum_{\substack{T^{A}_{d'} \in \mathcal{T}^{A}_{d'}\\d' \geq 2}} T^{A}_{d'}+\sum_{\substack{T^{AB}_{d'} \in \mathcal{T}^{AB}_{d'}\\ d' \geq 2}} T^{AB}_{d'}+O_{\prec}\big(\frac{1}{\sqrt{N}}\big)\,.
\end{align}

In general, for an arbitrary type-AB term $T^{AB}_d \in \mathcal{T}^{AB}_d$ of the form (\ref{form_F_2}) with fixed $n$ given in (\ref{n_number}), we extend the arguments as in Step 1 in Subsection~\ref{toy_fourth}, using the differentiation rules~(\ref{dH}) and~(\ref{int_1}) and that $\{h_{ij}\}$ commute with $\Dim$ in (\ref{dim}). We hence obtain the analogue of~(\ref{step_1}),
\begin{align}\label{T_expand_step_1}
T_d^{AB}=\sum_{T_{d}^{A} \in \mathcal{T}_{d}^{A}}T_{d}^{A} +\sum_{ \substack{ T_{d'}^{AB} \in \mathcal{T}_{d'}^{AB}\\ d' \geq d+1}} T_{d'}^{AB}+O_{\prec}\big(\frac{1}{\sqrt{N}}\big)\,,
\end{align}
where the summations above denote a sum of at most two type-A terms of degree $d$ and a sum of at most $6(n+4)$ type-AB terms of degrees not less than $d+1$. The number of the Green function entries in each term above is at most $n+4$. Iterating the expansion procedure (\ref{T_expand_step_1}) $D-d$ times and using the local law in~(\ref{G}), we expand  $T^{AB}_d\in \mathcal{T}^{AB}_d$ as a sum of at most $(6(n+4D))^D$ type-A terms of degrees at least~$d$, up to negligible error. We write for short
\begin{align}\label{step_11A}
T_d^{AB}=\sum_{d\leq d'<D} \sum_{T_{d'}^{A} \in \mathcal{T}_{d'}^{A}}T_{d'}^{A} +O_{\prec}\big(\frac{1}{\sqrt{N}}+\Psi^D\big)\,,
\end{align}
where the number of the Green function entries in each type-A term above is bounded by~$(n+4D)$.

Therefore, from (\ref{I_1_expand_2}) and (\ref{step_11A}), the first term $I_1 \in \mathcal{T}^{AB}_0$ given in (\ref{general_derivative}) can be reduced into the following sum of type-A terms,
\begin{align}\label{temp00}
I_1=&-\frac{1}{2N^2} \sum_{a,b} s^{(2,2)}_{ab} \E \Big[ F'(\X) \Dim \Big( (G_{aa})^2  ( \ud{G} )^2 \Big)\Big]+\sum_{2 \leq d <D} \sum_{T^A_d \in \mathcal{T}^A_d}T^A_d+O_{\prec}\big(\frac{1}{\sqrt{N}}+\Psi^D\big)\,,
\end{align}
where the number of type-A terms above is bounded by $(C_1D)^{c_1D}$ and the number of the Green function entries in each type-A term is bounded by $C_1D$  for some constants $c_1, C_1>0$.

Next, we expand the resulting type-A terms on the right side of (\ref{temp00}) into linear combinations of type-0 terms by further eliminating one pair of the index $a$. In general, for any type-A term $T^{A}_d \in \mathcal{T}^{A}_d$ of the form (\ref{form_F_2}), using similar arguments as in Step 2 in Subsection \ref{toy_fourth}, we obtain the analogue of (\ref{step_22}),
\begin{align}\label{step_22A}
T_d^{A}=\sum_{d \leq d' < D}  \sum_{T_{d'} \in \mathcal{T}_{d'}}T_{d'}+O_{\prec}\big(\frac{1}{\sqrt{N}}+\Psi^D\big)\,,
\end{align}
where the number of these type-0 terms is bounded by $(6(n+4D))^D$, and the number of the Green function entries in each type-0 term is bounded by~$(n+4D)$.

Similar to (\ref{I_1_expand}) and (\ref{I_1_expand_2}), we further eliminate the index $a$ and expand $I_1 \in \mathcal{T}^{AB}_0$ in (\ref{temp00}) into type-0 terms using (\ref{step_22A}), \ie
\begin{align}
I_1=-\frac{s_4}{2}\E \Big[  F'(\X)\Big(  \Dim\, ( \ud{G})^4 \Big) \Big]+\sum_{2 \leq d <D} \sum_{T_d \in \mathcal{T}_d}T_d+O_{\prec}\big(\frac{1}{\sqrt{N}}+\Psi^D\big)\,,
\end{align}
with $s_4$ given in (\ref{s_4}), where the number of the type-0 terms in the sum above is bounded by $(C_2D)^{c_2D}$.

We now turn to the remaining terms in~\eqref{general_derivative}. We only sketch the arguments for sake of brevity.
We start with $I_4\in \mathcal{T}^{AB}_{0}$ in~\eqref{general_derivative}. Similarly to $I_1 \in \mathcal{T}_0^{AB}$, $I_4$ can be expanded as
\begin{align}
I_4=&-\frac{s_4}{2} \E \Big[ F''(\X)  \Big(\Dim  (\ud{G})^2  \Big)^2 \Big]+\sum_{2 \leq d <D} \sum_{T_d \in \mathcal{T}_d}T_d+O_{\prec}\big(\frac{1}{\sqrt{N}}+\Psi^D\big)\,.
\end{align}
Further, using (\ref{step_11A}) and (\ref{step_22A}), $I_2$, $I_3$, $I_5 \in \mathcal{T}^{AB}_2$ from~\eqref{general_derivative} can also be expanded as sums of finitely many type-0 terms of degrees at least two up to negligible error. Moreover, the last three terms $I_6, I_7, I_8 \in \mathcal{T}^{AB}_4$ can be expanded similarly into type-0 terms of degrees  at least four. 

In sum, we have expanded $K_{2,2}$ given in (\ref{K_22}) as a finite sum of type-0 terms,
\begin{align}
K_{2,2}=&-\frac{s_4}{2} \Big\{ \E \Big[  F'(\X) \Big( \Dim  ( \ud{G})^4 \Big) \Big]+\E \Big[ F''(\X) \Big( \Dim  (\ud{G})^2  \Big)^2 \Big]\Big\}+\sum_{2 \leq d <D}\sum_{{T_d \in \mathcal{T}_d}}T_d+O_{\prec}\big(\frac{1}{\sqrt{N}}+\Psi^D\big),\nonumber
\end{align}
where the number of the type-0 terms in the sum above is bounded by $(C_3D)^{c_3D}$ for some $c_3,C_3>0$. This completes the proof of Lemma \ref{expand_K_1_lemma}.
\end{proof}

It then suffices to estimate the resulting type-0 terms on the right side of (\ref{K_1}) in Lemma \ref{expand_K_1_lemma}.

\subsection{Estimate of type-0 terms}\label{sec:degreetwo}
In this subsection, we first show that all the resulting type-0 terms of degrees $d \geq 2$ on the right side of (\ref{K_1}) are bounded by $O_{\prec}(N^{-1/3})$. Using the estimate in (\ref{img_wigner}) and similar arguments as in the proof of (\ref{tracek}) in Lemma~\ref{lemma_trace}, we establish the following analogue of Lemma~\ref{lemma_trace_wigner}.

\begin{lemma}\label{T_2_lemma}
For any type-0 term $T_d \in \mathcal{T}_d$ of the form (\ref{form_F_2}) of degree $d \geq 2$, we have
\begin{align}\label{T_2}
|T_d(t,z_1,z_2)| =O_{\prec}\big(N^{-1/3}\big)\,,
\end{align}
uniformly in $t \in \R^+$, $z_1,z_2 \in S_{\mathrm{edge}}$ given in (\ref{S_edge}).
\end{lemma}
 \begin{proof}
Given any type-0 term $T_d \in \mathcal{T}_d$ of the form (\ref{form_F_2}), we no longer emphasize the indices $a$,$b$ for notational simplicity. We then write $T_d$ from the definition of $\Dim$ in (\ref{dim}) as
\begin{equation*}\label{temp_use_td}
 \E \Big[ F^{(\alpha)}(\X) \frac{1}{N^{\# \mathcal I} } \sum_{\mathcal{I}} c_{\mathcal{I}}  \prod_{i=1}^{i_0}  \Big( \prod_{l=1}^{n_i}G_{x^{(i)}_{l} y^{(i)}_l} (t,z_1)-\prod_{l=1}^{n_i} G_{x^{(i)}_{l} y^{(i)}_l}(t,\overline{z_1})-\prod_{l=1}^{n_i}G_{x^{(i)}_{l} y^{(i)}_l}(t,z_2) +\prod_{l=1}^{n_i} G_{x^{(i)}_{l} y^{(i)}_l}(t,\overline{z_2})\Big],
\end{equation*}
with $t\geq 0$, $z_1,z_2 \in S_{\mathrm{edge}}$, and $\alpha,m,i_0,n_i \in \N$, where each summation index $v_j \in \mathcal{I}:=\{v_j\}_{j=1}^m$ appears exactly once in the row index set $\{x^{(i)}_{l}\}$ and once in the column index set $\{x^{(i)}_{l}\}$ of the Green function entries. In particular, we have $\#\mathcal{I}=n=\sum_{i=1}^{i_0}n_i$. For $1\leq j \leq m$, if there exist $x^{(i)}_{l}=y^{(i)}_{l}=v_j$, then we say $v_j$ is isolated. For any $1\leq j \neq j' \leq m$, if there exist $1 \leq i \leq i_0, 1 \leq l \leq n_i$ such that either $x^{(i)}_{l}=v_j$, $y^{(i)}_{l}=v_{j'}$ or $y^{(i)}_{l}=v_j$, $x^{(i)}_{l}=v_{j'}$, then we say that $v_j$ and $v_{j'}$ are connected indices.  We then write out $T_d$ as a linear combination of the terms in the following form, which are rearranged using clusters of connected indices, denoted by $\{v^{(q)}_1, \ldots, v_{l_q}^{(q)}\}_{q}$,
\begin{align}
(**):= \E \Big[ F^{(\alpha)}(\X) \frac{1}{N^{\# \mathcal I} } \sum_{\mathcal{I}} c_{\mathcal{I}}  \prod_{q} \Big( G_{v^{(q)}_1v^{(q)}_2}(t,z_{1}^{(q)}) G_{v^{(q)}_2v^{(q)}_3}(t,z_{2}^{(q)})   \cdots G_{v^{(q)}_{l_q}v^{(q)}_1}(t,z_{l_q}^{(q)}) \Big)\Big],
\end{align}
where $\sum_{q}l_q=n$, $z_l^{(q)}$ for any $q$ and $1 \leq l \leq l_q$ takes the values $z_1,\overline{z_1}, z_2$, or $\overline{z_2}$. Because the degree $d\ge 2$, there exists at least one cluster of connected indices such that $l_q \geq 2$. We may assume that $q=1$. Recall that the coefficients $\{c_{\mathcal{I}}\}$ are uniformly bounded and that the function $F$ has bounded derivatives. Then using the local law in (\ref{G}) and the properties of stochastic domination in Lemma \ref{dominant}, we have that
$$|(**)| \prec \E \Big[\frac{1}{N^{l_1}}  \sum_{v^{(1)}_1, \ldots, v^{(1)}_{l_1}=1}^N\Big|  G_{v^{(1)}_1v^{(1)}_2}(t,z_{1}^{(1)}) G_{v^{(1)}_2v^{(1)}_3}(t,z_{2}^{(1)})   \cdots G_{v^{(1)}_{l_1}v^{(1)}_1}(t,z_{l_1}^{(1)})  \Big| \Big].$$
In combination with Young's inequality and the Ward identity (\ref{ward}), we find, similarly to~\eqref{traces}, that
\begin{align}
|(**)| \prec & \frac{\E [\Im m_N(t,z_1)]}{(N \eta)^{l_1-1}} +\frac{\E[\Im m_N(t,z_2)]}{(N \eta)^{l_1-1}}, \qquad l_1 \geq 2, ~z_1,z_2 \in S_{\mathrm{edge}}.
\end{align}
Together with the estimate (\ref{img_wigner}) on $\E[\Im m_{N}(t,z)]$ in the edge scaling and the fact that $\eta \geq N^{-1+\epsilon}$, we obtain the estimate in (\ref{T_2}).
\end{proof}

Applying Lemma~\ref{T_2_lemma} to~\eqref{K_1}, we find that
\begin{align}\label{K_11}
K_{2,2}=&-\frac{s_4}{2} \Big\{ \E \Big[  F'(\X) \Big( \Dim  ( \ud{G})^4 \Big) \Big]+\E \Big[ F''(\X) \Big( \Dim  (\ud{G})^2  \Big)^2 \Big]\Big\}+O_{\prec}(N^{-1/3}+\Psi^D),
\end{align}
uniformly in $t\geq 0$ and $z_1,z_2 \in S_{\mathrm{edge}}$. It then suffices to estimate the remaining type-0 terms of degree zero on the right side of (\ref{K_11}). Using the definition of $\Dim$ in (\ref{dim}), the estimate in (\ref{img_wigner}) of $\E [\Im \ud{G}(t,z)]$ for $z \in S_{\mathrm{edge}}$ and $t \geq 0$, the properties of stochastic domination Lemma \ref{dominant} and that the function $F$ has bounded derivatives, we conclude, for any fixed $D\ge 1$, that
\begin{align}\label{Kone}
|K_{2,2}|=O_{\prec}(N^{-1/3+\epsilon}+\Psi^D)\,,
\end{align}
uniformly in $t \in \R^+$ and $z_1, z_2 \in S_{\mathrm{edge}}$.

\subsection{Estimate of $E_2$}\label{sec:diagonal}
In this subsection, we estimate the second order term $E_2$ given in (\ref{E_1}) similarly as $K_{2,2}$. Using (\ref{int_1}) and (\ref{dH}), we write $E_2$ as
\begin{align}
E_2=&-\frac{1}{2 N} \sum_{a=1}^N ( s^{(2)}_{aa} -1)\E \Big[F'(\X)  \Dim (G_{aa})^2 \Big]-\frac{1}{2 N} \sum_{a=1}^N ( s^{(2)}_{aa} -1)\E \Big[F''(\X)  \Big(\Dim (G_{aa}) \Big)^2 \Big].
\end{align}
Observe that the above two terms are both type-A terms in $\mathcal{T}^A_0$ of the form (\ref{form_F_2}), where the index $b$ no longer plays a special role. Using the combination of the identity (\ref{resolvent_identity}) and the cumulant expansion formula, we expand $E_2$ into a sum of finitely many type-0 terms, similarly to~(\ref{K_1}). That is,
\begin{align}\label{E_2_step1}
E_2=&-\frac{s_2 -1}{2} \Big\{ \E \Big[F'(\X)  \Dim (\ud G)^2 \Big]+\E \Big[F''(\X)  \Big(\Dim (\ud G) \Big)^2 \Big] \Big\}+\sum_{\substack{T_{d} \in \mathcal{T}_{d}\\2 \leq d <D}}T_{d}+O_{\prec}\Big(\frac{1}{\sqrt{N}}+\Psi^D\Big),
\end{align}
uniformly in $t\geq 0$ and $z_1,z_2 \in S$, with 
\begin{align}\label{s_2}
s_2 \equiv s_2(t):=\frac{1}{N} \sum_{a=1}^N  s^{(2)}_{aa}(t),
\end{align}
where $s^{(2,2)}_{aa}(t)$ are the second order cumulants given in (\ref{cumulant_k}) of the time-dependent scaled entries $\sqrt{N} h_{aa}$. Moreover, the number of the type-0 terms in the summation on the right side of (\ref{E_2_step1}) is bounded by $(CD)^{cD}$ for some numerical constants $c,C>0$. 

Similarly to (\ref{Kone}), we conclude from Lemma \ref{T_2_lemma} and the estimate in (\ref{img_wigner}) that, for any $D\geq 1$,
\begin{align}\label{E}
|E_2|=O_{\prec}\big(N^{-1/3+\epsilon}+\Psi^D\big)\,,
\end{align}
uniformly in $t \in \R^+$ and $z_1, z_2 \in S_{\mathrm{edge}}$.

Plugging (\ref{Kone}), (\ref{E}), (\ref{L}), and (\ref{Kother}) into (\ref{derivative_gene}) and by choosing $D \geq \frac{1}{\epsilon}$ with $\epsilon>0$ as in (\ref{control_eq}), we hence finish the proof of Proposition \ref{prop_theta}.
\end{proof}

\section{Real symmetric Wigner matrices}\label{sec: real}
In this section, we prove the Green function comparison theorem, Theorem \ref{green_comparison}, for real Wigner matrices, using similar ideas as for the complex Hermitian case. To simplify the discussion, we will only address the differences.

Consider the real-valued matrix Ornstein-Uhlenbeck process $\big(h_{ab}(t)\big)_{a,b=1}^N$: 
\begin{equation}\label{ou_process_goe}
\dd h_{ab}(t)=\sqrt{\frac{1+\delta_{ab}}{N}} \dd \beta_{ab}(t)- \frac{1}{2} h_{ab}(t) \dd t, \qquad h_{ab}(0)=(H_N)_{ab}\,,
\end{equation}
 where $\big(\beta_{ab}(t)\big)_{a \leq b }$ are independent real standard Brownian motions with $\beta_{ba}(t)={\beta_{ab}(t)}$. The initial condition $H_N$ is a real symmetric Wigner matrix satisfying Assumption \ref{assump}. In distribution this is equivalent to writing
 \begin{equation}\label{sum_H_goe}
H(t)=\mathrm{e}^{-\frac{t}{2}}H_N +\sqrt{1-\mathrm{e}^{-t}} \mathrm{GOE}_N\,, \qquad t \in \R^+.
 \end{equation}
 
 As the analogue of (\ref{dH}), we have a new differentiation rule for the Green function entry of a real symmetric matrix,
\begin{equation}\label{dH_real}
\frac{\partial G_{ij}}{ \partial h_{ab}}=-\frac{G_{ia} G_{bj}+G_{ib} G_{aj}}{1+\delta_{ab}}.
\end{equation}
Then using Ito's formula similarly to (\ref{derivative2}), we obtain
\begin{align}\label{derivative_goe}
\dd G_{ij}(t,z)=\,\dd M_{ij}+\Theta_{ij} \dd t\,,
\end{align}
where the diffusion term $\dd M_{ij}:=-\frac{1}{\sqrt{N}} \sum_{a \leq b}\frac{1}{\sqrt{1+\delta_{ab}}}  \Big( G_{ia}G_{bj} +G_{ib}G_{aj}\Big)\dd \beta_{ab}$, and the drift term
$$\Theta_{ij}:=\frac{1}{2}\sum_{a, b}  h_{ab}  G_{ia} G_{bj}  +\frac{1}{2N} \sum_{a,b} \Big( 2 G_{ia} G_{ab} G_{bj}+ G_{ib} G_{bj} G_{aa}+  G_{ia} G_{aj} G_{bb} \Big)\,.$$

Recall $F$ in \eqref{q_function} and $\X$ in (\ref{X}). Applying Ito's formula on $F(\X)$ and using (\ref{derivative_goe}), we derive the dynamics of $F(\X)$ in the real symmetric case,
$$\dd F(\X)=\dd M+\Theta \dd t\,,$$
where the diffusion term $\dd M$ yields a martingale after integration, see Remark \ref{diffusion_mart}, and the drift term is given by (we omit the parameters $t$ and $2+x+\ii \eta$ of the following Green function entries)
\begin{align}\label{theta_real}
\Theta=& F'(\X) \Im \intkappa  \frac{1}{2} \sum_{i, a, b} \Big(  h_{ab}G_{ia} G_{bi}+\frac{2}{N}G_{ia} G_{ab} G_{bi}+\frac{1}{N} G_{ib} G_{bi} G_{aa}+\frac{1}{N}  G_{ia} G_{ai} G_{bb}   \Big) \dd x\nonumber\\
&+F''(\X) \frac{1}{N} \sum_{i,j} \sum_{a, b} \Big( \Im \intkappa G_{ia} G_{bi} \dd x \Big)\Big( \Im \intkappa G_{jb} G_{aj} \dd x \Big)\nonumber\\
=& \frac{1}{2} \sum_{a, b}  h_{ab}\Big( F'(\X)  \Delta \Im G_{ba} \Big)+\frac{1}{N}   \sum_{a, b} \Big( F'(\X)\Delta \Im (G_{aa} G_{bb})\Big) +\frac{1}{N}   \sum_{a, b} \Big( F'(\X)\Delta \Im (G_{ab})^2 )\Big)\nonumber\\
&+\frac{1}{N}   \sum_{a, b} \Big( F''(\X) ( \Delta \Im G_{ab} )( \Delta \Im G_{ba}) \Big) \,.
\end{align}
where we abbreviate, for any function $P\,:\,\R^+ \times \C \setminus \R \longrightarrow\C$,
\begin{align}\label{delta_img}
\Delta \Im P \equiv (\Delta \Im P)(t,z_1,z_2):=\Im P(t,z_2)-\Im P(t,z_1),
\end{align}
with $t \in \R^+$, $z_1=2+\kappa_1+\ii \eta,z_2=2+\kappa_2+\ii \eta \in S_{\mathrm{edge}}$, as in (\ref{le zs}). In fact, comparing with the drift term in (\ref{theta}) for complex Hermitian matrices, the notation $\widetilde{\Im}$ in (\ref{tilde}) is replaced with the imaginary part $\Im$. This is because $\{h_{ab}\}$ commute with taking the imaginary part, and the Green function of a real symmetric matrix satisfies
\begin{align}\label{symmetry_g}
G_{ij}(z)=G_{ji}(z), \qquad z \in \C \setminus \R.
\end{align}
Moreover, using (\ref{dH_real}), it is easy to find the analogous differentiation rule to (\ref{int_1}),
 \begin{align}\label{int_1_real}
\frac{ \partial F'(\X) }{ \partial h_{ab} }=-\frac{2}{1+\delta_{ab}} F''(\X) \sum_{i=1}^N \Im  \Big( \intkappa G_{ia} G_{bi}(2+x+\ii \eta) \dd x \Big)=-\frac{2}{1+\delta_{ab}} F''(\X) \Delta \Im G_{ab}\,,
\end{align}
with $\Delta \Im$ given in (\ref{delta_img}).

Next, we return to the right side of (\ref{theta_real}). Applying the real cumulant expansion formula in Lemma~\ref{cumulant} for the independent entries $\{h_{ab}\}_{a\leq b}$ in the first term up to the fourth order and using the differentiation rules (\ref{dH_real}) and (\ref{int_1_real}), the second order terms in the cumulant expansions are canceled exactly by the last three terms on the right side of~(\ref{theta_real}). We hence obtain the real analogue of~(\ref{derivative_gene}),
\begin{align}\label{drift term beta1}
\E [\Theta] =& \frac{1}{2 N} \sum_{a=1}^N ( s^{(2)}_{aa} -2)\E \Big[\frac{ \partial F'(\X)  \Delta \Im G_{aa}}{ \partial h_{aa} }\Big]  +\frac{1}{4N^{3/2} }  \sum_{a, b} s^{(3)}_{ab} \E \Big[\frac{ \partial^2 F'(\X) \Delta \Im G_{ba} }{\partial h^2_{ab}}\Big]\nonumber\\
&+\frac{1}{12 N^{2} }   \sum_{a, b} s^{(4)}_{ab} \E \Big[\frac{ \partial^3 F'(\X)  \Delta \Im G_{ba}}{ \partial h^3_{ab} }\Big]+O_{\prec}\big(\frac{1}{\sqrt{N}}\big)\,,
\end{align}
where the error $O_{\prec}(\frac{1}{\sqrt{N}})$ is from the truncation of the cumulant expansion, and $s^{(k)}_{ab}$ is the $k$-th cumulant defined in (\ref{cumulant_k}) of the rescaled entries $\sqrt{N}h_{ab}$. 

We now claim that Proposition \ref{prop_theta} holds true in the real case, which leads to Theorem \ref{green_comparison} for $\beta=1$. The arguments in the complex case discussed before can be applied similarly, using the modified differentiation rules (\ref{dH_real}) and~(\ref{int_1_real}), and the real cumulant expansion formula in Lemma~\ref{cumulant}.

To simplify the statement, we only consider the simplest version of the Green function comparison theorem for $F(x)=x$, as proved in Proposition \ref{GCT_mn} for complex Hermitian Wigner matrices. The Green function comparison theorem for general functions $F$ can be proved using the same idea, following the arguments in Section \ref{sec:fourth} for the complex Hermitian case.

Applying (\ref{derivative_goe}) to the time dependent normalized trace of the Green function, $m_N(t,z)$, we find the real analogue of (\ref{diff_eq_1}), \ie
\begin{align}\label{diff_eq_1_goe}
\dd (m_N(t,z))=&\dd M_0+ \Theta_0 \dd t\,,
\end{align}
with the diffusion term $\dd M_0:=\frac{1}{N} \sum_{v=1}^N \dd M_{vv}$ which yields a martingale term after integration; see Remark \ref{diffusion_mart}, and the drift term $\Theta_0 \dd t:=\frac{1}{N} \sum_{v=1}^N\Theta_{vv} \dd t$. Applying the real cumulant expansion formula as in (\ref{drift term beta1}), the drift term satisfies the real analogue of (\ref{step0}), \ie
\begin{align}\label{derivative_gene_goe}
\E [\Theta_0] =& \frac{1}{2 N^2} \sum_{v,a} ( s^{(2)}_{aa} -2)\E \Big[\frac{ \partial (G_{va} G_{bv})}{ \partial h_{aa} }\Big]  +\frac{1}{4N^{5/2} }  \sum_{v,a, b} s^{(3)}_{ab} \E \Big[\frac{ \partial^2 (G_{va} G_{bv}) }{\partial h^2_{ab}}\Big]\nonumber\\
&+\frac{1}{12 N^{3} }   \sum_{v,a, b} s^{(4)}_{ab} \E \Big[\frac{ \partial^3 (G_{va} G_{bv})}{ \partial h^3_{ab} }\Big]  +O_{\prec}(\frac{1}{\sqrt{N}})=:J_2+J_3+J_4+O_{\prec}\big(\frac{1}{\sqrt{N}}\big)\,.
\end{align}

It then suffices to prove the estimate (\ref{claim}) in the real symmetric case. Using (\ref{dH_real}), the terms $J_2,J_3,J_4$ above can be written out again in the form (\ref{product}). The degree of a term in the form (\ref{product}) is defined as in~(\ref{degree_0}). We recall from (\ref{symmetry_g}) that the row and column index of a Green function entry can be switched. 

Following the idea from complex Hermitian case, the proof of (\ref{claim}) consists of three steps: 1) the third order terms from $J_3$ are unmatched and thus negligible (\cf Proposition \ref{unmatch_lemma}); 2) expanding the fourth order terms from $J_4$ (as well as the second order terms in $J_2$) as linear combinations of type-0 terms of degrees at least two up to arbitrary order (\cf Proposition \ref{lemma_expand_type}); 3) estimating the resulting type-0 terms in 2) of degrees at least two (\cf Lemma \ref{lemma_trace_wigner}).

We start with the first step. Recall Definition \ref{unmatch_def} for unmatched terms in the complex Hermitian case. Because of (\ref{symmetry_g}), we can ignore the difference from the row and column index of a Green function entry of a real symmetric matrix. 
\begin{definition}[Terms with unmatched indices in the real case]\label{unmatch_def_real}
Given any term, denoted by $Q_d$, of the form (\ref{product}) of degree $d$, let $\nu_j$ be the number of times the free summation index $v_j \in \mathcal{I}$ appears as the row or column index in the product of the Green function entries, \ie
\begin{align}\label{nu_number}
\nu_j:=\#\{ 1 \leq i \leq n\,:\, x_i=v_j \} +\#\{ 1 \leq i \leq n\,:\, y_i=v_j\}\,, \qquad 1\leq j\leq m\,.
\end{align}
We define the set of the unmatched summation indices as 
$$\mathcal{I}^o := \{1 \leq j \leq m: \nu_j \mbox{ is odd }\} \subset \mathcal{I}.$$ 
Note that $\#\mathcal{I}^o$ is even. If $\mathcal{I}_0 = \emptyset$, then we say $Q_d$ is matched. Otherwise, $Q_d$ is an unmatched term, denoted by $Q^o_d$. The collection of the unmatched terms in the form (\ref{product}) of degree $d$ is denoted by $\mathcal{Q}_d^o$. 
\end{definition}

Then the third order terms from $J_3$ on the right side of (\ref{derivative_gene_goe}) are of the form (\ref{product}) with an extra $\sqrt{N}$ in front and are unmatched with $\nu_a=\nu_b=3$ defined in~(\ref{nu_ab_goe}) below. Following the arguments in Section~\ref{sec:unmatch}, using the relation (\ref{resolvent_identity}), the real cumulant expansion formula, and the new differentiation rule of the Green function entry ~(\ref{dH_real}), we observe a similar cancellation to the first order and then expand a unmatched term of the form ~(\ref{product}) iteratively and prove that Proposition \ref{unmatch_lemma} holds true in the real symmetric case. Therefore, we have 
\begin{align}\label{J_3}
|J_3|=O_{\prec}\big(N^{-1/2}+\sqrt{N} \Psi^D\big)\,.
\end{align}

Next, in the second step, we expand the remaining terms of the form (\ref{product}) from $J_2$ and $J_4$ that are matched. Recall a special case of matched terms as in (\ref{form}) with two summation indices $a,b$ singled out and Definition \ref{def_type_AB} for type-AB, type-A, type-0 terms in the complex case.
\begin{definition}[Type-AB terms, type-A terms, type-0 terms]
Given any term of the form in (\ref{form}) of degree $d$ with two special indices $a$ and $b$, recall $\nu_j$ in (\ref{nu_number}) for any $v_j \in \mathcal{I}$ and define similarly 
\begin{align}\label{nu_ab_goe}
\nu_a:&=\#\{ 1 \leq i \leq n\,:\, x_i=a \}+\#\{ 1 \leq i \leq n\,:\, y_i=a\}\,,\nonumber\\ 
\nu_b:&=\#\{ 1 \leq i \leq n\,:\, x_i=b \}+\#\{ 1 \leq i \leq n\,:\,  y_i=b\}\;.
\end{align} 
If for any $1\leq j\leq m$, $\nu_j=2$ and $\nu_a=\nu_b=4$, then such a term is a type-AB term. A type-A term has $\nu_a=4$, and $\nu_b=\nu_j=2$ ($1\leq j\leq m$). Finally, a type-0 term is defined to be in the form (\ref{form}) with $\nu_a=\nu_b=\nu_j=2$ ($1\leq j\leq m$). The collection of the type-AB, type-A, type-0 terms of degree $d$ is denoted by $\mathcal{P}_d^{AB}$, $\mathcal{P}_d^{A}$, and $\mathcal{P}_d$, respectively.
\end{definition}
Following the arguments in Section \ref{sec:match}, using the relations (\ref{resolvent_identity}) and (\ref{dH_real}), and the real cumulant expansion formula, we expand any type-AB (or type-A) term iteratively and prove that Proposition \ref{lemma_expand_type} holds true in the real symmetric case. Therefore, expanding the type-AB terms from $J_2$ and the type-A terms from $J_4$ and then combining with~(\ref{J_3}), we write~(\ref{derivative_gene_goe}) as
\begin{align}\label{theta_0_goe}
\E[\Theta_0(t,z)]=\sum_{ \substack{P_{d} \in \mathcal{P}_{d} \\ 2 \leq d \leq D-1}} \E [ P_d(t, z) ]+O_{\prec}\big(\frac{1}{\sqrt{N}}+\Psi^D\big)\,,
\end{align}
where the summation on the right side above denotes a linear combination of at most $(CD)^{cD}$ type-0 terms of degrees at least two, for some numerical constants $C,c$.

In the last step, we aim to show that any type-0 term of degree $d \geq 2$ can be bounded by $O_{\prec}(N^{-1/3})$ for real symmetric Wigner matrices, as in Lemma \ref{lemma_trace_wigner}. This reduces to prove Lemma \ref{lemma_trace} for the GOE.
\begin{lemma}\label{lemma_traceimg_goe}
For any $z \in S_{\mathrm{edge}}(\epsilon,C_0)$ given in (\ref{S_edge}) and $t \geq 0$, we have the following uniform estimate:
\begin{equation}\label{traceimg_goe}
\frac{1}{N} \E^{\mathrm{GOE}} \Big[ \Im \Tr G (z)\Big]=O\big(N^{-1/3+\epsilon}\big)\,.
\end{equation}
\end{lemma}

The corresponding estimate (\ref{tracek}) of the type-0 terms of degree $d \geq 2$ considering the GOE follows directly from Lemma~\ref{lemma_traceimg_goe}. Following the iterative comparison idea in the proof of Lemma \ref{lemma_trace_wigner}, one proves Lemma \ref{lemma_trace_wigner} similarly in the real case, using (\ref{dH_real}), (\ref{derivative_goe}) and the real cumulant expansion formula. Therefore, we obtain from (\ref{theta_0_goe}) that (\ref{claim}) holds true in the real case and we hence finish the proof of Proposition \ref{GCT_mn} for real Wigner matrices.

\begin{proof}[Proof of Lemma \ref{lemma_traceimg_goe}]
The proof is similar to that of Lemma \ref{lemma_trace}. For the one-point correlation function of the GOE and the corresponding diagonal kernel $K_{N,1}$, we refer to \cite{AGZ} and~\cite{metha}. From Chapter 3.9 in \cite{AGZ}, we write
\begin{align}\label{one_point_real}
K_{N, 1}(x,x)=&K_{N,2}(x,x)+\frac{\sqrt{N}}{4}\phi_{N-1}(x) \Big( \int_{-\infty}^{\infty} \mathrm{sgn}(x-t) \phi_N(t) \dd t \Big)+\frac{1}{2 I_{N-1}} \phi_{N-1}(x) \one_{N=2m+1},
\end{align}
where $K_{N,2}(x,x)$ is the one-point correlation function for the GUE given by (\ref{kernel_xx}), $\{\phi_k\}$ are the Hermite functions in~\eqref{le hermite functions}, and we use $\beta=1,2$ to denote the symmetry class. Moreover, we set
\begin{align}\label{int_even}
I_{2m}:=\int_0^{\infty} \phi_{2m}(t) \dd t=\frac{1}{2}\int_\R \phi_{2m}(t) \dd t=2^{-1/4} \pi^{1/4} \sqrt{\frac{(2m)!}{2^{2m} (m!)^2}}\sim m^{-1/4},
\end{align}
by the Stirling approximation; see Proposition 3.9.28 in \cite{AGZ}. In addition, from Lemma 1 in \cite{goe_kernel}, we have
\begin{align}\label{int_odd}
I_{2m+1}:=\int_0^{\infty} \phi_{2m+1}(t) \dd t=O(m^{-1/4})\,.
\end{align}
Note that the trace identity for the kernel $K_{N,1}$ still holds as in (\ref{eq_1}). Next, we change the variable as in (\ref{edge}) and define
\begin{align}\label{K_edge_goe}
K^{\mathrm{edge}}_{N, 1}(x,x):=\frac{1}{N^{1/6}} K_{N,1}\Big( 2\sqrt{N}+\frac{x}{N^{1/6}},2\sqrt{N}+\frac{x}{N^{1/6}} \Big)\,.
\end{align}
From Theorem 1.1 in \cite{convergence_kernel}, as the real analogue of Theorem \ref{kernel_diff}, for any $L_0\in \R$, we have, in the limit of large $N$, that
\begin{align}\label{edge_kernel_goe}
K^{\mathrm{edge}}_{N, 1}(x,x)=&K_{\mathrm{airy}}(x,x)+\frac{1}{2} \Ai(x) \int^{x}_{-\infty} \Ai(t) \dd t+o(1)\,,
\end{align}
uniformly in $x \in [L_0,\infty)$. In addition, the right side of~\eqref{edge_kernel_goe} is uniformly bounded for $x > L_0$; see Chapter 3 in \cite{AGZ} for a reference. Now we are ready to estimate
\begin{align}\label{sum_0_goe}
\frac{1}{N} \E^{\mathrm{GOE}} \Big[ \Im \Tr G (z)\Big]=\frac{N \eta}{N^2} \E^{\mathrm{GOE}} \Big[ \sum_{j=1}^N \frac{1}{|\lambda_j-z|^2}\Big]=\frac{N\eta}{N^\frac{2}{3}}\int_{\R} \frac{ K^{\mathrm{edge}}_{N,1}(x,x)}{|x-N^{2/3}\kappa-\ii N^{2/3} \eta|^2} \dd x\,,
\end{align}
for $z=2+\kappa+\ii\eta\in S_{\mathrm{edge}}$, in a similar way as in the proof of Lemma \ref{lemma_trace}. Note that (\ref{sum_1}) and (\ref{sum_2}) still hold true for the GOE. We will focus on the regime $-N^{2/3} < x\le L_0$, for some fixed $L_0<0$. Recalling the estimate (\ref{sum_3}) for the GUE, it suffices to prove, for any $x\in (-N^{2/3},L_0]$, that
\begin{align}\label{diff_even}
\Big| K^{\mathrm{edge}}_{N, 1}(x,x)-K^{\mathrm{edge}}_{N,2}(x,x)\Big|=O(1)\,,
\end{align}
which then leads to
\begin{align}\label{diff_even2}
\frac{1}{N^{2/3}} \int_{-N^{2/3}}^{L_0} \frac{K^{\mathrm{edge}}_{N, 1}(x,x)}{|x-N^{2/3} \kappa+\ii N^{2/3} \eta|^2} \dd x=O\Big(\frac{1}{N^{\frac{4}{3}-\epsilon} \eta}\Big)\,.
\end{align}
We hence obtain (\ref{traceimg_goe}) for the GOE. In order to prove (\ref{diff_even}), we split into two cases below and follow ideas from \cite{goe_kernel}.

{\bf Case 1: $N$ is even.} Let $N=2m$ and the last term in (\ref{one_point_real}) is vanishing. Since $\phi_{N}$ is even, we write
\begin{align}\label{12diff}
K^{\mathrm{edge}}_{N, 1}(x,x)=&K^{\mathrm{edge}}_{N,2}(x,x)+\frac{1}{2} N^{1/3} \phi_{N-1}(y)  \int_{0}^{y} \phi_{N}(t) \dd t\,,
\end{align}
where we set for simplicity,
\begin{align}\label{y_bound}
y=2\sqrt{N}+\frac{x}{N^{1/6}}\,, \qquad \mbox{with} \quad -N^{2/3}<x\le L_0\,,
\end{align}
which implies that $\sqrt{N}< y < 2 \sqrt{N}+L_0N^{-1/6}$. From \cite{goe_kernel} and references therein, we have the following asymptotic formula of $\phi_N(t)$. In the domain 
\begin{align}\label{domian_hermite}
|t| \leq \sqrt{2} \Big((2N+1)^{1/2}-(2N+1)^{-1/6} \Big)\,,
\end{align}
we have as $N \rightarrow \infty$,
\begin{align}\label{hermite_expand}
\phi_N(t)=A_N(t)+O\Big(N^{1/2} (4N+2-t^2)^{-7/4}\Big)\,,
\end{align}
with 
\begin{align}\label{A_N}
 A_N(t):=\sqrt{\frac{2}{\pi}} (4N+2-t^2)^{-1/4}\cos\Big( \frac{(2N+1)(2 \alpha_N-\sin 2 \alpha_N) -\pi}{4}\Big)\,,
\end{align}
and $\alpha_N:=\arccos(t(4N+2)^{-1/2})$. We choose $L_0<0$ in (\ref{y_bound}) sufficiently small so that the upper bound~$y$ of the integral in (\ref{12diff}) satisfies (\ref{domian_hermite}). Thus we have from (\ref{hermite_expand}) that
\begin{align}\label{tues}
\int_{0}^{y} \phi_{N}(t) \dd t  =&\int_{0}^{y} A_{N}(t) \dd t +O\Big( \sqrt{N} \int_{0}^{y}  (4N+2-t^2)^{-7/4} \dd t \Big)=\int_{0}^{y} A_{N}(t) \dd t+O(N^{-1/4})\,.
 \end{align}
Integrating $A_N$ given in (\ref{A_N}) and using integration by parts, it was shown in (14) in \cite{goe_kernel} that
 \begin{align}
 \Big|\int_{0}^{y} A_{N}(t) \dd t \Big| \leq C (4N+2-y^2)^{-3/4}=O(N^{-1/4})\,,
 \end{align}
 with $\sqrt{N}< y < 2 \sqrt{N}+L_0N^{-1/6}$. Thus we have from (\ref{tues}) that $\Big|\int_{0}^{y} \phi_{N}(t) \dd t  \Big|=O(N^{-1/4}),$ for $y$ given in (\ref{y_bound}). Combining with (\ref{upper_bound_hermite}), the estimate (\ref{diff_even}) then follows from (\ref{12diff}).

{\bf Case 2: $N$ is odd. } Let $N=2m+1$. Since $\phi_{N}$ is an odd function, we write
\begin{align}
K^{\mathrm{edge}}_{N, 1}(x,x)=&K^{\mathrm{edge}}_{N,2}(x,x)+\frac{1}{2} N^{1/3} \phi_{N-1}(y) \int_{0}^{y} \phi_{N}(t) \dd t-\frac{1}{2} N^{1/3} \phi_{2m}(y) I_{2m+1}+\frac{1}{2 N^{1/6} I_{2m}} \phi_{2m}(y)\,,\nonumber
\end{align}
with $y$ given in (\ref{y_bound}). Using (\ref{upper_bound_hermite}), (\ref{int_even}), and (\ref{int_odd}), the last two terms above are bounded by $O(1)$. The second term can be estimated similarly as in the case $N=2m$. Thus (\ref{diff_even}) also hold true for $N=2m+1$.

We hence have finished the proof of Lemma \ref{lemma_traceimg_goe}.
\end{proof}

\appendix
\section*{Appendix}\label{Appendix A}

In this appendix we prove Lemma~\ref{lemma2} and Lemma~\ref{lemma1}. To prove Lemma~\ref{lemma2}, we follow the arguments in \cite{rigidity}.
\begin{proof}[Proof of Lemma \ref{lemma2}]
Recall the mollifier $\theta_\eta$ given in~\eqref{le mollifier} and the indicator function $\chi_E$ given in~\eqref{le indicator chi}, where $N^{-1} \ll \eta  \ll E_L-E  \leq C N^{-2/3+\epsilon}$, with $\epsilon>0$ as in~\eqref{ddd}. It suffices to estimate the linear eigenvalue statistics
$$\Tr \chi_E(H)- \Tr \chi_{E} \star \theta_{\eta}(H)=\Tr g(H)=\sum_{j=1}^N g(\lambda_j),$$
where 
\begin{equation}\label{appendix_1}
g(x):=\chi_E(x)-\chi_E\star \theta_\eta(x)=\Big(\int_{\R} \one_{[E,E_L]}(x)-\int_{E-x}^{E_L-x} \Big) \app(y) \dd y\,.
\end{equation}
 We first consider the function $g$. Note that for any $E>0$, we have
$$\frac{c \eta}{E+\eta} \leq \int_{E}^{\infty} \app(y) \dd y=\frac{1}{\pi} \int_{E}^{\infty} \frac{\eta}{y^2 +\eta^2} \dd y \leq \frac{C \eta}{E+\eta}.$$
Because of the symmetry of the integrand, we have a similar estimate for the integral over $(-\infty, E]$ with $E<0$.
Thus, if $x \in [E,E_L]$, we have from (\ref{appendix_1}) that
$$|g(x)| =\Big( \int^{E-x}_{-\infty} +\int^{\infty}_{E_L-x}\Big) \app(y) \dd y\leq C \eta \Big( \frac{1}{|x-E|+\eta} +\frac{1}{|x-E_L|+\eta} \Big).$$
Else, if $x \in [E,E_L]^c$, we have from the positiveness of $\app(y)$ that
\begin{equation}
|g(x)| =\int_{E-x}^{E_L-x} \app(y) \dd y \leq  \begin{cases}
		\frac{C \eta}{|x-E|+\eta}, & \mbox{if } x<E, \\
		\frac{C \eta}{|x-E_L|+\eta}, & \mbox{if } x > E_L\,,
		\end{cases}
\end{equation}
It is easy to check that 
\begin{equation}\label{appendix_2}
|g(x)| \leq 2C, \qquad \mbox{for } x \in \R.
\end{equation}
Now we choose a parameter $l_1$ such that $\eta \ll l_1 \ll E_L-E  \leq C N^{-2/3+\epsilon}$.  If we further assume $\min\{|x-E|, |x-E_L|\} \geq l_1$, then we have
\begin{equation}\label{appendix_3}
|g(x)| \leq \frac{2C\eta}{l_1}, \qquad \mbox{for } |x-E|>l_1,~|x-E_L|<l_1.
\end{equation}
Plugging (\ref{appendix_2}) and (\ref{appendix_3}) into (\ref{appendix_1}), we hence obtain
$$\Big|\Tr \chi_E(H)- \Tr \chi_{E} \star \theta_{\eta}(H)\Big| \leq C\Big( \mathcal{N}(E-l_1,E+l_1)+\mathcal{N}(E_L-l_1,\infty)+\frac{\eta}{l_1} \mathcal{N}(E,E_L)+\Tr f(H)\Big),$$
where 
$$f(x):=\big(\chi_E\star \theta_\eta\big)(x)\, \one_{x \leq E-l_1} \,.$$
Using the rigidity of eigenvalues in (\ref{rigidity1}), we obtain that
\begin{align}\label{append_temp3}
\Big|\Tr \chi_E(H)- \Tr \chi_{E} \star \theta_{\eta}(H)\Big| \leq C\Big( \mathcal{N}(E-l_1,E+l_1)+\frac{\eta}{l_1} N^{2\epsilon}+\Tr f(H)\Big),
\end{align}
with high probability, \ie with probability bigger than $1-N^{-\Gamma}$ for any large $\Gamma>0$, for $N$ sufficiently large. It is then sufficient to estimate $\Tr f(H)$. We write
\begin{align}\label{trace_f}
\Tr f(H) =\sum_{\lambda_i \leq E -l_1}f(\lambda_i) =\sum_{k=0}^{\infty} \sum_{\lambda_i \in \mathcal{I}_k} f(\lambda_i)\,, \qquad \mathcal{I}_k:=(E-3^{k+1}l_1, E-3^{k}l_1]\,.
\end{align}
If $x\leq E-l_1$, then $E_L-x \geq E-x \geq l_1 \gg \eta$, and we have
\begin{align}
f(x) =&\int_{E-x}^{E_L-x} \app(y) \dd y=\arctan\Big(\frac{E_L-x}{\eta} \Big)-\arctan\Big(\frac{E-x}{\eta} \Big)\nonumber\\
=&\arctan\Big(\frac{\eta}{E-x} \Big)-\arctan\Big(\frac{\eta}{E_L-x} \Big) \leq \frac{C \eta (E_L-E)}{(E_L-x)(E-x)}\nonumber\\
\leq & C \min\Big\{\frac{(E_L-E)\eta}{(E-x)^2}, \frac{ \eta}{E-x}\Big\}.\nonumber
\end{align}
In combination with (\ref{trace_f}), we have
\begin{align}\label{upper_f}
\Tr f(H) \leq C \sum_{k=0}^{\infty} \min\Big\{\frac{(E_L-E)\eta}{3^{2k}l_1^2}, \frac{ \eta}{3^k l_1}\Big\} \mathcal N_k\,, \qquad \mathcal{N}_k:=\#\{i: \lambda_i \in \mathcal{I}_k\}\,.
\end{align}
We next estimate $\mathcal{N}_k$ using the local law in (\ref{G}). Consider
\begin{align}\label{append_temp2}
\Im m_N(E-2 \cdot 3^{k} l_1+\ii 3^{k} l_1)=\frac{1}{N} \sum_{i=1}^N \frac{3^k l_1}{ |\lambda_i-(E-2 \cdot 3^k l_1)|^2+(3^k l_1)^2} \geq \frac{1}{N} \frac{\mathcal{N}_k}{2 \cdot 3^k l_1}\,.
\end{align}
Using the local law in (\ref{G}) and (\ref{22}), for any small $\tau>0$ and large $\Gamma>0$, we find an upper bound for the left hand side above as
\begin{align}
\Im m_N(E-2 \cdot 3^{k} l_1+\ii 3^{k} l_1) \leq & \Im m_{sc}(E-2 \cdot 3^{k} l_1+\ii 3^{k} l_1)+\frac{N^{\epsilon+\tau}}{ N 3^{k} l_1}\nonumber\\
\leq& C \sqrt{3^k l_1+|E-2 \cdot 3^{k} l_1-2|}+\frac{N^{\epsilon+\tau}}{ N 3^{k} l_1} \leq C \Big( \sqrt{3^k l_1}+\frac{N^{\epsilon+\tau}}{ N 3^{k} l_1}+N^{-1/3+\epsilon}\Big)\,,\nonumber
\end{align}
with probability bigger than $1-N^{-\Gamma}$.
By choosing $\tau<\epsilon$, we hence obtain from (\ref{append_temp2}) that
\begin{align}
\mathcal{N}_{k} \leq C\Big( (3^k l_1)^{3/2}N+N^{2\epsilon} +3^k l_1 N^{2/3+\epsilon}\Big)\,,\nonumber
\end{align}
with high probability. Combining with (\ref{upper_f}), we have
\begin{align}
\Tr f(H)  \leq& C \sum_{k=0}^{\infty}\min\Big\{\frac{(E_L-E)\eta}{3^{2k}l_1^2}, \frac{ \eta}{3^k l_1}\Big\} \Big( (3^k l_1)^{3/2}N+N^{2\epsilon} +3^k l_1 N^{2/3+\epsilon}\Big) \nonumber\\
\leq &\frac{C N^{1/3+\epsilon} \eta}{\sqrt{l_1}}+\frac{CN^{2\epsilon} \eta}{l_1} \leq \frac{C'N^{2\epsilon} \eta}{l_1}\,,\nonumber
\end{align}
with high probability. Together with (\ref{append_temp3}), we hence obtain
$$\Big|\Tr \chi_E(H)- \Tr \chi_{E} \star \theta_{\eta}(H)\Big| \leq C'\Big( \mathcal{N}(E-l_1,E+l_1)+\frac{\eta}{l_1} N^{2 \epsilon}\Big)\,,$$ 
with high probability. This completes the proof of Lemma~\ref{lemma2}.
\end{proof}

Next, we use Lemma \ref{lemma2} to prove Lemma \ref{lemma1}.
\begin{proof}[Proof of Lemma \ref{lemma1}]
Under the same assumption in Lemma \ref{lemma2}, we choose a parameter $l$ satisfying $N^{-1} \ll \eta \ll l_1 \ll l \ll E_L-E  \leq C N^{-2/3+\epsilon}$. We have from Lemma \ref{lemma1} that
\begin{align}\label{append_temp1}
\Tr \chi_E(H) \leq & l^{-1} \int_{E-l}^E \Tr \chi_{y}(H) \dd y\nonumber\\
 \leq &  l^{-1}\int_{E-l}^E  \Tr \chi_{y} \star \theta_{\eta}(H) \dd y+C l^{-1} \int_{E-l}^E \Big( \mathcal{N}(y-l_1,y+l_1)+\frac{\eta}{l_1} N^{2 \epsilon} \Big) \dd y\nonumber\\
 \leq  & \Tr \chi_{E-l} \star \theta_{\eta}(H) +C \Big(N^{2 \epsilon} \frac{ \eta}{l_1} +\frac{l_1}{l} \mathcal{N}(E-2l,E+l)\Big),
\end{align}
with high probability. Using the rigidity result (\ref{rigidity3}) and $l \ll N^{-2/3+\epsilon}$, we have
$$\mathcal{N}(E-2l,E+l) \leq \int_{E-2l}^{E+l} N \rho_{sc}(x) \dd x +N^{\epsilon} \leq CN^{\epsilon},$$
with high probability. Thus we obtain from (\ref{append_temp1}) that with high probability
$$\Tr \chi_E(H) -  \Tr \chi_{E-l} \star \theta_{\eta}(H)  \leq C N^{2 \epsilon} \Big( \frac{ \eta}{l_1} +\frac{l_1}{l}\Big).$$
One obtains a lower bound similarly. Therefore, for any large $\Gamma>0$, we have
$$\Tr \chi_{E+l} \star \theta_{\eta}(H)-C N^{2 \epsilon} \Big( \frac{ \eta}{l_1} +\frac{l_1}{l}\Big) \leq \Tr \chi_E(H)   \leq\Tr \chi_{E-l} \star \theta_{\eta}(H)+C N^{2 \epsilon} \Big( \frac{ \eta}{l_1} +\frac{l_1}{l}\Big),$$
with probability bigger than $1-N^{-\Gamma}$. We pick $l_1=N^{3 \epsilon} \eta$ and $l=N^{3 \epsilon} l_1$ such that $N^{2 \epsilon} \Big( \frac{ \eta}{l_1} +\frac{l_1}{l}  \Big)=N^{-\epsilon}.$ Since the counting function $\mathcal{N}(E,E_L)=\Tr \chi_E(H)$ is integer valued, we have
$$\P \Big( \mathcal{N}(E, E_L)=0\Big) \leq \P \Big( \Tr \chi_{E+l} \star \theta_{\eta}(H) \leq 1/9 \Big)+N^{-\Gamma} \leq \E \Big[ F\Big( \Tr \chi_{E+l} \star \theta_{\eta}(H) \Big)\Big]+N^{-\Gamma},$$
where $F$ is the cut-off function given in (\ref{q_function}). In the other direction, we have
$$\E \Big[  F \Big( \Tr \chi_{E-l} \star \theta_{\eta}(H) \Big)\Big] \leq \P \Big( \Tr \chi_{E-l} \star \theta_{\eta}(H) \leq 2/9 \Big) \leq \P \Big( \mathcal{N}(E, E_L) =0\Big)+N^{-\Gamma}.$$
Therefore, together with (\ref{rigidity1}), we obtain
\begin{equation*}
\E\Big[   F \Big( \Tr \chi_{E-l} \star \theta_{\eta}(H) \Big)\Big] -N^{-\Gamma}\leq \P \Big( \mathcal{N}(E, \infty) =0\Big) \leq   \E \Big[  F  \Big( \Tr \chi_{E+l} \star \theta_{\eta}(H) \Big)\Big]+N^{-\Gamma}.
\end{equation*}
This completes the proof of Lemma~\ref{lemma1}.
\end{proof}

\end{document}